\documentclass{article}
\usepackage[a4paper, total={6.5in, 8.5in}]{geometry}
\usepackage[utf8]{inputenc}
\usepackage{amssymb}
\usepackage{amsthm}
\usepackage[T1]{fontenc}
\usepackage[dvipsnames]{xcolor}
\usepackage{url}
\usepackage{booktabs}
\usepackage{amsfonts}
\usepackage{amsmath}
\usepackage{amssymb}
\usepackage{nicefrac}
\usepackage{bm}
\usepackage{pgfplots}
\usepgfplotslibrary{patchplots}
\pgfplotsset{compat=1.18}
\usepackage{tcolorbox}
\usepackage{mathtools}
\usepackage{amsthm}
\usepackage{authblk}
\usepackage{thmtools}
\usepackage[colorlinks = true, linkcolor=NavyBlue,citecolor=ForestGreen]{hyperref}
\usepackage[nameinlink]{cleveref}
\usepackage{bm}

\usepackage{multirow}

\usepackage{graphicx}
\usepackage{tikz}
\usetikzlibrary{arrows.meta}
\usepackage{xcolor}
\usetikzlibrary{calc}
\usetikzlibrary{decorations.markings}

\usepackage[linesnumbered,ruled,vlined]{algorithm2e}
\SetKwInput{KwInput}{Input}               
\SetKwInput{KwOutput}{Output}    
\SetKwFunction{KwFn}{Function}

\newtheorem{theorem}{Theorem}
\newtheorem{lemma}{Lemma}
\newtheorem{definition}{Definition}

\newtheorem{corollary}{Corollary}
\newtheorem{conjecture}{Conjecture}

\newtheorem{observation}{Observation}

\definecolor{colbg}{HTML}{23373b} 
\definecolor{colB}{HTML}{6699CC} 
\colorlet{col1}{red!60!white}
\colorlet{col3}{green!70!blue}
\colorlet{grey}{black!30!white}
\definecolor{cutRed}{RGB}{220, 25, 25}
\definecolor{arsCyan}{RGB}{0, 200, 230}   
\definecolor{connGreen}{RGB}{0, 180, 0}   
\definecolor{decompGray}{RGB}{150, 150, 150} 

\tikzstyle{vertex}=[circle, draw, fill=black, inner sep=0pt, minimum size=5pt]
\tikzstyle{cvertex}=[circle, thick, draw=colbg, fill=colB, inner sep=0pt, minimum size=7pt]
\tikzstyle{rvertex}=[circle, thick, draw=colbg, fill=col1, inner sep=0pt, minimum size=7pt]
\tikzstyle{ovvertex}=[circle, thick, draw=colbg, fill=col3, inner sep=0pt, minimum size=7pt]
\tikzstyle{gvertex}=[circle, draw, fill=grey, inner sep=0pt, minimum size=7pt]

\tikzstyle{edge}=[line width=1.5pt,black!50!white]
\tikzstyle{redge}=[line width=1.5pt,red]
\tikzstyle{baseEdge}=[line width=1.0pt]
\tikzstyle{ovedge}=[line width=1.5pt,col3]
\tikzstyle{cedge}=[line width=1.5pt, colB]
\tikzstyle{bedge}=[line width=1.5pt,blue]

\tikzset{
    midarrow/.style={
        postaction={decorate},
        decoration={markings, mark=at position 0.55 with {\arrow[very thick]{to}}}
    },
    rdedge/.style={draw=red, thick, midarrow},
    ovdedge/.style={draw=col3, thick, midarrow},
    bdedge/.style={draw=blue, thick, midarrow},
    cdedge/.style={draw=colB, thick, midarrow}
}

\tikzstyle{cutVertex}=[circle, draw=black, fill=black, inner sep=0pt, minimum size=7pt, text=white, font=\tiny\bfseries]
\tikzstyle{arsVertex}=[circle, draw=arsCyan, fill=arsCyan, inner sep=0pt, minimum size=7pt]

\tikzstyle{cutEdge}=[baseEdge, draw=cutRed]
\tikzstyle{arsEdge}=[baseEdge, draw=arsCyan]
\tikzstyle{mixEdge}=[baseEdge, draw=connGreen]
\tikzstyle{decompEdge}=[line width=0.8pt, draw=decompGray, dashed] 
\tikzstyle{parallelEdge}=[line width=0.8pt, draw=arsCyan]  %
\tikzstyle{rarsEdge}=[baseEdge, draw=red]
\tikzstyle{virtualEdge}=[line width=0.8pt, draw=arsCyan, dashed]
\tikzstyle{rvirtualEdge}=[line width=0.8pt, draw=red, dashed]  

\tikzstyle{blockNode}=[rectangle, draw=black, fill=arsCyan, rounded corners=3pt, minimum size=20pt, align=center, font=\small\bfseries]
\tikzstyle{treeCutNode}=[circle, draw=black, fill=black, inner sep=0pt, minimum size=12pt, text=white, font=\tiny\bfseries]
\tikzstyle{treeEdge}=[line width=1.0pt, black!80]
\tikzstyle{blockcutNode}=[rectangle, draw=black, fill=red!80, rounded corners=3pt, minimum size=24pt, align=center, font=\small\bfseries]
\tikzstyle{cutNode}=[circle, draw=black, fill=black, inner sep=0pt, minimum size=14pt, text=white, font=\footnotesize\bfseries]

\tikzstyle{treeEdge}=[line width=1.2pt, black!80]

        \newcommand{\drawKfiveARS}[3]{
            \begin{scope}[shift={(#1)}]
                \coordinate (center) at (#2:2.2);
                \def\rad{0.8}
                
                \node[arsVertex] (k3) at ($(center) + ({#2+180-40}:\rad)$) {};
                \node[arsVertex] (k4) at ($(center) + ({#2+180+40}:\rad)$) {};
                
                \node[arsVertex] (k1) at ($(center) + ({#2}:\rad)$) {};
                \node[arsVertex] (k2) at ($(center) + ({#2+72}:\rad)$) {};
                \node[arsVertex] (k5) at ($(center) + ({#2-72}:\rad)$) {};

                \draw[mixEdge] (#1) -- (k3); \draw[mixEdge] (#1) -- (k4);
                \draw[arsEdge] (k1)--(k2) (k1)--(k3) (k1)--(k4) (k1)--(k5);
                \draw[arsEdge] (k2)--(k3) (k2)--(k4) (k2)--(k5);
                \draw[arsEdge] (k3)--(k5); \draw[arsEdge] (k4)--(k5);
                
                \draw[decompEdge] (k3)--(k4);

                \node at ($(center) + (0, -1.5)$) {#3}; 
            \end{scope}
        }

        \newcommand{\drawKsixARS}[3]{
            \begin{scope}[shift={(#1)}]
                \coordinate (center) at (#2:2.5);
                \def\orad{1.0} \def\irad{0.35}
                \node[arsVertex] (k1) at ($(center) + ({#2+180-35}:\orad)$) {}; 
                \node[arsVertex] (k4) at ($(center) + ({#2+180+35}:\orad)$) {};
                \node[arsVertex] (k2) at ($(center) + ({#2+90}:\orad)$) {};
                \node[arsVertex] (k3) at ($(center) + ({#2-90}:\orad)$) {};
                \node[arsVertex] (k5) at ($(center) + ({#2}:\irad)$) {};
                \node[arsVertex] (k6) at ($(center) + ({#2+180}:\irad)$) {};

                \draw[mixEdge] (#1) -- (k1); \draw[mixEdge] (#1) -- (k4);
                \draw[arsEdge] (k1)--(k2)--(k3)--(k4);
                \draw[arsEdge] (k5)--(k1) (k5)--(k2) (k5)--(k3) (k5)--(k4);
                \draw[arsEdge] (k6)--(k1) (k6)--(k2) (k6)--(k3) (k6)--(k4);

                \draw[decompEdge] (k1)--(k4);

                \node at ($(center) + (1.8, 0)$) {#3};
            \end{scope}
        }

\newcommand{\localstep}[1]{
    \node[vertex] (#1-01) at (-0.8, -0.6) {};
    \node[vertex] (#1-02) at (-0.8, 0.6) {};
    \node[vertex] (#1-03) at (0, 0.9) {};
    \node[vertex] (#1-11) at (1, -0.9) {};
    \node[vertex] (#1-12) at (1, 0.9) {};
    \node[vertex] (#1-2)  at (1.7, 0) {};
}

\newcommand{\singlel}{
    \begin{tikzpicture}[scale=0.9, every node/.style={transform shape}]
    
    \begin{scope}[xshift=0cm, yshift=0cm]
        \localstep{s1}
        \node[cvertex] (v1) at (0,0) {}; \node[below right] at (v1) {$v_1$};
        \node[cvertex] (v2) at (1,0) {}; \node[below right] at (v2) {$v_2$};
        
        \draw[edge] (v1)--(s1-01) (v1)--(s1-02) (v1)--(v2);
        \draw[redge] (v1)--(s1-03); 
        \draw [-to,shorten >=-1pt,red, very thick] (0, 0.45) to (0, 0.5); 
        
        \draw[edge] (v2)--(s1-11) (v2)--(s1-12) (v2)--(s1-2);
        
        \draw[colB,very thick,->] (2, 0.2) -- (3, 1.8);
        \draw[colB,very thick,->] (2, 0)   -- (3, 0);
        \draw[colB,very thick,->] (2, -0.2)-- (3, -1.8);
    \end{scope}

    \begin{scope}[xshift=4.5cm, yshift=2.2cm]
        \localstep{s2a}
        \node[rvertex] (v1) at (0,0) {}; \node[red,below right] at (v1) {$v_1$};
        \node[cvertex] (v2) at (1,0) {};\node[below right] at (v2) {$v_2$};
        \draw[redge] (v1)--(s2a-01); \draw[-to,red,very thick] (-0.35, -0.25) to (-0.4, -0.3);
        \draw[redge] (v1)--(s2a-02); \draw[-to,red,very thick] (-0.35, 0.25) to (-0.3, 0.2);
        \draw[redge] (v1)--(s2a-03); \draw[-to,red,very thick] (0, 0.45) to (0, 0.5);
        \draw[redge] (v2)--(v1);     \draw[-to,red,very thick] (0.5, 0) to (0.45, 0);
        
        \draw[edge] (v2)--(s2a-11) (v2)--(s2a-12) (v2)--(s2a-2);
    \end{scope}

    \begin{scope}[xshift=4.5cm, yshift=0cm]
        \localstep{s2b}
        \node[rvertex] (v1) at (0,0) {}; \node[red, below right] at (v1) {$v_1$};
        \node[cvertex] (v2) at (1,0) {};\node[below right] at (v2) {$v_2$};
        \draw[redge] (s2b-01)--(v1); \draw[-to,red,very thick] (-0.35, -0.25) to (-0.3, -0.2);
        \draw[redge] (v1)--(s2b-02); \draw[-to,red,very thick] (-0.35, 0.25) to (-0.4, 0.3);
        \draw[redge] (v1)--(s2b-03); \draw[-to,red,very thick] (0, 0.45) to (0, 0.5);
        \draw[redge] (v2)--(v1);     \draw[-to,red,very thick] (0.5, 0) to (0.45, 0);
        
        \draw[edge] (v2)--(s2b-11) (v2)--(s2b-12) (v2)--(s2b-2);
    \end{scope}

    \begin{scope}[xshift=4.5cm, yshift=-2.2cm]
        \localstep{s2c}
        \node[rvertex] (v1) at (0,0) {}; \node[red, below right] at (v1) {$v_1$};
        \node[cvertex] (v2) at (1,0) {};\node[below right] at (v2) {$v_2$};
        \draw[redge] (s2c-01)--(v1); \draw[-to,red,very thick] (-0.35, -0.25) to (-0.3, -0.2);
        \draw[redge] (s2c-02)--(v1); \draw[-to,red,very thick] (-0.35, 0.25) to (-0.3, 0.2);
        \draw[redge] (v1)--(s2c-03); \draw[-to,red,very thick] (0, 0.45) to (0, 0.5);
        \draw[redge] (v1)--(v2);     \draw[-to,red,very thick] (0.5, 0) to (0.55, 0);
        
        \draw[edge] (v2)--(s2c-11) (v2)--(s2c-12) (v2)--(s2c-2);
    \end{scope}

    \end{tikzpicture}
}

\newcommand{\Gpath}{
    \begin{tikzpicture}[scale=.65]
        \begin{scope}[xshift=-1.5cm, yshift=0]
        
        \node[vertex] (01) at (-1.12, -0.8) {};
        \node[vertex] (02) at (-1.12, 0.8) {};
        \node[vertex] (03) at (0, 1) {};
        \node[cvertex] (0) at (0, 0) {}; \node at (-0.7, 0) {$u_0$};
        
        \node[vertex] (11) at (1, -1) {};
        \node[vertex] (12) at (1, 1) {};
        \node[cvertex] (1) at (1, 0) {}; \node at (1.4, -0.4) {$u_1$};

        \node[vertex] (21) at (3.12, -0.8) {};
        \node[vertex] (22) at (3.12, 0.8) {};
        \node[vertex] (23) at (2, 1) {};
        \node[cvertex] (2) at (2, 0) {}; \node at (2.8, 0) {$u_2$};

        \draw[edge] (0) to (01);    
        \draw[edge] (0) to (02);
        \draw[redge] (0) to (03);
        \draw[edge] (0) to (1);

        \draw [-to,shorten >=-1pt,red, very thick] (0, 0.58) to (0, 0.6);

        \draw[edge] (1) to (2);
        \draw[edge] (1) to (11);
        \draw[redge] (1) to (12);

        \draw [-to,shorten >=-1pt,red, very thick] (1, 0.58) to (1, 0.6);
        
        \draw[edge] (2) to (21);
        \draw[edge] (2) to (22);
        \draw[redge] (2) to (23);
        
        \draw [-to,shorten >=-1pt,red, very thick] (2, 0.58) to (2, 0.6);
       
     \end{scope}
        \draw[colB,very thick,->] (2.6, 0) -- (4.8, 4.5);
        \draw[colB,very thick,->] (2.6, 0) -- (4.8, 0);
        \draw[colB,very thick,->] (2.6, 0) -- (4.8, -6.5);

        \begin{scope}[xshift=6.5cm, yshift=4.5cm]
            
            \node[vertex] (01) at (-.8, -0.8) {};
            \node[vertex] (02) at (-.8, 0.8) {};
            \node[vertex] (03) at (0, 1) {};
            \node[rvertex] (0) at (0, 0) {}; \node at (-0.7, 0) {\textcolor{red}{$u_0$}};
        
            \node[vertex] (11) at (1, -1) {};
            \node[vertex] (12) at (1, 1) {};
            \node[cvertex] (1) at (1, 0) {}; \node at (1.4, -0.4) {$u_1$};

            \node[vertex] (21) at (3.12, -0.8) {};
            \node[vertex] (22) at (3.12, 0.8) {};
            \node[vertex] (23) at (2, 1) {};
            \node[cvertex] (2) at (2, 0) {}; \node at (2.8, 0) {$u_2$};

            \draw[redge] (0) to (01);    
            \draw [-to,shorten >=-1pt,red, very thick] (-.46, -.46) to (-.48, -.48);
            
            \draw[redge] (0) to (02);
            \draw [-to,shorten >=-1pt,red, very thick] (-.48, .48) to (-.46, .46);
            
            \draw[redge] (0) to (03);
            \draw [-to,shorten >=-1pt,red, very thick] (0, 0.58) to (0, 0.6);
            
            \draw[redge] (0) to (1);
            \draw [-to,shorten >=-1pt,red, very thick] (.55, 0) to (.53, 0);

            \draw[edge] (1) to (2);
            \draw[edge] (1) to (11);
            \draw[redge] (1) to (12);

            \draw [-to,shorten >=-1pt,red, very thick] (1, 0.58) to (1, 0.6);
        
            \draw[edge] (2) to (21);
            \draw[edge] (2) to (22);
            \draw[redge] (2) to (23);
        
            \draw [-to,shorten >=-1pt,red, very thick] (2, 0.58) to (2, 0.6);
        
            \draw[edge] (2) to (21);
            \draw[edge] (2) to (22);
            \draw[redge] (2) to (23);
        
            \draw [-to,shorten >=-1pt,red, very thick] (2, 0.58) to (2, 0.6);

            \draw[colB,very thick,->] (3.5, 0) -- (5, 0);
            
                \begin{scope}[xshift=6.5cm, yshift=0cm]

                    \node[vertex] (01) at (-.8, -0.8) {};
                    \node[vertex] (02) at (-.8, 0.8) {};
                    \node[vertex] (03) at (0, 1) {};
                    \node[rvertex] (0) at (0, 0) {}; \node at (-0.7, 0) {\textcolor{red}{$u_0$}};
            
                    \node[vertex] (11) at (1, -1) {};
                    \node[vertex] (12) at (1, 1) {};
                    \node[rvertex] (1) at (1, 0) {}; \node at (1.4, -0.4) {\textcolor{red}{$u_1$}};

                    \node[vertex] (21) at (3.12, -0.8) {};
                    \node[vertex] (22) at (3.12, 0.8) {};
                    \node[vertex] (23) at (2, 1) {};
                    \node[cvertex] (2) at (2, 0) {}; \node at (2.8, 0) {$u_2$};
    
                    \draw[redge] (0) to (01);    
                    \draw [-to,shorten >=-1pt,red, very thick] (-.46, -.46) to (-.48, -.48);
                
                    \draw[redge] (0) to (02);
                    \draw [-to,shorten >=-1pt,red, very thick] (-.48, .48) to (-.46, .46);
                
                    \draw[redge] (0) to (03);
                    \draw [-to,shorten >=-1pt,red, very thick] (0, 0.58) to (0, 0.6);
                
                    \draw[redge] (0) to (1);
                    \draw [-to,shorten >=-1pt,red, very thick] (.58, 0) to (.56, 0);

                    \draw[redge] (1) to (2);
                    \draw [-to,shorten >=-1pt,red, very thick] (1.55, 0) to (1.53, 0);

                    \draw[redge] (1) to (11);
                    \draw [-to,shorten >=-1pt,red, very thick] (1, -0.55) to (1, -0.53);
    
                    \draw[redge] (1) to (12);
                    \draw [-to,shorten >=-1pt,red, very thick] (1, 0.58) to (1, 0.6);

                    \draw[edge] (2) to (21);
                    \draw[edge] (2) to (22);
                    \draw[redge] (2) to (23);
            
                    \draw [-to,shorten >=-1pt,red, very thick] (2, 0.58) to (2, 0.6);
            
                    \draw[edge] (2) to (21);
                    \draw[edge] (2) to (22);
                    \draw[redge] (2) to (23);
                    \draw [-to,shorten >=-1pt,red, very thick] (2, 0.58) to (2, 0.6);

            \end{scope}

        \end{scope}

        \begin{scope}[xshift=6.5cm, yshift=0cm]
            
            \node[vertex] (01) at (-.8, -0.8) {};
            \node[vertex] (02) at (-.8, 0.8) {};
            \node[vertex] (03) at (0, 1) {};
            \node[rvertex] (0) at (0, 0) {}; \node at (-0.7, 0) {\textcolor{red}{$u_0$}};
        
            \node[vertex] (11) at (1, -1) {};
            \node[vertex] (12) at (1, 1) {};
            \node[cvertex] (1) at (1, 0) {}; \node at (1.4, -0.4) {$u_1$};

            \node[vertex] (21) at (3.12, -0.8) {};
            \node[vertex] (22) at (3.12, 0.8) {};
            \node[vertex] (23) at (2, 1) {};
            \node[cvertex] (2) at (2, 0) {}; \node at (2.8, 0) {$u_2$};

            \draw[redge] (0) to (01);    
            \draw [-to,shorten >=-1pt,red, very thick] (-.48, -.48) to (-.46, -.46);
            
            \draw[redge] (0) to (02);
            \draw [-to,shorten >=-1pt,red, very thick] (-.46, .46) to (-.48, .48);
            
            \draw[redge] (0) to (03);
            \draw [-to,shorten >=-1pt,red, very thick] (0, 0.58) to (0, 0.6);
            
            \draw[redge] (0) to (1);
            \draw [-to,shorten >=-1pt,red, very thick] (.55, 0) to (.53, 0);

            \draw[edge] (1) to (2);
            \draw[edge] (1) to (11);
            \draw[redge] (1) to (12);

            \draw [-to,shorten >=-1pt,red, very thick] (1, 0.58) to (1, 0.6);
        
            \draw[edge] (2) to (21);
            \draw[edge] (2) to (22);
            \draw[redge] (2) to (23);
        
            \draw [-to,shorten >=-1pt,red, very thick] (2, 0.58) to (2, 0.6);
        
            \draw[edge] (2) to (21);
            \draw[edge] (2) to (22);
            \draw[redge] (2) to (23);
        
            \draw [-to,shorten >=-1pt,red, very thick] (2, 0.58) to (2, 0.6);

            \draw[colB,very thick,->] (3.5, 0) -- (5, 0);

            \begin{scope}[xshift=6.5cm, yshift=0cm]
                
                \node[vertex] (01) at (-.8, -0.8) {};
                \node[vertex] (02) at (-.8, 0.8) {};
                \node[vertex] (03) at (0, 1) {};
                \node[rvertex] (0) at (0, 0) {}; \node at (-0.7, 0) {\textcolor{red}{$u_0$}};
        
                \node[vertex] (11) at (1, -1) {};
                \node[vertex] (12) at (1, 1) {};
                \node[rvertex] (1) at (1, 0) {}; \node at (1.4, -0.4) {\textcolor{red}{$u_1$}};

                \node[vertex] (21) at (3.12, -0.8) {};
                \node[vertex] (22) at (3.12, 0.8) {};
                \node[vertex] (23) at (2, 1) {};
                \node[cvertex] (2) at (2, 0) {}; \node at (2.8, 0) {$u_2$};

                \draw[redge] (0) to (01);    
                \draw [-to,shorten >=-1pt,red, very thick] (-.48, -.48) to (-.46, -.46);
            
                \draw[redge] (0) to (02);
                \draw [-to,shorten >=-1pt,red, very thick] (-.46, .46) to (-.48, .48);
            
                \draw[redge] (0) to (03);
                \draw [-to,shorten >=-1pt,red, very thick] (0, 0.58) to (0, 0.6);
            
                \draw[redge] (0) to (1);
                \draw [-to,shorten >=-1pt,red, very thick] (.58, 0) to (.56, 0);

                \draw[redge] (1) to (2);
                \draw [-to,shorten >=-1pt,red, very thick] (1.55, 0) to (1.53, 0);

                \draw[redge] (1) to (11);
                \draw [-to,shorten >=-1pt,red, very thick] (1, -0.55) to (1, -0.53);

                \draw[redge] (1) to (12);
                \draw [-to,shorten >=-1pt,red, very thick] (1, 0.58) to (1, 0.6);

                \draw[edge] (2) to (21);
                \draw[edge] (2) to (22);
                \draw[redge] (2) to (23);
        
                \draw [-to,shorten >=-1pt,red, very thick] (2, 0.58) to (2, 0.6);
        
                \draw[edge] (2) to (21);
                \draw[edge] (2) to (22);
                \draw[redge] (2) to (23);

                \draw [-to,shorten >=-1pt,red, very thick] (2, 0.58) to (2, 0.6);

            \end{scope}

        \end{scope}

        \begin{scope}[xshift=6.5cm, yshift=-6.5cm]
            
            \node[vertex] (01) at (-.8, -0.8) {};
            \node[vertex] (02) at (-.8, 0.8) {};
            \node[vertex] (03) at (0, 1) {};
            \node[rvertex] (0) at (0, 0) {}; \node at (-0.7, 0) {\textcolor{red}{$u_0$}};
        
            \node[vertex] (11) at (1, -1) {};
            \node[vertex] (12) at (1, 1) {};
            \node[cvertex] (1) at (1, 0) {}; \node at (1.4, -0.4) {$u_1$};

            \node[vertex] (21) at (3.12, -0.8) {};
            \node[vertex] (22) at (3.12, 0.8) {};
            \node[vertex] (23) at (2, 1) {};
            \node[cvertex] (2) at (2, 0) {}; \node at (2.8, 0) {$u_2$};

            \draw[redge] (0) to (01);    
            \draw [-to,shorten >=-1pt,red, very thick] (-.48, -.48) to (-.46, -.46);
            
            \draw[redge] (0) to (02);
            \draw [-to,shorten >=-1pt,red, very thick] (-.48, .48) to (-.46, .46);
            
            \draw[redge] (0) to (03);
            \draw [-to,shorten >=-1pt,red, very thick] (0, 0.58) to (0, 0.6);
            
            \draw[redge] (0) to (1);
            \draw [-to,shorten >=-1pt,red, very thick] (.58, 0) to (.6, 0);

            \draw[edge] (1) to (2);
            \draw[edge] (1) to (11);
            \draw[redge] (1) to (12);

            \draw [-to,shorten >=-1pt,red, very thick] (1, 0.58) to (1, 0.6);
        
            \draw[edge] (2) to (21);
            \draw[edge] (2) to (22);
            \draw[redge] (2) to (23);
        
            \draw [-to,shorten >=-1pt,red, very thick] (2, 0.58) to (2, 0.6);
        
            \draw[edge] (2) to (21);
            \draw[edge] (2) to (22);
            \draw[redge] (2) to (23);
        
            \draw [-to,shorten >=-1pt,red, very thick] (2, 0.58) to (2, 0.6);

            \draw[colB,very thick,->] (4, 0) -- (5, 3);
            \draw[colB,very thick,->] (4, 0) -- (5, -3);

            \begin{scope}[xshift=6.5cm, yshift=3cm]
            
                \node[vertex] (01) at (-.8, -0.8) {};
                \node[vertex] (02) at (-.8, 0.8) {};
                \node[vertex] (03) at (0, 1) {};
                \node[rvertex] (0) at (0, 0) {}; \node at (-0.7, 0) {\textcolor{red}{$u_0$}};
        
                \node[vertex] (11) at (1, -1) {};
                \node[vertex] (12) at (1, 1) {};
                \node[rvertex] (1) at (1, 0) {}; \node at (1.4, -0.4) {\textcolor{red}{\textcolor{red}{$u_1$}}};

                \node[vertex] (21) at (3.12, -0.8) {};
                \node[vertex] (22) at (3.12, 0.8) {};
                \node[vertex] (23) at (2, 1) {};
                \node[cvertex] (2) at (2, 0) {}; \node at (2.8, 0) {$u_2$};

                \draw[redge] (0) to (01);    
                \draw [-to,shorten >=-1pt,red, very thick] (-.48, -.48) to (-.46, -.46);
            
                \draw[redge] (0) to (02);
                \draw [-to,shorten >=-1pt,red, very thick] (-.48, .48) to (-.46, .46);
            
                \draw[redge] (0) to (03);
                \draw [-to,shorten >=-1pt,red, very thick] (0, 0.58) to (0, 0.6);
            
                \draw[redge] (0) to (1);
                \draw [-to,shorten >=-1pt,red, very thick] (.58, 0) to (.6, 0);

                \draw[redge] (1) to (2);
                \draw [-to,shorten >=-1pt,red, very thick] (1.55, 0) to (1.53, 0);

                \draw[redge] (1) to (11);
                \draw [-to,shorten >=-1pt,red, very thick] (1, -0.58) to (1, -0.6);
                
                \draw[redge] (1) to (12);
                \draw [-to,shorten >=-1pt,red, very thick] (1, 0.58) to (1, 0.6);
        
                \draw[edge] (2) to (21);
                \draw[edge] (2) to (22);
                \draw[redge] (2) to (23);
        
                \draw [-to,shorten >=-1pt,red, very thick] (2, 0.58) to (2, 0.6);
        
                \draw[edge] (2) to (21);
                \draw[edge] (2) to (22);
                \draw[redge] (2) to (23);
        
                \draw [-to,shorten >=-1pt,red, very thick] (2, 0.58) to (2, 0.6);

            \end{scope}

            \begin{scope}[xshift=6.5cm, yshift=-3cm]

                \node[vertex] (01) at (-.8, -0.8) {};
                \node[vertex] (02) at (-.8, 0.8) {};
                \node[vertex] (03) at (0, 1) {};
                \node[rvertex] (0) at (0, 0) {}; \node at (-0.7, 0) {\textcolor{red}{$u_0$}};
        
                \node[vertex] (11) at (1, -1) {};
                \node[vertex] (12) at (1, 1) {};
                \node[rvertex] (1) at (1, 0) {}; \node at (1.4, -0.4) {\textcolor{red}{$u_1$}};

                \node[vertex] (21) at (3.12, -0.8) {};
                \node[vertex] (22) at (3.12, 0.8) {};
                \node[vertex] (23) at (2, 1) {};
                \node[cvertex] (2) at (2, 0) {}; \node at (2.8, 0) {$u_2$};

                \draw[redge] (0) to (01);    
                \draw [-to,shorten >=-1pt,red, very thick] (-.48, -.48) to (-.46, -.46);
            
                \draw[redge] (0) to (02);
                \draw [-to,shorten >=-1pt,red, very thick] (-.48, .48) to (-.46, .46);
            
                \draw[redge] (0) to (03);
                \draw [-to,shorten >=-1pt,red, very thick] (0, 0.58) to (0, 0.6);
            
                \draw[redge] (0) to (1);
                \draw [-to,shorten >=-1pt,red, very thick] (.58, 0) to (.6, 0);

                \draw[redge] (1) to (2);
                \draw [-to,shorten >=-1pt,red, very thick] (1.58, 0) to (1.6, 0);

                \draw[redge] (1) to (11);
                \draw [-to,shorten >=-1pt,red, very thick] (1, -0.55) to (1, -0.53);

                \draw[redge] (1) to (12);
                \draw [-to,shorten >=-1pt,red, very thick] (1, 0.58) to (1, 0.6);

                \draw[edge] (2) to (21);
                \draw[edge] (2) to (22);
                \draw[redge] (2) to (23);
        
                \draw [-to,shorten >=-1pt,red, very thick] (2, 0.58) to (2, 0.6);
        
                \draw[edge] (2) to (21);
                \draw[edge] (2) to (22);
                \draw[redge] (2) to (23);
        
                \draw [-to,shorten >=-1pt,red, very thick] (2, 0.58) to (2, 0.6);               
            \end{scope}
        \end{scope}

    \end{tikzpicture}
}

\newcommand{\GDoubleEdge}{
    \begin{tikzpicture}[scale=.65]
        
        \begin{scope}[xshift=0cm, yshift=0cm]
            \node[cvertex] (u0) at (0, 0) {}; \node at (-0.7, 0) {$u_0$};
            \node[cvertex] (u1) at (1.5, 0) {}; \node at (1.6, -0.7) {$u_1$};
            \node[cvertex] (u2) at (3, 0) {}; \node at (3.1, -0.7) {$u_2$};

            \node[vertex] (n0a) at (-.8, 1) {};
            \node[vertex] (n0b) at (-.8, -1) {};
            
            \node[vertex] (n2a) at (3.8, 1) {};
            \node[vertex] (n2b) at (3.8, -1) {};

            \draw[redge] (n0a) to (u0);
            \draw [-to,shorten >=-1pt,red, very thick] (-0.4, 0.5) to (-0.35, 0.42);
            
            \draw[edge] (n0b) to (u0);

            \draw[edge] (u0) to[bend left=20] (u1);
            \draw[edge] (u0) to[bend right=20] (u1);

            \draw[edge] (u1) to[bend left=20] (u2);
            \draw[edge] (u1) to[bend right=20] (u2);

            \draw[edge] (u2) to (n2a);
            \draw[edge] (u2) to (n2b);

        \end{scope}

        \draw[colB,very thick,->] (4.2, 0.5) -- (5.5, 3.5);
        \draw[colB,very thick,->] (4.2, -0.5) -- (5.5, -3.5);
        \node at (4.5, 2.5) {\textbf{(2)}};
        \node at (4.5, -2.5) {\textbf{(1)}};
        \begin{scope}[xshift=6.6cm, yshift=3.5cm]

            \node[rvertex] (u0) at (0, 0) {}; \node at (-0.7, 0) {\textcolor{red}{$u_0$}};
            \node[cvertex] (u1) at (1.5, 0) {}; \node at (1.6, -0.7) {$u_1$};
            \node[cvertex] (u2) at (3, 0) {}; \node at (3.1, -0.7) {$u_2$};

            \node[vertex] (n0a) at (-.8, 1) {};
            \node[vertex] (n0b) at (-.8, -1) {};
            
            \node[vertex] (n2a) at (3.8, 1) {};
            \node[vertex] (n2b) at (3.8, -1) {};

            \draw[redge] (n0a) to (u0);
            \draw [-to,shorten >=-1pt,red, very thick] (-0.4, 0.5) to (-0.35, 0.42);

            \draw[redge] (u0) to (n0b);
            \draw [-to,shorten >=-1pt,red, very thick] (-0.35, -0.42) to (-0.4, -0.5);

            \draw[redge] (u0) to[bend left=20] (u1);
            \draw [-to,shorten >=-1pt,red, very thick] (0.65, 0.2) to (.8, 0.2); 

            \draw[redge] (u0) to[bend right=20] (u1);
            \draw [-to,shorten >=-1pt,red, very thick] (.8, -0.2) to (0.65, -0.2); 

            \draw[edge] (u1) to[bend left=20] (u2);
            \draw[edge] (u1) to[bend right=20] (u2);
            \draw[edge] (u2) to (n2a);
            \draw[edge] (u2) to (n2b);

            \draw[colB,very thick,->] (3.7, 0) -- (5, 0);
            \node at (4.4, .5) {\textbf{(2)}};

            \begin{scope}[xshift=6.3cm]
                \node[rvertex] (u0) at (0, 0) {}; \node at (-0.7, 0) {\textcolor{red}{$u_0$}};
                \node[rvertex] (u1) at (1.5, 0) {}; \node at (1.6, -0.7) {\textcolor{red}{$u_1$}};
                \node[cvertex] (u2) at (3, 0) {}; \node at (4.1, -0.7) {$u_2$};

                \node[vertex] (n0a) at (-.8, 1) {};
                \node[vertex] (n0b) at (-.8, -1) {};
                
                \node[vertex] (n2a) at (3.8, 1) {};
                \node[vertex] (n2b) at (3.8, -1) {};
                
                \draw[redge] (n0a) to (u0);
                \draw [-to,shorten >=-1pt,red, very thick] (-0.4, 0.5) to (-0.35, 0.42);

                \draw[redge] (u0) to (n0b);
                \draw [-to,shorten >=-1pt,red, very thick] (-0.35, -0.42) to (-0.4, -0.5);
    
                \draw[redge] (u0) to[bend left=20] (u1);
                \draw [-to,shorten >=-1pt,red, very thick] (0.65, 0.2) to (.8, 0.2); 

                \draw[redge] (u0) to[bend right=20] (u1);
                \draw [-to,shorten >=-1pt,red, very thick] (.8, -0.2) to (0.65, -0.2); 

                \draw[redge] (u1) to[bend left=20] (u2);
                \draw [-to,shorten >=-1pt,red, very thick] (2.15, 0.2) to (2.3, 0.2); 

                \draw[redge] (u1) to[bend right=20] (u2);
                \draw [-to,shorten >=-1pt,red, very thick] (2.3, -0.2) to (2.15, -0.2); 

                \draw[edge] (u2) to (n2a);
                \draw[edge] (u2) to (n2b);

            \end{scope}

        \end{scope}

        \begin{scope}[xshift=6.6cm, yshift=-3.5cm]
            \node[rvertex] (u0) at (0, 0) {}; \node at (-0.7, 0) {\textcolor{red}{$u_0$}};
            \node[cvertex] (u1) at (1.5, 0) {}; \node at (1.6, -0.7) {$u_1$};
            \node[cvertex] (u2) at (3, 0) {}; \node at (3.1, -0.7) {$u_2$};

            \node[vertex] (n0a) at (-.8, 1) {};
            \node[vertex] (n0b) at (-.8, -1) {};
            
            \node[vertex] (n2a) at (3.8, 1) {};
            \node[vertex] (n2b) at (3.8, -1) {};

            \draw[redge] (n0a) to (u0);
            \draw [-to,shorten >=-1pt,red, very thick] (-0.4, 0.5) to (-0.35, 0.42);

            \draw[redge] (u0) to (n0b);
            \draw [-to,shorten >=-1pt,red, very thick] (-0.4, -0.5) to (-0.35, -0.42);

            \draw[redge] (u0) to[bend left=20] (u1);
            \draw [-to,shorten >=-1pt,red, very thick] (0.65, 0.2) to (.8, 0.2); 

            \draw[redge] (u0) to[bend right=20] (u1);
            \draw [-to,shorten >=-1pt,red, very thick] (.65, -0.2) to (0.8, -0.2); 

            \draw[edge] (u1) to[bend left=20] (u2);
            \draw[edge] (u1) to[bend right=20] (u2);
            \draw[edge] (u2) to (n2a);
            \draw[edge] (u2) to (n2b);

            \draw[colB,very thick,->] (3.7, 0) -- (5, 0);
            \node at (4.4, .5) {\textbf{(1)}};

            \begin{scope}[xshift=6.3cm]
                \node[rvertex] (u0) at (0, 0) {}; \node at (-0.7, 0) {\textcolor{red}{$u_0$}};
                \node[rvertex] (u1) at (1.5, 0) {}; \node at (1.6, -0.7) {\textcolor{red}{$u_1$}};
                \node[cvertex] (u2) at (3, 0) {}; \node at (3.1, -0.7) {$u_2$};
    
                \node[vertex] (n0a) at (-.8, 1) {};
                \node[vertex] (n0b) at (-.8, -1) {};
                
                \node[vertex] (n2a) at (3.8, 1) {};
                \node[vertex] (n2b) at (3.8, -1) {};
    
                \draw[redge] (n0a) to (u0);
                \draw [-to,shorten >=-1pt,red, very thick] (-0.4, 0.5) to (-0.35, 0.42);

                \draw[redge] (u0) to (n0b);
                \draw [-to,shorten >=-1pt,red, very thick] (-0.4, -0.5) to (-0.35, -0.42);
    
                \draw[redge] (u0) to[bend left=20] (u1);
                \draw [-to,shorten >=-1pt,red, very thick] (0.65, 0.2) to (.8, 0.2); 

                \draw[redge] (u0) to[bend right=20] (u1);
                \draw [-to,shorten >=-1pt,red, very thick] (0.65, -0.2) to (.8, -0.2); 
    
                \draw[redge] (u1) to[bend left=20] (u2);
                \draw [-to,shorten >=-1pt,red, very thick] (2.15, 0.2) to (2.3, 0.2); 

                \draw[redge] (u1) to[bend right=20] (u2);
                \draw [-to,shorten >=-1pt,red, very thick] (2.15, -0.2) to (2.3, -0.2); 
                
                \draw[edge] (u2) to (n2a);
                \draw[edge] (u2) to (n2b);

            \end{scope}

        \end{scope}

    \end{tikzpicture}
}

\newcommand{\separablexanple}{
\begin{tikzpicture}[scale=0.52]

    \begin{scope}[xshift=-5cm]
        \node at (-2, -7) { \textbf{(a) The separable graph $G_{30}$}};

        \coordinate (u1c) at (90:1.5);
        \coordinate (u2c) at (-30:1.5);
        \coordinate (u3c) at (210:1.5);
        
        \node[arsVertex] (u1) at (u1c) {};
        \node[arsVertex] (u2) at (u2c) {};
        \node[arsVertex] (u3) at (u3c) {};
        \node at (0,0) {\large $\mathcal{D}_3$};

        \node[cutVertex] (v3) at (150:3.2) {};  \node[above=0.15cm]      at (v3) {$v_3$};
        \node[cutVertex] (v4) at (30:3.2)  {}; \node[below=0.15cm]      at (v4) {$v_4$};
        \node[cutVertex] (v5) at (270:3.2) {}; \node[right=0.15cm]      at (v5) {$v_5$};
        \node[cutVertex] (v1) at ($(v3) + (170:2.5)$) {}; \node[above=0.15cm]      at (v1) {$v_1$};
        \node[cutVertex] (v2) at ($(v3) + (230:2.5)$) {};\node[right=0.15cm]      at (v2) {$v_2$};

        \draw[arsEdge] (u1)--(u2)--(u3)--(u1);
        \draw[virtualEdge] (u1) to[bend left=20] (u2);
        \draw[virtualEdge] (u2) to[bend left=20] (u3);
        \draw[virtualEdge] (u3) to[bend left=20] (u1);

        \draw[cutEdge] (v1)--(v2)--(v3)--(v1);
        \draw[mixEdge] (v3)--(u1); \draw[mixEdge] (v4)--(u1);
        \draw[mixEdge] (v4)--(u2); \draw[mixEdge] (v5)--(u2);
        \draw[mixEdge] (v5)--(u3); \draw[mixEdge] (v3)--(u3);

        \drawKfiveARS{v1}{170}{$K_5^1$}
        \drawKfiveARS{v4}{50}{$K_5^2$} 
        \drawKsixARS{v2}{230}{$G_6^1$}
        \drawKsixARS{v5}{270}{$G_6^2$}
    \end{scope}

    \begin{scope}[xshift=6cm]
        \node at (1, -7) { \textbf{ (b) The Block-Cut Tree of $G_{30}$}};

        \node[blockNode] (D3) at (0, 0) {$\mathcal{D}_3$};

        \node[treeCutNode] (v3) at (130:2.5) {$v_3$};
        \draw[treeEdge] (D3) -- (v3);
        
        \node[blockcutNode] (C3) at ($(v3) + (130:2.0)$) {$C_3$};
        \draw[treeEdge] (v3) -- (C3);

        \node[treeCutNode] (v1) at ($(C3) -(90:2.0)$) {$v_1$};
        \draw[treeEdge] (C3) -- (v1);
        \node[blockNode] (K51) at ($(v1) - (90:1.5)$) {$K_5^1$};
        \draw[treeEdge] (v1) -- (K51);

        \node[treeCutNode] (v2) at ($(C3) - (180:2.0)$) {$v_2$};
        \draw[treeEdge] (C3) -- (v2);
        \node[blockNode] (G61) at ($(v2) -(180:1.5)$) {$G_6^1$};
        \draw[treeEdge] (v2) -- (G61);

        \node[treeCutNode] (v4) at (30:2.5) {$v_4$};
        \draw[treeEdge] (D3) -- (v4);
        \node[blockNode] (K25) at ($(v4) + (30:2.0)$) {$K_2^5$};
        \draw[treeEdge] (v4) -- (K25);

        \node[treeCutNode] (v5) at (270:2.5) {$v_5$};
        \draw[treeEdge] (D3) -- (v5);
        \node[blockNode] (G62) at ($(v5) + (270:2.0)$) {$G_6^2$};
        \draw[treeEdge] (v5) -- (G62);
    \end{scope}
    
    \draw[black!20, line width=2pt] (0.5, 5) -- (0.5, -6);

\end{tikzpicture}
}

\newcommand{\ARSDthree}{
\begin{tikzpicture}[scale=0.75]

    \coordinate (v1c) at (-1.2, 0);
    \coordinate (v2c) at (1.2, 0);
    \coordinate (v3c) at (0, 2); 

    \draw[rvirtualEdge] (v1c) to[bend left=15] (v3c);
    \draw[rvirtualEdge] (v1c) to[bend right=15] (v3c);
    \draw[rvirtualEdge] (v2c) to[bend left=15] (v3c);
    \draw[rvirtualEdge] (v2c) to[bend right=15] (v3c);

    \draw[rarsEdge] (v1c) to[bend left=15] (v2c);
    \draw[rarsEdge] (v1c) to[bend right=15] (v2c);

    \node[rvertex] (v1) at (v1c) {}; \node[below=0.15cm] at (v1) {$u_1$};
    \node[rvertex] (v2) at (v2c) {}; \node[below=0.15cm] at (v2) {$u_2$};
    \node[rvertex] (v3) at (v3c) {}; \node[above=0.15cm] at (v3) {$u_3$};
    
    \node[cutVertex] (c1) at (-3, 2.5) {};  \node[below left=0.05cm] at (c1) {$v_1$};
    \node[cutVertex] (c2) at (-3, -0.5) {}; \node[above left=0.05cm] at (c2) {$v_2$};
    \node[cutVertex] (c3) at (3, 2.5) {};   \node[below right=0.05cm] at (c3) {$v_3$};
    \node[cutVertex] (c4) at (3, -0.5) {};  \node[above right=0.05cm] at (c4) {$v_4$};

    \draw[mixEdge] (v1) -- (c1); \draw[mixEdge] (v1) -- (c2);
    \draw[mixEdge] (v2) -- (c3); \draw[mixEdge] (v2) -- (c4);
    \draw[mixEdge] (v3) -- (c1); \draw[mixEdge] (v3) -- (c2);
    \draw[mixEdge] (v3) -- (c3); \draw[mixEdge] (v3) -- (c4);

    \coordinate (x1) at ($(c1) + (160:1.5)$);
    \coordinate (y1) at ($(c1) + (90:1.5)$);
    \coordinate (z1) at ($(c1) + (125:2.8)$);
    
    \draw[mixEdge] (c1) -- (x1); 
    \draw[mixEdge] (c1) -- (y1);
    \draw[arsEdge] (x1) -- (y1);
    \draw[virtualEdge] (x1) to[bend left=25] (y1); 
    \draw[arsEdge] (x1) to[bend left=15] (z1); \draw[arsEdge] (x1) to[bend right=15] (z1);
    \draw[arsEdge] (y1) to[bend left=15] (z1); \draw[arsEdge] (y1) to[bend right=15] (z1);
    
    \node[arsVertex] at (x1) {}; \node[arsVertex] at (y1) {}; \node[arsVertex] at (z1) {};

    \coordinate (x2) at ($(c2) + (270:1.5)$);
    \coordinate (y2) at ($(c2) + (200:1.5)$);
    \coordinate (z2) at ($(c2) + (235:2.8)$);
    
    \draw[mixEdge] (c2) -- (x2); 
    \draw[mixEdge] (c2) -- (y2);
    \draw[arsEdge] (x2) -- (y2);
    \draw[virtualEdge] (x2) to[bend left=25] (y2); 
    \draw[arsEdge] (x2) to[bend left=15] (z2); \draw[arsEdge] (x2) to[bend right=15] (z2);
    \draw[arsEdge] (y2) to[bend left=15] (z2); \draw[arsEdge] (y2) to[bend right=15] (z2);
    
    \node[arsVertex] at (x2) {}; \node[arsVertex] at (y2) {}; \node[arsVertex] at (z2) {};

    \coordinate (x3) at ($(c3) + (20:1.5)$);
    \coordinate (y3) at ($(c3) + (90:1.5)$);
    \coordinate (z3) at ($(c3) + (55:2.8)$);
    
    \draw[mixEdge] (c3) -- (x3); 
    \draw[mixEdge] (c3) -- (y3);
    \draw[arsEdge] (x3) -- (y3);
    \draw[virtualEdge] (x3) to[bend right=25] (y3); 
    \draw[arsEdge] (x3) to[bend left=15] (z3); \draw[arsEdge] (x3) to[bend right=15] (z3);
    \draw[arsEdge] (y3) to[bend left=15] (z3); \draw[arsEdge] (y3) to[bend right=15] (z3);
    
    \node[arsVertex] at (x3) {}; \node[arsVertex] at (y3) {}; \node[arsVertex] at (z3) {};

    \coordinate (x4) at ($(c4) + (340:1.5)$);
    \coordinate (y4) at ($(c4) + (270:1.5)$);
    \coordinate (z4) at ($(c4) + (305:2.8)$);
    
    \draw[mixEdge] (c4) -- (x4); 
    \draw[mixEdge] (c4) -- (y4);
    \draw[arsEdge] (x4) -- (y4);
    \draw[virtualEdge] (x4) to[bend right=25] (y4); 
    \draw[arsEdge] (x4) to[bend left=15] (z4); \draw[arsEdge] (x4) to[bend right=15] (z4);
    \draw[arsEdge] (y4) to[bend left=15] (z4); \draw[arsEdge] (y4) to[bend right=15] (z4);
    
    \node[arsVertex] at (x4) {}; \node[arsVertex] at (y4) {}; \node[arsVertex] at (z4) {};

\end{tikzpicture}
}

\newcommand{\fourbond}{
    \node[vertex] (u) at (0, 0) {};
    \node[vertex] (v) at (3, 0) {}; 
    \draw[edge] (u) to[bend left=30] (v);
    \draw[edge] (u) to[bend left=10] (v);
    \draw[edge] (u) to[bend right=10] (v);
    \draw[edge] (u) to[bend right=30] (v);
    \path (1.5, 0.8) (1.5, -0.8); 
}

\newcommand{\dfour}{
    \node[vertex] (v1) at (0, 1.5) {};
    \node[vertex] (v2) at (1.5, 1.5) {};
    \node[vertex] (v3) at (1.5, 0) {};
    \node[vertex] (v4) at (0, 0) {};
    \foreach \a/\b in {v1/v2, v2/v3, v3/v4, v4/v1} {
        \draw[edge] (\a) to[bend left=15] (\b);
        \draw[edge] (\a) to[bend right=15] (\b);
    }
}

\newcommand{\drawARS}[2]{
        \begin{scope}[shift={(#1)}]
            \coordinate (center) at (#2:2.0);
            \def\rad{0.75}

            \node[cvertex] (k3) at ($(center) + ({#2+180-25}:\rad)$) {};
            \node[cvertex] (k4) at ($(center) + ({#2+180+25}:\rad)$) {};
            \node[cvertex] (k1) at ($(center) + ({#2}:\rad)$) {};
            \node[cvertex] (k2) at ($(center) + ({#2+72}:\rad)$) {};
            \node[cvertex] (k5) at ($(center) + ({#2-72}:\rad)$) {};

            \draw[ovedge] (#1) -- (k3);
            \draw[ovedge] (#1) --(k4);

            \draw[cedge] (k1)--(k2) (k1)--(k3) (k1)--(k4) (k1)--(k5);
            \draw[cedge] (k2)--(k3) (k2)--(k4) (k2)--(k5);
            \draw[cedge] (k3)--(k5);
            \draw[cedge] (k4)--(k5);
        \end{scope}
    }

\newcommand{\hthirtyfive}{
\begin{tikzpicture}[scale=0.85]
    
    \node[vertex] (3) at (0, 0) {$3$}; 
    \node[vertex] (1) at (-1.8, 1) {$1$};
    \node[vertex] (2) at (-1.8, -1) {$2$};

    \node[vertex] (5) at (2.5, 0) {$5$}; 
    \node[vertex] (4) at (1.25, 1.5) {$4$}; 

    \node[vertex] (6) at (4.3, 1) {$6$};
    \node[vertex] (7) at (4.3, -1) {$7$};

    \draw[redge] (1)--(2); 
    \draw[redge] (1)--(3);
    \draw[redge] (2)--(3);
    \draw[redge] (3)--(4);
    \draw[redge] (3)--(5);
    \draw[redge] (4)--(5);
    \draw[redge] (5)--(6);
    \draw[redge] (5)--(7);
    \draw[redge] (6)--(7);

    \drawARS{1}{135}

    \drawARS{2}{225}

    \drawARS{4}{90}

    \drawARS{6}{45}

    \drawARS{7}{-45}

\end{tikzpicture}

}

\newcommand{\SI}{
\begin{tikzpicture}
    \node[cvertex] (u1) at (-0.65, 0.5) {}; \node[above=0.15cm]      at (u1) {$u_1$};
    \node[cvertex] (u2) at (0.65, 0.5) {};\node[above=0.15cm]      at (u2) {$u_2$};
    \node[cvertex] (u3) at (-0.65, -0.5) {}; \node[below=0.15cm]      at (u3) {$u_3$};
    \node[cvertex] (u4) at (0.65, -0.5) {};\node[below=0.15cm]      at (u4) {$u_4$};
    
    \draw[redge] (u1) to (u2);
    \draw[redge] (u3) to (u4);
    \draw[gray!60,densely dashed] (0,0) circle (1.3cm);

    \draw[very thick,->] (0,-1.5) -- node[right=0.1cm] {$\mathcal{S}_1$} (0,-2.5);

    \begin{scope}[yshift=-4.3cm]
        \node[cvertex] (u1) at (-0.8, 0.5) {};\node[above=0.15cm]      at (u1) {$u_1$};
        \node[cvertex] (u2) at (0.8, 0.5) {};\node[above=0.15cm]      at (u2) {$u_2$};
        \node[cvertex] (u3) at (-0.8, -0.5) {}; \node[below=0.15cm]      at (u3) {$u_3$};
        \node[cvertex] (u4) at (0.8, -0.5) {}; \node[below=0.15cm]      at (u4) {$u_4$};
        \node[ovvertex] (v) at (0, 0) {}; \node[above=0.15cm]      at (v) {$v$};
        
        \draw[gray!60,densely dashed] (0,0) circle (1.65cm);
        \draw[bedge] (u1) to (v);
        \draw[bedge] (u2) to (v);
        \draw[bedge] (u3) to (v);
        \draw[bedge] (u4) to (v);
    \end{scope}
\end{tikzpicture}
}

\newcommand{\SII}{
\begin{tikzpicture}
    \node[gvertex] (v) at (0, 0) {}; \node[above=0.15cm]      at (v) {$v$};
    \node[cvertex] (u1) at (-0.6, 0.6) {}; \node[above=0.15cm]      at (u1) {$u_1$};
    \node[cvertex] (u2) at (0.6, 0.6) {}; \node[above=0.15cm]      at (u2) {$u_2$};
    \node[cvertex] (u3) at (-0.6, -0.6) {}; \node[below=0.15cm]      at (u3) {$u_3$};
    \node[cvertex] (u4) at (0.6, -0.6) {};  \node[below=0.15cm]      at (u4) {$u_4$};
    
    \draw[redge] (v) to (u1);
    \draw[redge] (v) to (u2);
    \draw[redge] (v) to (u3);
    \draw[redge] (v) to (u4);
    
    \draw[gray!60,densely dashed] (0,0) circle (1.3cm);

    \draw[very thick,->] (0,-1.5) -- node[right=0.1cm] {$\mathcal{S}_2$} (0,-2.5);

    \begin{scope}[yshift=-4.3cm]
        \node[cvertex] (u1) at (-0.85, 0.85) {}; \node[above=0.15cm]      at (u1) {$u_1$};
        \node[cvertex] (u2) at (0.85, 0.85) {}; \node[above=0.15cm]      at (u2) {$u_2$};
        \node[cvertex] (u3) at (-0.85, -0.85) {};  \node[below=0.15cm]      at (u3) {$u_3$};
        \node[cvertex] (u4) at (0.85, -0.85) {};  \node[below=0.15cm]      at (u4) {$u_4$};
        
        \node[ovvertex] (v1) at (-0.35, 0.35) {}; \node[above =0.1cm]      at (v1) {$v_1$};
        \node[ovvertex] (v2) at (0.35, 0.35) {}; \node[above =0.1cm]      at (v2) {$v_2$};
        \node[ovvertex] (v3) at (-0.35, -0.35) {}; \node[below =0.1cm]      at (v3) {$v_3$};
        \node[ovvertex] (v4) at (0.35, -0.35) {}; \node[below =0.1cm]      at (v4) {$v_4$};
        
        \draw[gray!60,densely dashed] (0,0) circle (1.65cm);
        
        \draw[ovedge] (v1) to (v2);
        \draw[ovedge] (v1) to (v3);
        \draw[ovedge] (v1) to (v4);
        \draw[ovedge] (v2) to (v3);
        \draw[ovedge] (v2) to (v4);
        \draw[ovedge] (v3) to (v4);
        
        \draw[bedge] (u1) to (v1);
        \draw[bedge] (u2) to (v2);
        \draw[bedge] (u3) to (v3);
        \draw[bedge] (u4) to (v4);
    \end{scope}
\end{tikzpicture}
}

\newcommand{\SIII}{
\begin{tikzpicture}
    \node[cvertex] (u1) at (-0.6, 0) {}; \node[above=0.15cm]      at (u1) {$u_1$};
    \node[cvertex] (u2) at (0.6, 0) {};  \node[above=0.15cm]      at (u2) {$u_2$};
    
    \draw[redge] (u1) to (u2);
    \draw[gray!60,densely dashed] (0,0) circle (1.3cm);

   \draw[very thick,->] (0,-1.5) -- node[right=0.1cm] {$\mathcal{S}_3$} (0,-2.5);

    \begin{scope}[yshift=-4.3cm]
        \node[cvertex] (u1) at (-1.3, 0) {}; \node[above=0.15cm]      at (u1) {$u_1$};
        \node[cvertex] (u2) at (1.3, 0) {}; \node[above=0.15cm]      at (u2) {$u_2$};
        
        \node[ovvertex] (v1) at (-0.6, 0) {}; \node[above=0.1cm]      at (v1) {$v_1$};
        \node[ovvertex] (v2) at (0.6, 0) {}; \node[above=0.1cm]      at (v2) {$v_2$};
        \node[ovvertex] (v3) at (0, 0.75) {}; \node[above=0.15cm]      at (v3) {$v_3$};
        \node[ovvertex] (v4) at (-0.4, -0.6) {}; \node[below=0.15cm]      at (v4) {$v_4$};
        \node[ovvertex] (v5) at (0.4, -0.6) {}; \node[below=0.15cm]      at (v5) {$v_5$};

        \draw[gray!60,densely dashed] (0,0) circle (1.65cm);

        \draw[ovedge] (v3) to (v1); \draw[ovedge] (v3) to (v2);
        \draw[ovedge] (v3) to (v4); \draw[ovedge] (v3) to (v5);
        \draw[ovedge] (v4) to (v1); \draw[ovedge] (v4) to (v2);
        \draw[ovedge] (v4) to (v5);
        \draw[ovedge] (v5) to (v1); \draw[ovedge] (v5) to (v2);

        \draw[bedge] (u1) to (v1);
        \draw[bedge] (u2) to (v2);
    \end{scope}
\end{tikzpicture}
}

\newcommand{\MIone}{
\begin{tikzpicture}[scale=0.95]
    \node[cvertex] (u1) at (0, 0.8) {}; \node[above=0.15cm]      at (u1) {$u_1$};
    \node[cvertex] (u2) at (-0.7, -0.3) {}; \node[left=0.1cm]      at (u2) {$u_2$};
    \node[cvertex] (u3) at (0.7, -0.3) {}; \node[right=0.1cm]      at (u3) {$u_3$};
    
    \draw[redge] (u1) to (u2);
    \draw[redge] (u1) to (u3);
    \draw[gray!60,densely dashed] (0,0) circle (1.3cm);

    \draw[very thick,->] (0,-1.5) -- node[right=0.1cm] {$\mathcal{M}_1$} (0,-2.5);

    \begin{scope}[yshift=-4.3cm]
        \node[cvertex] (u1) at (0, 1.1) {}; \node[above=0.15cm]      at (u1) {$u_1$};
        \node[cvertex] (u2) at (-1, -0.3) {}; \node[left=0.1cm]      at (u2) {$u_2$};
        \node[cvertex] (u3) at (1, -0.3) {}; \node[right=0.1cm]      at (u3) {$u_3$};
        \node[ovvertex] (v) at (0, 0) {}; \node[below=0.15cm]      at (v) {$v$};
        
        \draw[gray!60,densely dashed] (0,0) circle (1.65cm);
        
        \draw[bedge] (u2) to (v);
        \draw[bedge] (u3) to (v);
        
        \draw[bedge] (u1) to[bend left=25] (v);
        \draw[bedge] (u1) to[bend right=25] (v);
    \end{scope}
\end{tikzpicture}
}

\newcommand{\MItwo}{
\begin{tikzpicture}[scale=0.95]
    \node[cvertex] (u1) at (-0.8, 0.15) {}; \node[above=0.15cm]      at (u1) {$u_1$};
    \node[cvertex] (u2) at (0.8, 0.15) {}; \node[above=0.15cm]      at (u2) {$u_2$};

    \draw[redge] (u1) to[bend left=20] (u2);
    \draw[redge] (u1) to[bend right=20] (u2);

    \draw[gray!60,densely dashed] (0,0) circle (1.3cm);

    \draw[very thick,->] (0,-1.5) -- node[right=0.1cm] {$\mathcal{M}_1$} (0,-2.5);

    \begin{scope}[yshift=-4.3cm]
        \node[cvertex] (u1) at (-1, 0.45) {}; \node[above=0.15cm]      at (u1) {$u_1$};
        \node[cvertex] (u2) at (1, 0.45) {}; \node[above=0.15cm]      at (u2) {$u_2$};
        \node[ovvertex] (v) at (0, -0.7) {}; \node[below=0.15cm]      at (v) {$v$};
         
        \draw[gray!60,densely dashed] (0,0) circle (1.65cm);

        \draw[bedge] (u1) to[bend left=25] (v);
        \draw[bedge] (u1) to[bend right=25] (v);
        \draw[bedge] (u2) to[bend left=25] (v);
        \draw[bedge] (u2) to[bend right=25] (v);
        
    \end{scope}
\end{tikzpicture}
}

\newcommand{\MII}{
\begin{tikzpicture}[scale=0.95]
    \node[cvertex] (u1) at (-0.8, 0) {}; \node[above=0.15cm]      at (u1) {$u_1$};
    \node[cvertex] (u2) at (0.8, 0) {}; \node[above=0.15cm]      at (u2) {$u_2$};
    
    \draw[redge] (u1) to (u2);
    \draw[gray!60,densely dashed] (0,0) circle (1.3cm);

    \draw[very thick,->] (0,-1.5) -- node[right=0.1cm] {$\mathcal{M}_2$} (0,-2.5);

    \begin{scope}[yshift=-4.3cm]
        \node[cvertex] (u1) at (-1.3, 0) {}; \node[above=0.15cm]      at (u1) {$u_1$};
        \node[cvertex] (u2) at (1.3, 0) {}; \node[above=0.15cm]      at (u2) {$u_2$};
        
        \node[ovvertex] (v1) at (-0.65, 0) {}; \node[below=0.15cm]      at (v1) {$v_1$};
        \node[ovvertex] (v2) at (0.65, 0) {}; \node[below=0.15cm]      at (v2) {$v_2$};
        
        \draw[gray!60,densely dashed] (0,0) circle (1.65cm);
        
        \draw[bedge] (u1) to (v1);
        \draw[bedge] (u2) to (v2);
        
        \draw[ovedge] (v1) to (v2);                     
        \draw[ovedge] (v1) to[bend left=45] (v2);       
        \draw[ovedge] (v1) to[bend right=45] (v2);      
    \end{scope}
\end{tikzpicture}
}

\newcommand{\Kfourorient}{
\begin{tikzpicture}[scale=0.75]

    \begin{scope}[xshift=0cm, yshift=0cm]
        \node[cvertex] (w1) at (0, 1.5) {};       \node[above=0.15cm]      at (w1) {$v_1$};
        \node[cvertex] (w2) at (-1.3, -0.75) {}; \node[below=0.12cm]  at (w2) {$v_2$};
        \node[cvertex] (w3) at (1.3, -0.75) {};  \node[below=0.12cm] at (w3) {$v_3$};
        \node[cvertex] (w4) at (0, 0) {};        \node[below=0.12cm] at (w4) {$v_4$};
        
        \draw[rdedge] (w3) -- (w1);
        \draw[rdedge] (w4) -- (w1);
        \draw[rdedge] (w1) -- (w2);
        \draw[rdedge] (w4) -- (w2);
        \draw[rdedge] (w2) -- (w3);
        \draw[rdedge] (w3) -- (w4);
        
    \end{scope}

    \begin{scope}[xshift=4cm, yshift=0cm]
        \node[cvertex] (w1) at (0, 1.5) {};       \node[above=0.15cm]      at (w1) {$v_1$};
        \node[cvertex] (w2) at (-1.3, -0.75) {}; \node[below=0.12cm]  at (w2) {$v_2$};
        \node[cvertex] (w3) at (1.3, -0.75) {};  \node[below=0.12cm] at (w3) {$v_3$};
        \node[cvertex] (w4) at (0, 0) {};        \node[below=0.12cm] at (w4) {$v_4$};
        
        \draw[bdedge] (w3) -- (w1);
        \draw[bdedge] (w1) -- (w4);
        \draw[bdedge] (w2) -- (w1);
        \draw[bdedge] (w4) -- (w2);
        \draw[bdedge] (w3) -- (w2);
        \draw[bdedge] (w4) -- (w3);
        
    \end{scope}

    \begin{scope}[xshift=8cm, yshift=0cm]
        \node[cvertex] (w1) at (0, 1.5) {};       \node[above=0.15cm]      at (w1) {$v_1$};
        \node[cvertex] (w2) at (-1.3, -0.75) {}; \node[below=0.12cm]  at (w2) {$v_2$};
        \node[cvertex] (w3) at (1.3, -0.75) {};  \node[below=0.12cm] at (w3) {$v_3$};
        \node[cvertex] (w4) at (0, 0) {};        \node[below=0.12cm] at (w4) {$v_4$};
        
        \draw[ovdedge] (w3) -- (w1);
        \draw[ovdedge] (w4) -- (w1);
        \draw[ovdedge] (w1) -- (w2);
        \draw[ovdedge] (w2) -- (w4);
        \draw[ovdedge] (w3) -- (w2);
        \draw[ovdedge] (w4) -- (w3);

    \end{scope}

    \begin{scope}[xshift=12cm, yshift=0cm]
        \node[cvertex] (w1) at (0, 1.5) {};       \node[above=0.15cm]      at (w1) {$v_1$};
        \node[cvertex] (w2) at (-1.3, -0.75) {}; \node[below=0.12cm]  at (w2) {$v_2$};
        \node[cvertex] (w3) at (1.3, -0.75) {};  \node[below=0.12cm] at (w3) {$v_3$};
        \node[cvertex] (w4) at (0, 0) {};        \node[below=0.12cm] at (w4) {$v_4$};
        
        \draw[cdedge] (w1) -- (w3);
        \draw[cdedge] (w4) -- (w1);
        \draw[cdedge] (w2) -- (w1);
        \draw[cdedge] (w4) -- (w2);
        \draw[cdedge] (w3) -- (w2);
        \draw[cdedge] (w3) -- (w4);

    \end{scope}

\end{tikzpicture}
}

\newcommand{\Kfive}{
\begin{tikzpicture}[scale=0.75]
    \node[cvertex] (0) at (90:1.5) {};
    \node[cvertex] (1) at (18:1.5) {};
    \node[cvertex] (2) at (306:1.5) {};
    \node[cvertex] (3) at (234:1.5) {};
    \node[cvertex] (4) at (162:1.5) {};
    
    \draw[thick] 
        (0)--(1) (0)--(2) (0)--(3) (0)--(4) 
        (1)--(2) (1)--(3) (1)--(4) 
        (2)--(3) (2)--(4) 
        (3)--(4);
\end{tikzpicture}
}

\newcommand{\Hocta}{
\begin{tikzpicture}[scale=0.8]
    \node[cvertex]       (v1) at (0, 1.5) {};
    \node[cvertex] (v2) at (1.5, -0.8) {};
    \node[cvertex]  (v3) at (-1.5, -0.8) {};

    \node[cvertex]       (v4) at (0, 0.5) {};
    \node[cvertex]  (v5) at (0.6, -0.4) {};
    \node[cvertex] (v6) at (-0.6, -0.4) {};

    \draw[thick] 
        (v1)--(v2) (v2)--(v3) (v3)--(v1)
        (v4)--(v5) (v5)--(v6) (v6)--(v4)
        (v1)--(v4) (v2)--(v5) (v3)--(v6)
        (v1)--(v5) (v2)--(v6) (v3)--(v4);
\end{tikzpicture}
}

\newcommand{\Gsevenfirst}{
\begin{tikzpicture}[scale=0.8]
    \node[cvertex]           (0) at (90:1.5) {};
    \node[cvertex]   (1) at ({90-360/7}:1.5) {};
    \node[cvertex] (2) at ({90-2*360/7}:1.5) {};
    \node[cvertex] (3) at ({90-3*360/7}:1.5) {};
    \node[cvertex] (4) at ({90-4*360/7}:1.5) {};
    \node[cvertex] (5) at ({90-5*360/7}:1.5) {};
    \node[cvertex] (6) at ({90-6*360/7}:1.5) {};
    
    \draw[thick] 
        (0)--(3) (0)--(4) (0)--(5) (0)--(6) 
        (1)--(3) (1)--(4) (1)--(5) (1)--(6) 
        (2)--(3) (2)--(4) (2)--(5) (2)--(6) 
        (3)--(5) 
        (4)--(6);
\end{tikzpicture}
}

\newcommand{\Gsevensecond}{
\begin{tikzpicture}[scale=0.8]
    \node[cvertex]           (0) at (90:1.5) {};
    \node[cvertex]   (1) at ({90-360/7}:1.5) {};
    \node[cvertex] (2) at ({90-2*360/7}:1.5) {};
    \node[cvertex] (3) at ({90-3*360/7}:1.5) {};
    \node[cvertex] (4) at ({90-4*360/7}:1.5) {};
    \node[cvertex] (5) at ({90-5*360/7}:1.5) {};
    \node[cvertex] (6) at ({90-6*360/7}:1.5) {};
    
    \draw[thick] 
        (0)--(2) (0)--(3) (0)--(4) (0)--(5) 
        (1)--(3) (1)--(4) (1)--(5) (1)--(6) 
        (2)--(4) (2)--(5) (2)--(6) 
        (3)--(5) (3)--(6) 
        (4)--(6);
\end{tikzpicture}
}

\newcommand{\GeightA}{
\begin{tikzpicture}[scale=0.7]
    \node[cvertex]  (1) at (90:1.5) {};
    \node[cvertex]  (2) at (45:1.5) {};
    \node[cvertex]   (3) at (0:1.5) {};
    \node[cvertex] (4) at (315:1.5) {};
    \node[cvertex] (5) at (270:1.5) {};
    \node[cvertex] (6) at (225:1.5) {};
    \node[cvertex] (7) at (180:1.5) {};
    \node[cvertex] (8) at (135:1.5) {};
    
    \draw[thick] 
        (1)--(4) (1)--(5) (1)--(6) (1)--(8) 
        (2)--(4) (2)--(6) (2)--(7) (2)--(8) 
        (3)--(5) (3)--(6) (3)--(7) (3)--(8) 
        (4)--(5) (4)--(7)
        (5)--(7) 
        (6)--(8);
\end{tikzpicture}
}
\newcommand{\GeightB}{
\begin{tikzpicture}[scale=0.7]
    \node[cvertex]  (1) at (90:1.5) {};
    \node[cvertex]  (2) at (45:1.5) {};
    \node[cvertex]   (3) at (0:1.5) {};
    \node[cvertex] (4) at (315:1.5) {};
    \node[cvertex] (5) at (270:1.5) {};
    \node[cvertex] (6) at (225:1.5) {};
    \node[cvertex] (7) at (180:1.5) {};
    \node[cvertex] (8) at (135:1.5) {};
    
    \draw[thick] 
        (1)--(3) (1)--(5) (1)--(6) (1)--(7) 
        (2)--(4) (2)--(5) (2)--(6) (2)--(8) 
        (3)--(5) (3)--(7) (3)--(8) 
        (4)--(6) (4)--(7) (4)--(8) 
        (5)--(7) 
        (6)--(8);
\end{tikzpicture}
}
\newcommand{\GeightC}{
\begin{tikzpicture}[scale=0.73]
    \node[cvertex]  (1) at (90:1.5) {};
    \node[cvertex]  (2) at (45:1.5) {};
    \node[cvertex]   (3) at (0:1.5) {};
    \node[cvertex] (4) at (315:1.5) {};
    \node[cvertex] (5) at (270:1.5) {};
    \node[cvertex] (6) at (225:1.5) {};
    \node[cvertex] (7) at (180:1.5) {};
    \node[cvertex] (8) at (135:1.5) {};
    
    \draw[thick] 
        (1)--(3) (1)--(5) (1)--(6) (1)--(7) 
        (2)--(4) (2)--(5) (2)--(7) (2)--(8) 
        (3)--(5) (3)--(6) (3)--(8) 
        (4)--(6) (4)--(7) (4)--(8) 
        (5)--(7) 
        (6)--(8);
\end{tikzpicture}
}

\newcommand{\Gnine}{
\begin{tikzpicture}[scale=0.73]
    \node[cvertex]  (1) at (90:1.5) {};
    \node[cvertex]  (2) at (50:1.5) {};
    \node[cvertex]  (3) at (10:1.5) {};
    \node[cvertex] (4) at (330:1.5) {};
    \node[cvertex] (5) at (290:1.5) {};
    \node[cvertex] (6) at (250:1.5) {};
    \node[cvertex] (7) at (210:1.5) {};
    \node[cvertex] (8) at (170:1.5) {};
    \node[cvertex] (9) at (130:1.5) {};
    
    \draw[thick] 
        (1)--(3) (1)--(5) (1)--(7) (1)--(8) 
        (2)--(4) (2)--(6) (2)--(7) (2)--(9) 
        (3)--(5) (3)--(7) (3)--(8) 
        (4)--(6) (4)--(8) (4)--(9) 
        (5)--(7) (5)--(9) 
        (6)--(8) (6)--(9);
\end{tikzpicture}
}

\newcommand{\GtenA}{
\begin{tikzpicture}[scale=0.8]
    \node[cvertex]     (3)  at (-1.8, 0) {};
    \node[cvertex] (10) at (1.8, 0) {};

    \node[cvertex] (5) at (0, 0.8) {};
    \node[cvertex]        (7) at (-1.2, 1.4) {};
    \node[cvertex]       (2) at (1.2, 1.4) {};
    \node[cvertex]       (9) at (0, 2.2) {};

    \node[cvertex] (4) at (0, -0.8) {};
    \node[cvertex]        (6) at (-1.2, -1.4) {};
    \node[cvertex]       (1) at (1.2, -1.4) {};
    \node[cvertex]       (8) at (0, -2.2) {};

    \draw[thick]
        (2)--(5) (2)--(7) (2)--(9) 
        (5)--(7) (5)--(9) 
        (7)--(9)
        
        (1)--(4) (1)--(6) (1)--(8) 
        (4)--(6) (4)--(8) 
        (6)--(8)
        
        (3)--(6) (3)--(8) (3)--(7) (3)--(9)
        
        (10)--(1) (10)--(4) (10)--(2) (10)--(5);
\end{tikzpicture}
}

\newcommand{\GtenB}{
\begin{tikzpicture}[scale=0.8]
    \node[cvertex]     (9)  at (-1.4, 0) {};
    \node[cvertex] (10) at (1.4, 0) {};

    \node[cvertex]  (2) at (-0.6, 0.8) {};
    \node[cvertex] (3) at (0.6, 0.8) {};
    \node[cvertex]  (6) at (-1, 1.7) {};
    \node[cvertex] (8) at (1, 1.7) {};
    \node[cvertex] (7) at (0, 2.2) {};

    \node[cvertex]  (1) at (-0.4, -1) {};
    \node[cvertex] (4) at (0.4, -1) {};
    \node[cvertex] (5) at (0, -2.2) {};

    \draw[thick]
        (2)--(6) (2)--(7) (2)--(8) 
        (3)--(6) (3)--(7) (3)--(8) 
        (6)--(7) (6)--(8) 
        (7)--(8)
        
        (1)--(4) (1)--(5) 
        (4)--(5)
        
        (9)--(1) (9)--(4) (9)--(5) 
        (9)--(2)
        
        (10)--(1) (10)--(4) (10)--(5) 
        (10)--(3);
\end{tikzpicture}
}

\newcommand{\Geleven}{
\begin{tikzpicture}[scale=0.8]
    \node[cvertex]  (3) at (-1.8, 0) {};
    \node[cvertex] (8) at (1.8, 0) {};

    \node[cvertex] (5) at (0, 0.8) {};
    \node[cvertex]        (7) at (-1.2, 1.4) {};
    \node[cvertex]       (2) at (1.2, 1.4) {};
    \node[cvertex]       (9) at (0, 2.2) {};

    \node[cvertex] (1)  at (0, -0.8) {};
    \node[cvertex]     (10) at (-1.2, -1.4) {};
    \node[cvertex]    (11) at (1.2, -1.4) {};
    \node[cvertex]  (4)  at (-0.6, -2.2) {};
    \node[cvertex] (6)  at (0.6, -2.2) {};

    \draw[thick]
        (2)--(5) (2)--(7) (2)--(9) 
        (5)--(7) (5)--(9) 
        (7)--(9)
        
        (1)--(4) (1)--(6) (1)--(10) (1)--(11)
        (4)--(6) (4)--(10) (4)--(11)
        (6)--(10) (6)--(11)
        
        (3)--(7) (3)--(9) 
        (3)--(10)
        
        (8)--(2) (8)--(5) 
        (8)--(11)
        
        (3)--(8);
\end{tikzpicture}
}

\newcommand{\Gtwelve}{
\begin{tikzpicture}[scale=0.8]
    \node[cvertex]     (3)  at (-1.8, 0) {};
    \node[cvertex] (10) at (1.8, 0) {};

    \node[cvertex] (1) at (0, 1.4) {};
    \node[cvertex]        (5) at (-1.2, 1.7) {};
    \node[cvertex]       (7) at (1.2, 1.7) {};
    \node[cvertex]       (9) at (0, 2.2) {};

    \node[cvertex]  (2)  at (-1.0, -0.8) {};
    \node[cvertex] (4)  at (1.3, -0.8) {};
    \node[cvertex]        (6)  at (-1.4, -1.8) {};
    \node[cvertex]    (11) at (1.4, -1.8) {};
    \node[cvertex]       (8)  at (-0.6, -2.6) {};
    \node[cvertex]    (12) at (0.6, -2.6) {};

    \draw[thick]
        (1)--(5) (1)--(7) (1)--(9) 
        (5)--(7) (5)--(9) 
        (7)--(9)
        
        (2)--(6) (2)--(8) (2)--(11) (2)--(12)
        (4)--(8) (4)--(11) (4)--(12)
        (6)--(8) (6)--(11) (6)--(12)
        (8)--(12)
        
        (3)--(7) (3)--(9) 
        (3)--(11)
        
        (10)--(1) (10)--(5)
        (10)--(4)
        
        (3)--(10);
\end{tikzpicture}
}
\newcommand{\Gthirteen}{
\begin{tikzpicture}[scale=0.8]
    \node[cvertex]  (3) at (-1.8, 0) {};
    \node[cvertex] (6) at (1.8, 0) {};

    \node[cvertex]     (9)  at (-0.8, 0.8) {};
    \node[cvertex] (12) at (0.8, 0.8) {};
    \node[cvertex]           (1)  at (-1.4, 1.6) {};
    \node[cvertex]          (4)  at (1.4, 1.6) {};
    \node[cvertex]  (13) at (-0.6, 2.4) {};
    \node[cvertex]    (7)  at (0.6, 2.4) {};

    \node[cvertex]          (8)  at (0, -1.0) {};
    \node[cvertex]        (10) at (-1.4, -1.6) {};
    \node[cvertex]       (11) at (1.4, -1.6) {};
    \node[cvertex]     (2)  at (-0.6, -2.4) {};
    \node[cvertex]    (5)  at (0.6, -2.4) {};

    \draw[thick]
        (1)--(4) (1)--(7) (1)--(9) (1)--(13)
        (4)--(7) (4)--(12) (4)--(13)
        (7)--(12) (7)--(13)
        (9)--(13)
        
        (2)--(5) (2)--(8) (2)--(10) (2)--(11)
        (5)--(8) (5)--(10) (5)--(11)
        (8)--(10) (8)--(11)
        
        (3)--(9) (3)--(12) 
        (3)--(10)
        
        (6)--(9) (6)--(12) 
        (6)--(11)
        
        (3)--(6);
\end{tikzpicture}
}

\newcommand{\Gfourteen}{
\begin{tikzpicture}[scale=0.8]
    \node[cvertex]  (10) at (-2.0, -0.3) {};
    \node[cvertex] (13) at (2.0, -0.3) {};

    \node[cvertex]    (6)  at (-1.2, 0.5) {};
    \node[cvertex]   (2)  at (1.2, 0.5) {};
    \node[cvertex] (12) at (-1.2, 1) {};
    \node[cvertex](11) at (1.2, 1) {};
    \node[cvertex]   (1)  at (-1, 1.6) {};
    \node[cvertex]   (5)  at (0, 2.2) {};
    \node[cvertex]   (7)  at (1, 1.6) {};

    \node[cvertex]    (3)  at (-1.2, -0.8) {};
    \node[cvertex]   (4)  at (1.2, -0.8) {};
    \node[cvertex]   (8)  at (-0.8, -2.2) {};
    \node[cvertex]   (9)  at (0, -1) {};
    \node[cvertex]  (14) at (0.8, -2.2) {};

    \draw[thick]
        (2)--(6)
        (6)--(12) (2)--(11)
        (1)--(12) (1)--(11)
        (5)--(12) (5)--(11)
        (7)--(12) (7)--(11)
        (1)--(5) (1)--(7) (5)--(7)
        
        (3)--(8) (3)--(9) (3)--(14)
        (4)--(8) (4)--(9) (4)--(14)
        (8)--(9) (8)--(14) (9)--(14)
        
        (10)--(2) (10)--(6) (10)--(3)
        (13)--(2) (13)--(6) (13)--(4)
        
        (10)--(13);
\end{tikzpicture}
}

\newcommand{\Gfifteen}{
\begin{tikzpicture}[scale=0.8]
    \node[cvertex]    (8)  at (0, 1.2) {};
    \node[cvertex] (14) at (-1.8, -0) {};
    \node[cvertex]    (4)  at (1.8, -0.) {};

    \node[cvertex]    (1)  at (2.5, 2.2) {};
    \node[cvertex]   (5)  at (3.5, 2.2) {};
    \node[cvertex]    (9)  at (2.5, 1.5) {};
    \node[cvertex](12) at (3.5, 1.5) {};

    \node[cvertex](10) at (-2.5, 2.2) {};
    \node[cvertex] (13) at (-3.5, 2.2) {};
    \node[cvertex]   (6)  at (-2.5, 1.5) {};
    \node[cvertex]    (2)  at (-3.5, 1.5) {};

    \node[cvertex]    (3)  at (-0.5, -1.5) {};
    \node[cvertex]    (7)  at (-0.5, -2.1) {};
    \node[cvertex](11) at (0.5, -1.5) {};
    \node[cvertex](15) at (0.5, -2.1) {};

    \draw[thick]
        (1)--(5) (1)--(12) (1)--(9) 
        (5)--(12) (5)--(9) 
        (12)--(9)
        
        (13)--(10) (13)--(2) (13)--(6) 
        (10)--(2) (10)--(6) 
        (2)--(6)
        
        (3)--(11) (3)--(7) (3)--(15) 
        (11)--(7) (11)--(15) 
        (7)--(15)
        
        (8)--(1)  (8)--(5)
        (8)--(10) (8)--(13)
        
        (14)--(2) (14)--(6)
        (14)--(3) (14)--(7)

        (4)--(9)  (4)--(12)
        (4)--(11) (4)--(15);
\end{tikzpicture}
}

\newcommand{\Heleven}{
\begin{tikzpicture}[scale=0.8]
\begin{scope}[rotate=90]
    \node[cvertex] (10) at (0, 0) {};

    \node[cvertex] (1) at (-0.8, 1.0) {};
    \node[cvertex] (2) at (-0.8, -1.0) {};
    \node[cvertex]      (7) at (-1.4, 1.8) {};
    \node[cvertex]      (9) at (-1.4, -1.8) {};
    \node[cvertex]       (5) at (-2.2, 0) {};

    \node[cvertex] (3) at (0.8, 1.0) {};
    \node[cvertex] (4) at (0.8, -1.0) {};
    \node[cvertex]       (6)  at (1.4, 1.8) {};
    \node[cvertex]       (8)  at (1.4, -1.8) {};
    \node[cvertex]    (11) at (2.2, 0) {};

    \draw[thick]
        (1)--(5) (1)--(7) (1)--(9) 
        (2)--(5) (2)--(7) (2)--(9)
        (5)--(7) (5)--(9) 
        (7)--(9)
        
        (3)--(6) (3)--(8) (3)--(11)
        (4)--(6) (4)--(8) (4)--(11)
        (6)--(8) (6)--(11) 
        (8)--(11)
        
        (10)--(1) (10)--(2)
        
        (10)--(3) (10)--(4);
        \end{scope}
\end{tikzpicture}
}

\newcommand{\Htwelve}{
\begin{tikzpicture}[scale=0.8]
    \begin{scope}[rotate=90]
        \node[cvertex] (6) at (0, 0) {};
    
        \node[cvertex]   (9)  at (-0.8, 1.0) {};
        \node[cvertex](10) at (-0.8, -1.0) {};
        
        \node[cvertex] (1) at (-1.6, 1.8) {};
        \node[cvertex]  (7) at (-2.2, 0) {};
        \node[cvertex] (4) at (-1.6, -1.8) {};
    
        \node[cvertex]   (3)  at (0.8, 1.0) {};
        \node[cvertex](11) at (0.8, -1.0) {};
        
        \node[cvertex]    (8)  at (1.4, 1.8) {};
        \node[cvertex] (12) at (2.2, 0.8) {};
        \node[cvertex]    (5)  at (2.2, -0.8) {};
        \node[cvertex]    (2)  at (1.4, -1.8) {};
    
        \draw[thick]
            (1)--(4) (1)--(7) (1)--(9) (1)--(10)
            (4)--(7) (4)--(9) (4)--(10)
            (7)--(9) (7)--(10)
            
            (3)--(8) (3)--(11) (3)--(12)
            (11)--(2) (11)--(5)
            (8)--(12) (8)--(5) (8)--(2)
            (12)--(5) (12)--(2)
            (5)--(2)
            
            (6)--(9) (6)--(10)
            
            (6)--(3) (6)--(11);
        \end{scope}
\end{tikzpicture}
}

\newcommand{\Hthirteen}{
\begin{tikzpicture}[scale=0.8]
    \begin{scope}[rotate=90]
    \node[cvertex] (7) at (0, 0) {};

    \node[cvertex] (11) at (-0.8, 1.0) {};
    \node[cvertex]    (3)  at (-0.8, -1.0) {};
    
    \node[cvertex]    (1)  at (-1.4, 1.8) {};
    \node[cvertex]     (5)  at (-2.2, 0.8) {};
    \node[cvertex]     (8)  at (-2.2, -0.8) {};
    \node[cvertex] (10) at (-1.4, -1.8) {};

    \node[cvertex] (12) at (0.8, 1.0) {};
    \node[cvertex]    (4)  at (0.8, -1.0) {};
    
    \node[cvertex]    (2)  at (1.4, 1.8) {};
    \node[cvertex]    (6)  at (2.2, 0.8) {};
    \node[cvertex] (13) at (2.2, -0.8) {};
    \node[cvertex]    (9)  at (1.4, -1.8) {};

    \draw[thick]
        (11)--(1) (11)--(5) 
        (3)--(8) (3)--(10)
        (11)--(3) 
        (1)--(5) (1)--(8) (1)--(10)
        (5)--(8) (5)--(10)
        (8)--(10)
        
        (12)--(2) (12)--(6)
        (4)--(9) (4)--(13)
        (12)--(4) 
        (2)--(6) (2)--(9) (2)--(13)
        (6)--(9) (6)--(13)
        (9)--(13)
        
        (7)--(11) (7)--(3)
        
        (7)--(12) (7)--(4);
    \end{scope}
\end{tikzpicture}
}

\newcommand{\Hfourteen}{
\begin{tikzpicture}[scale=0.8]
    \begin{scope}[rotate=90]
            \node[cvertex] (12) at (0, 0) {};
        
            \node[cvertex] (3) at (-0.8, 1.0) {};
            \node[cvertex] (4) at (-0.8, -1.0) {};
            \node[cvertex]      (8) at (-1.4, 1.8) {};
            \node[cvertex]      (9) at (-1.4, -1.8) {};
            \node[cvertex]    (11) at (-2.2, 0) {};
        
            \node[cvertex] (2) at (0.8, 1.0) {};
            \node[cvertex] (6) at (0.8, -1.0) {};
            \node[cvertex]    (10) at (1.4, 1.5) {};
            \node[cvertex]    (13) at (1.4, -1.5) {};
            \node[cvertex]       (1) at (2.5, 1.8) {};
            \node[cvertex]    (14) at (1.6, 0.6) {};
            \node[cvertex]       (5) at (1.6, -0.6) {};
            \node[cvertex]       (7) at (2.5, -1.8) {};
        
            \draw[thick]
                (3)--(8) (3)--(9) (3)--(11)
                (4)--(8) (4)--(9) (4)--(11)
                (8)--(9) (8)--(11) (9)--(11)
                
                (2)--(6) (2)--(10) (2)--(13)
                (6)--(10) (6)--(13)
                (10)--(1) (10)--(14)
                (13)--(5) (13)--(7)
                (1)--(14) (1)--(5) (1)--(7)
                (14)--(5) (14)--(7)
                (5)--(7)
                
                (12)--(3) (12)--(4)
                (12)--(2) (12)--(6);
        \end{scope}
    \end{tikzpicture}
}

\newcommand{\Hfifteen}{
\begin{tikzpicture}[scale=0.8]
    \begin{scope}[rotate=90]
            \node[cvertex] (14) at (0, 0) {};
        
            \node[cvertex] (3)  at (0.8, 1.0) {};
            \node[cvertex] (4)  at (0.8, -1.0) {};
            \node[cvertex]    (11) at (1.4, 1.8) {};
            \node[cvertex]       (9)  at (1.4, -1.8) {};
            \node[cvertex]    (15) at (2.2, 0) {};
        
            \node[cvertex]  (6)  at (-0.5, 1.0) {};
            \node[cvertex]  (8)  at (-0.5, -1.0) {};
            \node[cvertex]       (2)  at (-1.2, 1.5) {};
            \node[cvertex]    (13) at (-1.2, -1.5) {};
            \node[cvertex]    (12) at (-1.8, 1.5) {};
            \node[cvertex]    (10) at (-1.8, -1.5) {};
            \node[cvertex]       (5)  at (-2.4, 1.8) {};
            \node[cvertex]       (7)  at (-2.4, -1.8) {};
            \node[cvertex]        (1)  at (-3.0, 0) {};
        
            \draw[thick]
                (3)--(9) (3)--(11) (3)--(15)
                (4)--(9) (4)--(11) (4)--(15)
                (9)--(11) (9)--(15) (11)--(15)
                
                (6)--(8) (6)--(2) (6)--(13)
                (8)--(2) (8)--(13)
                (2)--(13)
                (2)--(12) (13)--(10) 
                (12)--(1) (12)--(5) (12)--(7)
                (10)--(1) (10)--(5) (10)--(7)
                (1)--(5) (1)--(7) (5)--(7)
                
                (14)--(3) (14)--(4)
                (14)--(6) (14)--(8);
        \end{scope}
    \end{tikzpicture}
}

\title{Computing and Bounding the Number of Eulerian Orientations for Certain Classes of $4$-Regular Graphs}
\date{}

\author[1,2]{Evangelos Bartzos}
\author[3] {Michalis Samaris}

\affil[1]{“Athena” Research Center, Maroussi, Greece}
\affil[2]{Department of Informatics \& Telecommunications, National \& Kapodistrian University of Athens, Athens, Greece 
}
\affil[3]{Athens University of Economics and Business, Athens, Greece}

\begin{document}
\maketitle

\maketitle

	\begin{abstract}
		An Eulerian orientation of a $4$-regular undirected graph (simple or multigraph) $G=(V,E)$ with $n=|V|$ vertices is an assignment of directions to its edges such that every vertex $v \in V$ has the same indegree and outdegree. 
        In the present article, we improve the bounds on the number of Eulerian orientations for certain classes of connected, loopless $4$-regular graphs. 
        The previous bound is due to M. Las Vergnas (1983) and is exactly $9\cdot 2^{n-3}$, which is a sharp bound for a certain family of multigraphs with $n\geq 4$.
        Here, we show that the number of Eulerian orientations for all biconnected $4$-regular multigraphs is at most $2^n+2$, which is also sharp.
        We exhibit families of graphs that attain this maximum value. 
        For simple graphs, we prove an upper bound of $\mathcal{O}(3^{n/2})$ in the biconnected case and $\mathcal{O}(6^{n/3})$ for the separable case.
        Additionally, we provide a divide-and-conquer algorithm that leverages structural properties to compute the exact number of Eulerian orientations for separable graphs without exhaustive enumeration.
        Finally, we analyze the effect of standard inductive construction operations, used to generate $4$-regular graphs from smaller ones, as shown by F.Bories et.al. (1983) for simple graphs and by G.Ding et.al. (2003) for multigraphs, on the number of Eulerian orientations.
	\end{abstract}

\section{Introduction}
        An orientation of a graph, whether it is a \emph{simple graph} or a \emph{multigraph}, is an assignment of a direction to each of its edges.
        Graph orientations with outdegree (or equivalently, indegree) constraints naturally arise in network flow problems~\cite{Asahiro2022}, but they also find applications in a variety of graph-theoretic contexts such as shortest paths \cite{Shortest_Paths}, graph coloring~\cite{graph_coloring}, maximum matching~\cite{matching}, or in other mathematical areas such as algebraic bounds~\cite{betISSAC} and statistical physics~\cite{welsh1990}.

        An \emph{Eulerian orientation} is an orientation in which each vertex has equal indegree and outdegree.
        Clearly, a graph admits an Eulerian orientation if and only if every vertex has an even degree.
        A special class of graphs that admit Eulerian orientations consists of $2\kappa$-regular graphs; that is, graphs in which all vertices have degree $2\kappa$, for some integer $\kappa \geq 1$.

        A fundamental problem regarding prescribed graph orientations is the determination of tight bounds on the number of valid orientations based on certain graph properties.
        Felsner et al.\ \cite{Felsner} provide several bounds on the orientations of planar graphs, while other works \cite{bev,bet} utilize graph orientations to derive upper bounds on the embeddings of rigid graphs.
        In the specific case of Eulerian orientations of $4$-regular graphs, the current best known bound is derived in \cite{LasVergnas} by M. Las Vergnas and is $9\cdot 2^{n-3}$ for $n\geq 4$, where $n$ denotes the number of vertices.
        This bound improved a more general one by Schrijver \cite{Schrijver1983}, which is valid for $n\geq 2$ and is exactly $6^{n/2}$.
        Note that while both Las Vergnas and Schrijver provided results for all $2\kappa$-regular graphs, we restrict our focus here to the $\kappa=2$ case.
        In a follow-up article \cite{LasVergnas2}, Las Vergnas also provided upper bounds for general $2\kappa$-regular graphs parameterized by the number of vertices and the number of pairwise edge-disjoint cycles.

        Beyond bounding the number of Eulerian orientations, several algorithms (both deterministic and randomized) have been proposed to count the Eulerian orientations of general graphs that are not necessarily regular (see, for example, \cite{MihWin}).
        Finally, the broader literature contains algorithms and bounds for the closely related problem of counting Eulerian tours \cite{Stefa, Punzi}.

        \medskip\textit{Our Contribution.} 
        In this work, we improve the bounds on the number of Eulerian orientations for certain classes of connected, loopless $4$-regular graphs.
        While Las Vergnas' bound is sharp for a specific family of $4$-regular multigraphs (that are separable), we prove that for \emph{biconnected multigraphs} the upper bound is strictly tightened to $2^n+2$. 
        This improves upon the bound in \cite{LasVergnas} and is also sharp for a specific family of biconnected multigraphs.
        Let us remark that the bound presented in \cite{Schrijver1983} is attained for a multigraph with $2$ vertices and $4$ parallel edges, which is also the case for our bound (the \emph{$4$-bond multigraph}, see Figure~\ref{fig:multi_b}).
        It is clear that as $n$ grows, Schrijver's bound becomes substantially looser than both Las Vergnas's bound and ours.

        We then derive a bound for simple biconnected $4$-regular graphs, which is considerably tighter than the multigraph bound and behaves asymptotically as $\mathcal{O}\left(3^{n/2}\right)$.
        Following this, we treat the case of simple connected separable $4$-regular graphs. 
        Here, we establish a comparatively weaker bound that behaves asymptotically as $\mathcal{O}\left(6^{n/3}\right)$.

        The following three theorems constitute the main contribution for this part of the paper, providing upper bounds for biconnected multigraphs,  biconnected simple graphs, and separable simple graphs — all of which are $4$-regular.

        \begin{theorem}\label{th:multi_b}
            The number of Eulerian orientations of any biconnected loopless $4$-regular graph with $n$ vertices is bounded by
            \begin{equation*}
                2^n + 2. 
            \end{equation*} 
            This bound is asymptotically $\mathcal{O}(2^n)$ and is attained by a specific family of graphs.
        \end{theorem}

        \begin{theorem}\label{th:bi}
            The number of Eulerian orientations of a biconnected $4$-regular simple graph with $n$ vertices is bounded by
            \begin{equation*} 
                2\cdot 3^{n/2}, 
            \end{equation*} 
            which corresponds to an asymptotic growth of $\mathcal{O}(1.7321^n)$.
        \end{theorem}

        \begin{theorem}\label{th:sep}
            The number of Eulerian orientations of any separable, simple, $4$-regular graph is bounded by
            \begin{equation*}
                6^{(n+1)/3}, 
            \end{equation*} 
            which behaves asymptotically as $\mathcal{O} \left(1.8172^{n} \right)$.
        \end{theorem}

        Furthermore, we compare our theoretical bounds for simple graphs with the exact number of Eulerian orientations, which we determined via exhaustive computation for all simple $4$-regular graphs with $n \leq 15$ vertices.

        In the sequel, we provide closed formulas on the number of Eulerian orientations for separable graphs, subject to structural conditions and, where applicable, the number of constrained orientations of specific subgraphs.
        We show that this approach yields a divide-and-conquer algorithm that takes advantage of the structural properties of separable $4$-regular graphs avoiding exhaustive enumeration.
        Additionally, this approach allows us to exhibit a family of separable simple graphs that strictly possesses a greater number of orientations than biconnected simple graphs.

        Finally, for a specific set of inductive construction operations \cite{Bories1983, DING2003329}, we provide closed formulas detailing their effect on the total number of Eulerian orientations.

        \medskip\textit{Organization.} 
        The rest of the paper is organized as follows. 
        In Section~\ref{sec:pre}, we establish our notation and present techniques - similar to those introduced in \cite{bev, bet} - to derive bounds on outdegree-constrained graph orientations, adapting them for multigraphs.
        Then, we provide definitions and properties regarding the structure of separable graphs that will be utilized in the subsequent sections.
        In Section~\ref{sec:bound}, we adapt the techniques presented in the previous section to prove our main theorems.
        We also demonstrate that the derived bound for multigraphs is sharp for a specific family of graphs.
        Additionally, we present the maximum number of orientations obtained via exhaustive computation for all simple $4$-regular graphs up to $n=15$ vertices.
        In Section~\ref{sec:ARS_count}, we categorize the separable graphs based on the local structure of their cut vertices and provide a divide-and-conquer algorithm for computing the exact number of their Eulerian orientations.
        Using this method, we construct a class of simple separable graphs whose number of orientations strictly exceeds the upper bound established for biconnected simple graphs; based on this family and our computational results, we conjecture an even tighter upper bound. 
        In Section~\ref{sec:constr}, we prove that specific inductive construction operations increase the number of Eulerian orientations by a deterministic multiplicative factor.
        Finally, in Section~\ref{sec:concl}, we summarize our results and discuss potential directions for future research.

    \section{Notation and Preliminaries}\label{sec:pre}

        Throughout this paper, we consider \emph{loopless} and \emph{undirected} graphs that may contain \emph{parallel} edges.
        Two or more edges are considered parallel if they share the same vertices as endpoints. 
        An edge that is not parallel to any other is a \emph{simple edge}.
        A graph containing parallel edges is called a \emph{multigraph}; otherwise, it is \emph{simple}.

        In this context, given any undirected graph $G=(V,E)$ - whether it is a simple graph or a multigraph -  we treat the edge set $E$ as a \emph{multiset}. 
        The \emph{multiplicity} of an edge $e \in E$ is the number of times it appears in the multiset (i.e., the number of edges sharing the exact same endpoints as $e$).

        Regarding vertex connectivity, for any vertex $u \in V$, $d(u)$ denotes the degree of $u$ in $G$. 
        Furthermore, for any subgraph $G^\prime \subseteq G$, $d_{G^\prime}(u)$ denotes the degree of $u$ strictly within the subgraph $G^\prime$.
        
        Given a subset of vertices $V^\prime \subseteq V$, we denote by $G\setminus V^{\prime}$ the graph obtained by removing the vertices of  $V^{\prime}$ and every edge that has some vertex of $V^{\prime}$ as an endpoint.
        If $V^{\prime}=\{v\}$ for a vertex $v\in V$, we denote this graph by $G\setminus v$ instead of $G\setminus \{v\}$.
        By $G[V^{\prime}]$, we denote the \emph{induced subgraph} $G\setminus (V\setminus V^{\prime})$ i.e., the graph where every vertex not belonging to $V^{\prime}$ is removed.
        A vertex $v \in V$ is a \emph{cut vertex} if its removal increases the number of connected components of $G$.
        Similarly, an edge is a \emph{bridge} if its removal increases the number of connected components of $G$.
        A connected graph $G$ is \emph{biconnected} if it has no cut vertices; otherwise $G$ is called \emph{separable}.
        A \emph{biconnected component}  of a separable graph $G$ is a maximal subgraph $G_b$ of $G$ such that $G_b$ is biconnected.
        
        An \emph{orientation} of $G$ is a directed graph $\vec{G}=(V,\vec{E})$ obtained by assigning a unique direction to each edge $e \in E$.
        An orientation is called an \emph{Eulerian orientation} if, for every vertex $v \in V$, the indegree equals the outdegree.  
        It is well-known that an undirected graph admits an Eulerian orientation if and only if every vertex has even degree; such graphs are called \emph{Eulerian graphs}.
        We denote by $\mathcal{E}(G)$ the set of Eulerian orientations of $G$. Consequently, its cardinality $|\mathcal{E}(G)|$ represents the number of Eulerian orientations of $G$.
        
        The \emph{Handshaking Lemma} states that for any loopless graph, the sum of degrees of all vertices is $2|E|$, which implies that in every graph the number of odd-degree vertices is even.
        A direct corollary of this is that every Eulerian graph is bridgeless, since the removal of a bridge would create connected components with an odd number of vertices of odd degree.
        
        In what follows in this section, we first present an elimination procedure analogous to the one introduced in \cite{bev,bet}.
        This procedure was originally developed to bound the number of outdegree-constrained orientations of a certain class of simple undirected graphs.
        Here, instead of introducing new graph structures (as was done in \cite{bev}), we use \emph{mixed graphs} to align with standard graph-theoretic terminology.
        We provide several definitions related to this procedure that will be used throughout the remainder of this paper.
        We also detail the modifications needed to include the orientation of parallel edges.
        
        Finally, we discuss properties regarding the orientations of specific subgraphs within a $4$-regular graph, recall the definition of the block-cut tree as presented in \cite{Harary69}, and analyze the constrained orientations of edges incident to cut vertices.

    \subsection{Elimination process}\label{sec:elim}
        Our goal is to bound the number of Eulerian orientations for all $4$-regular graphs with $n$ vertices.
        To this end, we apply an iterative process called the \emph{elimination process}.
        At each step of this process, we select a specific vertex (or a path of vertices) and consider all ways to orient its incident undirected edges to satisfy the prescribed outdegree constraints, thereby \emph{eliminated} the corresponding undirected edges and vertices.
        By multiplying the number of valid local orientations at each step, we can build a worst-case upper bound for the total number of Eulerian orientations.
        The process naturally terminates once the remaining undirected edges satisfy a specific stopping condition, as established in \cite{bev}.
        
        Initially, all edges are undirected.
        However, the first elimination step introduces directed edges, meaning all subsequent steps must operate on a hybrid structure. 
        We therefore recall a standard definition for this type of graph structure:

        \begin{definition}
            A \emph{mixed graph} $J=(V, E, D)$ consists of a set of vertices $V$, a multiset of undirected edges $E$, and a multiset of directed edges $D$.
        \end{definition}

        Let $V_J$ be the set of vertices in a mixed graph $J$ whose incident edges are not yet fully oriented. 
        We call the graph $G_J=(V_J,E)$ the \emph{undirected subgraph} of $J$.
        For a vertex $v \in V_J$, with degree $d(v)$ in $G_J$, we denote by $p(v)$ the number of its incident edges already directed outwards in $J$.

        The ordered pair $(d(v),p(v))$ is the \emph{profile} of $v$.  
        Then, the number of ways to orient the undirected incident edges of $v$ is denoted by $C(d(v),p(v))$ and is defined as the \emph{cost} of this elimination step for this vertex.
        This is given exactly by:

        \begin{equation}\label{eq:cost}
            C(d(v),p(v))=\binom{d(v)}{2-p(v)}.
        \end{equation}

        \begin{figure}[!htbp]
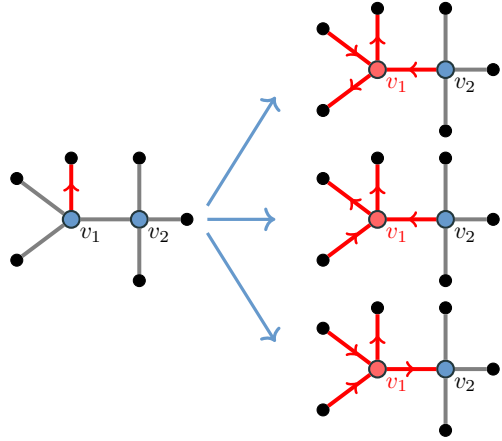

            \centering
            \singlel
            \caption{
            Elimination of vertex $v_1$ with profile $(3,1)$. 
            The cost of this step is $C(3,1)= \binom{3}{2-1} = 3$. 
            If $v_2$ is subsequently eliminated, the cumulative cost across both steps becomes $3\cdot 3=9$, since $v_2$ will have a profile of either $(3,0)$ or $(3,1)$, and in both cases the cost is $C(3,0)= C(3,1)= C^*(3)=3$.
            }\label{fig:singlel}
        \end{figure}

        Notice that in each step, a set of new mixed graphs is generated.
        In each resulting undirected subgraph $G_{J^\prime}$, there is exactly one fewer vertex and $d(v)$ fewer undirected edges. 
        Furthermore, all neighbors of $v$ may have different profiles in these resulting mixed graphs.
        To ensure the final orientation can be Eulerian, every eliminated vertex must attain an outdegree of exactly $2$; intermediate graphs satisfying this condition are termed \emph{valid mixed graphs}. 
        When the profile $(d(v), p(v))$ is $(1, 1)$ or $(2, 0)$, or when $p(v) = 2$, Equation~\eqref{eq:cost} yields a cost of $1$.
        In the elimination process, these are termed \emph{trivial vertices}.

        To determine upper bounds on the number of orientations, the \emph{maximizing cost} was utilized in \cite{bet,bev}, defined as 
        \[ C^*(d(v))= \max_{q\in \{0,1,2\}}  C(d(v), q). \]
        We observe that the total cost of iterative elimination steps is inherently multiplicative, in the sense that the elimination process can be seen as a branching process. 
        Eliminating a vertex $v$ branches a valid mixed graph $J$ into exactly $C(d(v), p(v))$ new valid mixed graphs. 
        Because this specific branching factor depends on previously oriented edges, we apply the maximizing cost $C^*(d(v))$ as a strict upper bound  to every branch.
        For example, eliminating $v_1$ yields at most $C^*(d(v_1))$ valid branches. 
        In the subsequent step, eliminating $v_2$ splits each of those resulting branches into at most $C^*(d(v_2))$ new graphs. 
        Therefore, the total number of valid mixed graphs generated after both steps is strictly bounded from above by the product $C^*(d(v_1)) \cdot C^*(d(v_2))$. 
        By applying this branching logic across the entire sequence, the global upper bound for the elimination process is obtained simply by multiplying the maximum costs of all individual steps (see Figure~\ref{fig:singlel}).

        As noted in \cite{bev}, certain profiles can significantly inflate the orientation bounds in the worst-case scenario due to the multiplicative nature of the costs.
        In the case of orientations with a prescribed outdegree of $2$, as in the present paper, this inflation is specifically driven by vertices possessing the profile $(2,1)$. 
        However, as shown in \cite{bet}, these degree-$2$ vertices are generated during the elimination process alongside vertices with lower costs in the other branches.
        Thus, one can consider the total cost for a \emph{path of vertices} of a uniform degree rather than the maximizing cost for a single vertex.

        Following the approach in \cite{bet}, we consider paths $\mathcal{P}=\{u_0, u_1, \dots, u_\ell\}$ where the \emph{leading vertex} $u_0$ is eliminated with a degree other than $2$, while each of the subsequent $\ell$ vertices has degree $2$ at the time of its elimination.
        Let $\mathcal{C}(\mathcal{P}_{u_i})$ denote the accumulated cost after the elimination of the vertices $\{u_0, u_1, \ldots,u_i\}$.
        Consequently, the \emph{total cost} of eliminating all the vertices of the path $\mathcal{P}=\{u_0, u_1, \dots, u_\ell\}$ is $\mathcal{C}(\mathcal{P}_{u_{\ell}})$.

        By factoring out the cost of the leading vertex and using the multiplicative property, the \emph{average cost} per degree-$2$ vertex is given by: 
        \begin{equation}\label{eq:average}
            \displaystyle \left(\frac{\mathcal{C}(\mathcal{P}_{u_{\ell}})}{C(d(u_0),p(u_0))} \right)^{1/\ell}.
        \end{equation}

        We track the branching of the mixed graphs at step $i$ by defining two quantities: let $\mathcal{B}(u_i)$ be the number of valid mixed graphs generated where $u_i$ has profile $(2,1)$, and let $\Gamma(u_i)$ be the number of valid mixed graphs where $u_i$ has profile $(2,0)$ or $(2,2)$ (i.e., $u_i$ is a trivial vertex). 
        Since a $(2,1)$ profile yields a cost of $2$ and a trivial profile yields a cost of $1$, the accumulated cost at this step is exactly:
        $$\mathcal{C}(\mathcal{P}_{u_i})= 2\cdot \mathcal{B}(u_i)+\Gamma(u_i).  $$
        Thus, we need recursive formulas for these branching factors as in \cite{bet} to compute the worst-case scenario for the average cost of such paths.

        To do so, we must distinguish between the structural cases that cause a vertex to have degree $2$ during the elimination process. 
        A vertex $u_i$ in such a path has degree $2$ in the undirected subgraph $G_J$ if one of the following conditions is met prior to the elimination of $u_{i-1}$:
        \begin{itemize}
            \item[a.] $u_i$ has degree $4$ and there are two parallel edges between $u_{i-1}$ and $u_i$.
            \item[b.] $u_i$ has degree $3$ and there is a simple edge between $u_{i-1}$ and $u_i$.
        \end{itemize}

        \begin{figure}[tbh]
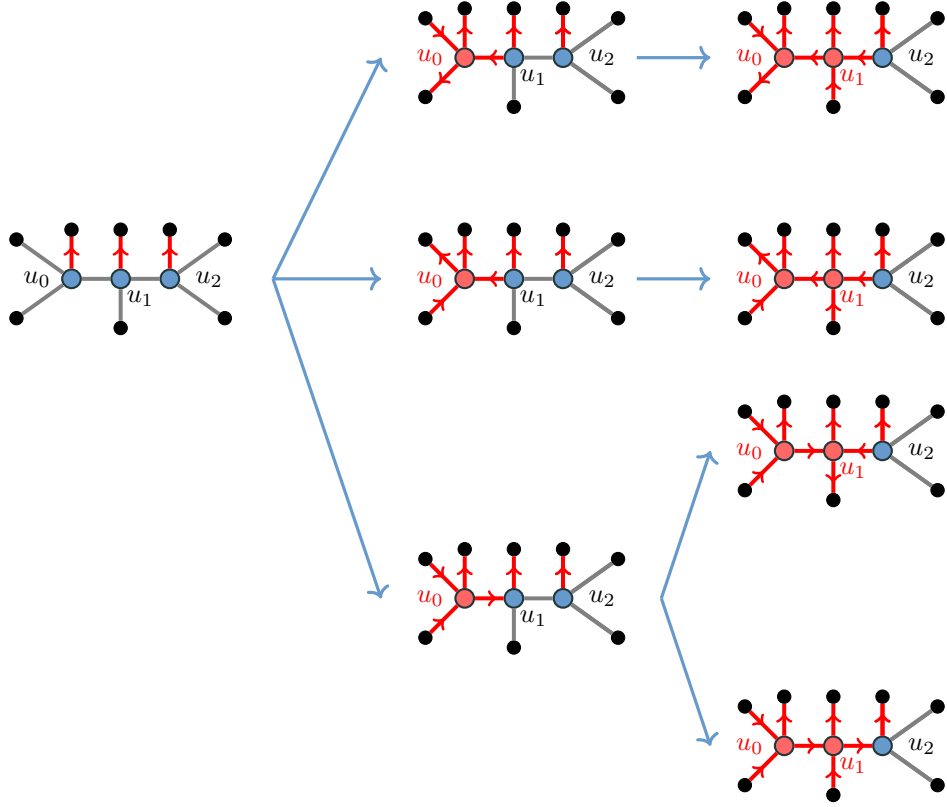

            \centering
            \Gpath
            \caption{
            Elimination of a path $\mathcal{P}=\{u_0,u_1,u_2\}$ connected by simple edges, demonstrating the elimination of $u_0$ and $u_1$. 
            The leading vertex $u_0$ is eliminated with a cost of $C(3,1)=3$. 
            In the valid mixed graphs generated by this initial elimination, we have $\mathcal{B}(u_1)=1$ and $\Gamma(u_1)=2$, corresponding to one vertex with profile $(2,1)$ and two trivial vertices. 
            Thus, the accumulated cost is $\mathcal{C}(\mathcal{P}_{u_1})=2\mathcal{B}(u_1)+\Gamma(u_1)=4$. 
            Subsequently, the elimination of $u_1$ yields $\mathcal{B}(u_2)=1$ and $\Gamma(u_2)=3$, resulting in a total path cost of $\mathcal{C}(\mathcal{P}_{u_2})=2\mathcal{B}(u_2)+\Gamma(u_2)=5$. 
            The average cost per degree-$2$ vertex is $(5/3)^{1/2}$, which is strictly lower than the worst-case bound for simple edges.
            } 
            \label{fig:pathel}
        \end{figure}

        The following lemma establishes a bound on the average cost of such paths for both structural cases.

        \begin{lemma}\label{lem:path}
            Let $\mathcal{P} = \{u_0, u_1, \dots, u_\ell\}$ be a path of $\ell+1$ vertices as described above, such that $d(u_0) \neq 2$. 
            If the path contains parallel edges, the average cost for the $\ell$ vertices eliminated with degree~$2$ is bounded by $\left(\displaystyle\frac{2^{\ell+1}+1}{3}\right)^{1/\ell}$. 
            If the graph is simple, this average cost is bounded by $5/3$.
        \end{lemma}
        \begin{proof}
            The bound for the simple graph case is proven in \cite{bet} (see Figure~\ref{fig:pathel}). 
            Here, we investigate the cost when vertices possessing parallel edges (Case a) are eliminated.
            
             First, consider paths consisting strictly of vertices connected by double edges.
             If the leading vertex has degree~$3$, then $\mathcal{B}(u_1)\leq 2$ and $\Gamma(u_1)\geq 1$, since in at least one orientation the edge $(u_0,u_1)$ is directed towards $u_1$. 
             Analogously, if the leading vertex has degree~$4$, then $\mathcal{B}(u_1)\leq 4$ and $\Gamma(u_1)\geq 2$.
             
             Assuming the worst-case scenario for $d(u_0)=3$, we initially have $\mathcal{B}(u_1)= 2$ and $\Gamma(u_1)= 1$. 
             Thus, the accumulated cost of the path up to this step is $\mathcal{C}(\mathcal{P}_{u_1})= 2\cdot 2 + 1 = 5$. 
             After the elimination of $u_1$, if $d(u_2)=2$, then $d(u_2)=4$ before the elimination of $u_1$, due to the double edge.
             This means that no edge incident to $u_2$ in the mixed graph $J$ is already directed in the previous steps of the elimination process (or else $d(u_2)>4$ in $J$), so $\mathcal{B}(u_2)= 4$ and $\Gamma(u_2)= 1$, and the accumulated cost is 
            $\mathcal{C}(\mathcal{P}_{u_2})=9$.
            
            Similarly, if a double edge is eliminated at every subsequent step of the path, we obtain the following recurrences:
            \begin{align}\label{eq:rec}
                \begin{split}
                    \mathcal{B}(u_{i+1})&=2\cdot \mathcal{B}(u_{i})\\
                    \Gamma(u_{i+1})&=\Gamma(u_{i})
                \end{split}
            \end{align}
            This recurrence dictates that at each step, the number of branches with cost~$2$ doubles, while the number of trivial branches remains constant. For a path of length $\ell$ originating from $d(u_0)=3$, this resolves to:
            \begin{align}\label{eq:cost_path}
                \begin{split}
                \mathcal{B}(u_{\ell}) &= 2^\ell\\
                \Gamma(u_{\ell}) &= 1\\
                \mathcal{C}(\mathcal{P}_{u_{\ell}}) &= 2^{\ell+1}+1
                \end{split}
            \end{align}
            
            In the case that $d(u_0)=4$, applying Equation~\eqref{eq:rec} yields:
            \begin{equation}\label{eq:cost_path4}
                \mathcal{C}(\mathcal{P}_{u_{\ell}})=2^{\ell+2}+2.
            \end{equation} 
            This concludes the case where only double edges are eliminated (see Figure~\ref{fig:pathleld} for an example).
            
            Let us now treat the case where both double and single edges exist in the path.
            Let $u_{i+1}$ be a vertex that attains degree~$2$ upon elimination, connected to the previously eliminated vertex $u_i$ by a single edge.
            Then according to \cite{bet}, the following equations hold in the scenario that maximizes the total cost:
            \begin{align}\label{eq:rec2}
                \begin{split}
                    \mathcal{B}(u_{i+1})&=  \mathcal{B}(u_{i}) + \Gamma(u_i)\\
                    \Gamma(u_{i+1})&= \mathcal{B}(u_{i}).
                \end{split}
            \end{align}
            
            Thus the subsequent cost is $\mathcal{C}(\mathcal{P}_{u_{i+1}})=3 \cdot \mathcal{B}(u_{i})+2\cdot \Gamma(u_i) $.
            
            This should be compared with $4 \cdot \mathcal{B}(u_{i})+ \Gamma(u_i)$, which is the cost derived by the recursive Equation~\eqref{eq:rec}.
            If $\Gamma(u_{i+1})\leq \mathcal{B}(u_{i+1})$, then Equations~\eqref{eq:cost_path} and \eqref{eq:cost_path4}—derived for paths with purely double edges—clearly bound the total cost for every possible path composition.
            A careful look at Equations~\eqref{eq:rec} and~\eqref{eq:rec2} confirms that this inequality holds universally. 
            A simple substitution into Equation~\eqref{eq:average} for both cases of leading vertices concludes the proof for $d(u_0)=3$ or $4$.
            
            Finally, if $d(u_0) = 1$, this vertex is excluded due to its trivial cost, and we instead consider the most recent vertex $v_0$ in the elimination sequence with $d(v_0) \notin \{1, 2\}$.
            Possibly, there are other vertices $v_i$ eliminated with degree~$2$ before $u_1$ in that case, or other degree-$1$ vertices that will also be excluded.
            We then apply a similar elimination procedure as described above for the extended path $\mathcal{P}^*=\{v_0, \dots, u_1, \dots, u_\ell\}$; the total cost remains unchanged for the subset of vertices eliminated with degree $\geq 2$.
        \end{proof}

        \begin{figure}[tbh]
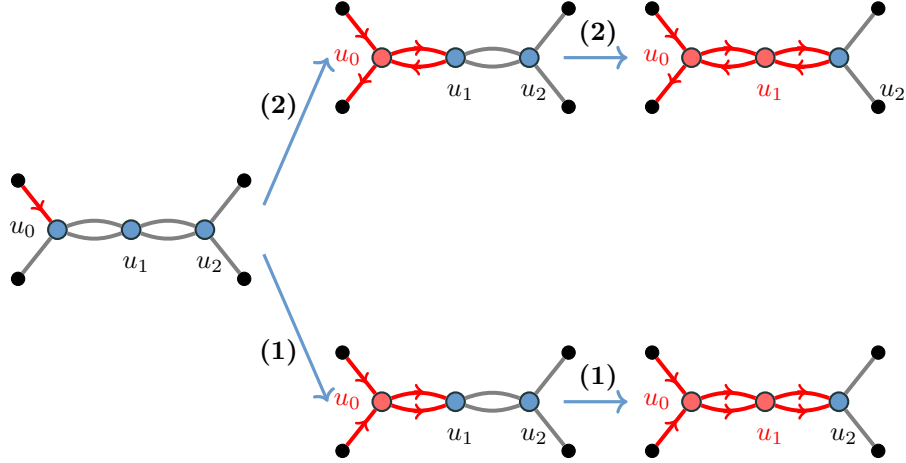

            \centering
            \GDoubleEdge
            \caption{
            Elimination of a path $\mathcal{P}=\{u_0,u_1,u_2\}$ connected by double edges, demonstrating the elimination of $u_0$ and $u_1$. 
            The bold numbers on the arrows indicate the multiplicity of each resulting branch. 
            The leading vertex $u_0$ is eliminated with a cost of $C(3,0)=3$. 
            In the mixed graphs generated by the elimination of $u_0$, we have $\mathcal{B}(u_1)=2$ and $\Gamma(u_1)=1$, corresponding to two branches where $u_1$ has profile $(2,1)$ and one where it is trivial. 
            Thus, the accumulated cost is $\mathcal{C}(\mathcal{P}_{u_1})=2\mathcal{B}(u_1)+\Gamma(u_1)=5$. 
            Following the multiplicative principle (as each corresponding arrow is labeled by $2$), the subsequent elimination of $u_1$ yields $\mathcal{B}(u_2)=4$ and $\Gamma(u_2)=1$. 
            This results in a total path cost of $\mathcal{C}(\mathcal{P}_{u_2})=2\mathcal{B}(u_2)+\Gamma(u_2)=9$. 
            The average cost per degree-$2$ vertex attains the maximum upper bound presented in Lemma~\ref{lem:path} for parallel edges.
            }
            \label{fig:pathleld}
        \end{figure}

        The previous lemma establishes that the total cost for a path of $\ell$ vertices is bounded by $(2^{\ell+1}+1)/3$ in multigraphs, and by $(5/3)^\ell$ in simple graphs.
     
        Finally, this iterative process terminates under the following stopping condition, originally proven in \cite{bev}:
        \begin{lemma}\label{lem:tree}
            Let $G_T=(V_T,E_T)$ be a simple undirected tree.  
            Then, for any prescribed outdegree constraints, there is at most one valid orientation satisfying those constraints.
        \end{lemma}
        
        We refer to this property as the \emph{tree condition}, which means that the step at which the undirected subgraph becomes a tree marks the stopping point of the elimination process.  
        Following \cite{bev,bet}, the undirected subgraph must remain connected at each step of the elimination process.  
        Therefore, if $G_T=(V_T,E_T)$ is the undirected subgraph of the mixed graph obtained at the end of the process, the equality $|V_T|=|E_T|+1$ holds.
        If the elimination process were to fracture the undirected subgraph into multiple connected components, this condition would fail and our bounding analysis would not apply.

        To summarize the complete elimination process:
        \begin{itemize}
            \item[a.] The process operates iteratively on a \emph{mixed graph} containing both directed and undirected edges.
            \item[b.] At each step, either the edges incident to one single vertex, or all the edges incident to a path of vertices are oriented and thus \emph{eliminated} from the undirected subgraph.
            This elimination must never disconnect the undirected subgraph.
            \item[c. ] In every step, a branching set of new valid mixed graphs is generated depending on the profiles of the neighbors of the eliminated vertices.
            The number of these new mixed graphs is the \emph{cost} of the step.
            \item[d.] The upper bound for the total number of branches generated after one step is given by the \emph{maximizing cost} $C^*$ for single vertex elimination step and by the \emph{total cost} for path eliminations (see Lemma~\ref{lem:path}).
            \item[e.] Since the steps form a branching process, the global upper bound for any sequence of eliminations is calculated by multiplying the individual upper bounds of each step.
            \item[f.] The process terminates when the undirected subgraph is a tree.
            
        \end{itemize}

    \subsection{On the structure of separable \texorpdfstring{$4$}{4}-regular graphs}\label{sec:stru_s}

        In this subsection, we introduce several definitions and properties that pertain primarily to the structure of separable $4$-regular graphs.
        These are utilized both in Section~\ref{sec:bound} and in Section~\ref{sec:ARS_count} to derive upper bounds and algorithmic results for connected separable graphs.
    
        First, we recall the definition of a block-cut tree \cite{Harary69}:
    
        \begin{definition}
            Let $G=(V,E)$ be a connected graph.  
            The \emph{block-cut tree} of $G$, is a graph in which each vertex represents either a biconnected component of $G$ or a cut vertex of $G$.  
            An edge in the block cut tree connects a biconnected component to a cut vertex if and only if the cut vertex belongs to that biconnected component.
        \end{definition}
        
        Block-cut trees are used for both the analysis of biconnected $4$-regular graphs and for separable graphs.
        Notice that if a biconnected component corresponds to a leaf in the block-cut tree, then it has exactly one cut vertex.
    
        Let us now state an observation about the cut vertices of $4$-regular graphs that follows directly from the fact that an Eulerian graph has no edges which are bridges.
    
        \begin{observation}\label{obs:hand}
            Let $G$ be a $4$-regular graph and $u$ be a cut vertex of $G$. 
            Then $G\setminus u$ has two connected components, $G_1$ and $G_2$, such that $u$ has exactly two incident edges connecting to each component.
        \end{observation}

        Building on this observation, we describe the structure of cut vertices.
        Given a $4$-regular loopless multigraph $G$, let $V_c$ be the subset of its cut vertices.
        Let $G[V_c]$ denote the subgraph of $G$ induced by $V_c$.
        The following lemma details the structure of $G[V_c]$.
        Note that a single isolated vertex is not considered as a biconnected component.

        \begin{lemma}\label{lem:structure 4regular}
            Each biconnected component of the induced subgraph $G[V_c]$ is a cycle.
        \end{lemma}

        \begin{proof}
            As stated in Observation~\ref{obs:hand}, the removal of a cut vertex  results in a graph with exactly two connected components, with the removed vertex having exactly two incident edges connecting to each component.
            This implies that in the induced subgraph $G[V_c]$  every vertex has degree 0, 2 or 4. 
            If any vertex were to have an odd degree in $G[V_c]$, it would necessitate the existence of a bridge in $G$, contradicting the fact that connected Eulerian graphs are bridgeless.
            
            Let $\mathcal{T}$ be a non-trivial connected component of $G[V_c]$, i.e. $\mathcal{T}$ is not an isolated vertex of degree $0$ in $G[V_c]$.
            Thus, every vertex in $\mathcal{T}$ has a degree of either $2$ or $4$.
            We proceed by induction on the number of degree-$4$ vertices in $\mathcal{T}$ to show that every biconnected component of $\mathcal{T}$ is a cycle.
            
            If $\mathcal{T}$ contains no vertices of degree $4$, then every vertex has exactly degree $2$. 
            Since $\mathcal{T}$ is connected, it must be a single cycle. 
            A cycle is inherently biconnected, meaning the single biconnected component of $\mathcal{T}$ is indeed a cycle.
        
            Assume that for any non-trivial connected component $\mathcal{T}$ with fewer than $k$ vertices of degree 4, every biconnected component is a cycle.
        
            We now prove the case where $\mathcal{T}$ has exactly $k$ vertices of degree $4$ ($k \ge 1$).
            Let $v$ be a vertex of degree 4 in $G[V_c]$.
            By Observation~\ref{obs:hand}, $G\setminus v$ consists of two connected components, $G_1$ and $G_2$, and $v$ has exactly two incident edges connecting to $G_1$ and two incident edges connecting to $G_2$.
            Since $v$ has degree $4$ in $G[V_c]$, the endpoints of all four of these edges must lie in $V_c$, meaning they all belong to $\mathcal{T}$. 
            Thus, $v$ is also a cut vertex of $\mathcal{T}$. 
            
            Let $\mathcal{T}_1$ and $\mathcal{T}_2$ be the subgraphs of $\mathcal{T}\setminus v$ that are contained within $G_1$ and $G_2$, respectively. 
            The two incidents edges of $v$ connecting to each of $G_1$ and $G_2$ are also incident to $\mathcal{T}_1$ and $\mathcal{T}_2$ respectively.
            Let $\mathcal{T}_1 \cup \{v\}$ be the subgraph obtained by restoring $v$ to $\mathcal{T}_1$ along with its two edges connecting to $\mathcal{T}_1$.
            Similarly, let $\mathcal{T}_2\cup \{v\}$ be the subgraph obtained by restoring $v$ and its two edges connecting to $\mathcal{T}_2$.
            Clearly, every biconnected component of $\mathcal{T}$ must be a biconnected component of either $\mathcal{T}_1\cup\{v\}$ or $\mathcal{T}_2\cup\{v\}$.
            
            Observe that except $v$, every vertex belonging in $\mathcal{T}_1\cup\{v\}$ or $\mathcal{T}_2\cup\{v\}$ has the same degree in the respective graph that belongs as in $\mathcal{T}$, whereas $v$ has degree 2 in both of them.
            This means that in $\mathcal{T}_1\cup\{v\}$ and $\mathcal{T}_2\cup\{v\}$ every vertex has degree 2 or 4 and the number of vertices of degree~4 in both of them is $k-1$ (since $v$ is no longer degree $4$).
            Consequently, each subgraph has strictly fewer than $k$ vertices of degree $4$. 
            By the induction hypothesis, every biconnected component of both $\mathcal{T}_1\cup\{v\}$ and $\mathcal{T}_2\cup\{v\}$ is a cycle.
            This concludes that every biconnected component of $\mathcal{T}$ is a cycle and the same holds also for $G[V_c]$.
        \end{proof}

        Following Lemma~\ref{lem:structure 4regular}, we characterize the cut vertices belonging to the biconnected components of $G[V_c]$ as \emph{cycle cut vertices}, while the isolated ones are the \emph{single cut vertices}.

        Let $V^2_{c}$ and $V^4_{c}$ denote the sets of cut vertices of a $4$-regular graph $G$ with degrees $2$ and $4$ in $G[V_c]$, respectively.
        Notice that non-cut vertices can be connected either with single cut vertices, or with $V^2_{c}$ cut vertices. 
        Then the following Lemma holds.

        \begin{lemma}\label{lem:count_c}
            Let $\mathcal{T}$ be a connected component of $G[V_c]$ which is not a single vertex.
            Let $V^2_{\mathcal{T}}$ and $V^4_{\mathcal{T}}$ be the subsets of vertices from $V^2_{c}$ and $V^4_{c}$ belonging to $V(\mathcal{T})$, respectively, and let $s$ be the number of cycles in $\mathcal{T}$.
            Then:
            \begin{itemize}
                \item[(i)]  $s=|V^4_{\mathcal{T}}|+1$.
                \item[(ii)] If every cycle in $\mathcal{T}$ is simple, then $|V^4_{\mathcal{T}}|\leq |V^2_{\mathcal{T}}|-3$.
                Equality holds when every cycle has length $3$.
            \end{itemize}
        \end{lemma}

        \begin{proof}
            We prove both (i) and (ii) by induction on the number of degree-$4$ vertices in a graph $\mathcal{T}$ where every vertex has degree $2$ or $4$.
            If $|V^4_{\mathcal{T}}|=0$ then $\mathcal{T}$ is a cycle. 
            Clearly, in $\mathcal{T}$ it holds that $s=|V^4_{\mathcal{T}}|+1=1$, satisfying (i).
            Furthermore, if this cycle is simple then it has at least 3 vertices.
            Thus, $|V^4_{\mathcal{T}}|\leq |V^2_{\mathcal{T}}|-3$. 
            The equality holds if the length of the cycle is 3, satisfying (ii).

            Now, assume that (i) and (ii) hold in  every graph $\mathcal{T}$ where every vertex has degree 2 or 4 and for the number of vertices of degree 4 $V^4_{\mathcal{T}}$  are strictly less than $k$ for some $k \geq 1$. 
            We will prove that they also hold when the graph has exactly $k$ vertices of degree $4$. 
            
            Let $v\in V^4_{\mathcal{T}}$, and let the subgraphs $\mathcal{T}_1\cup\{v\}$ and $\mathcal{T}_2\cup\{v\}$ be constructed as in the proof of Lemma~\ref{lem:structure 4regular}.
            Clearly, $\mathcal{T}_1\cup\{v\}$ and $\mathcal{T}_2\cup\{v\}$ satisfy the conditions of the induction hypothesis since, in each of them, every vertex has degree 2 or 4 and there are at most $k-1$ vertices degree 4.
            Let $V^2_{\mathcal{T}_1}$ and $V^4_{\mathcal{T}_1}$ be the vertices of degree 2 and 4 respectively in $\mathcal{T}_1\cup\{v\}$. 
            Similarly, let $V^2_{\mathcal{T}_2}$ and $V^4_{\mathcal{T}_2}$ be the vertices of degree 2 and 4 respectively in $\mathcal{T}_2\cup\{v\}$.
            
            Except for $v$, every vertex belonging in $\mathcal{T}_1\cup\{v\}$ or $\mathcal{T}_2\cup\{v\}$ has the same degree in the respective graph that belongs as in $\mathcal{T}$, whereas $v$ has degree 2 in both of them. 
            Let also  $s_1$ and $s_2$ be the number of cycles in $\mathcal{T}_1\cup\{v\}$ or $\mathcal{T}_2\cup\{v\}$ respectively. 
            Then, the following equalities hold:

            \begin{align*}
                |V^2_{\mathcal{T}}| &= (|V^2_{\mathcal{T}_1}|-1)+(|V^2_{\mathcal{T}_2}|-1) = |V^2_{\mathcal{T}_1}|+|V^2_{\mathcal{T}_2}|-2,  \\
                |V^4_{\mathcal{T}}| &= |V^4_{\mathcal{T}_1}|+|V^4_{\mathcal{T}_2}|+1, \text{ and } \\
                s &= s_1+s_2.
            \end{align*}
        
            \textbf{Proof of (i).} By the induction hypothesis, it holds that $s_1=|V^4_{\mathcal{T}_1}|+1$ and $s_2=|V^4_{\mathcal{T}_2}|+1$.
            Combining this with the above equalities yields:
            $$s = s_1+s_2 = (|V^4_{\mathcal{T}_1}|+1) + (|V^4_{\mathcal{T}_2}|+1) = |V^4_{\mathcal{T}}|+1.$$
        
            \textbf{Proof of (ii).} By the induction hypothesis, it holds that $|V^4_{\mathcal{T}_1}|\leq |V^2_{\mathcal{T}_1}|-3$ and $|V^4_{\mathcal{T}_2}|\leq |V^2_{\mathcal{T}_2}|-3$. 
            Combining this with the above equalities yields:
             $$|V^4_{\mathcal{T}}| = |V^4_{\mathcal{T}_1}|+|V^4_{\mathcal{T}_2}|+1 \leq (|V^2_{\mathcal{T}_1}|-3) + (|V^2_{\mathcal{T}_2}|-3) + 1 = |V^2_{\mathcal{T}}|-3.$$
        
            If $|V^4_{\mathcal{T}_1}|= |V^2_{\mathcal{T}_1}|-3$ and $|V^4_{\mathcal{T}_2}|= |V^2_{\mathcal{T}_2}|-3$, then every cycle in $\mathcal{T}_1$ and $\mathcal{T}_2$ has length $3$ by the induction hypothesis. 
            In this case, clearly every cycle in $\mathcal{T}$ also has length $3$, and the above inequality holds with equality, i.e., $|V^4_{\mathcal{T}}|=|V^2_{\mathcal{T}}|-3$.
        \end{proof}

        \begin{figure}[tbh]
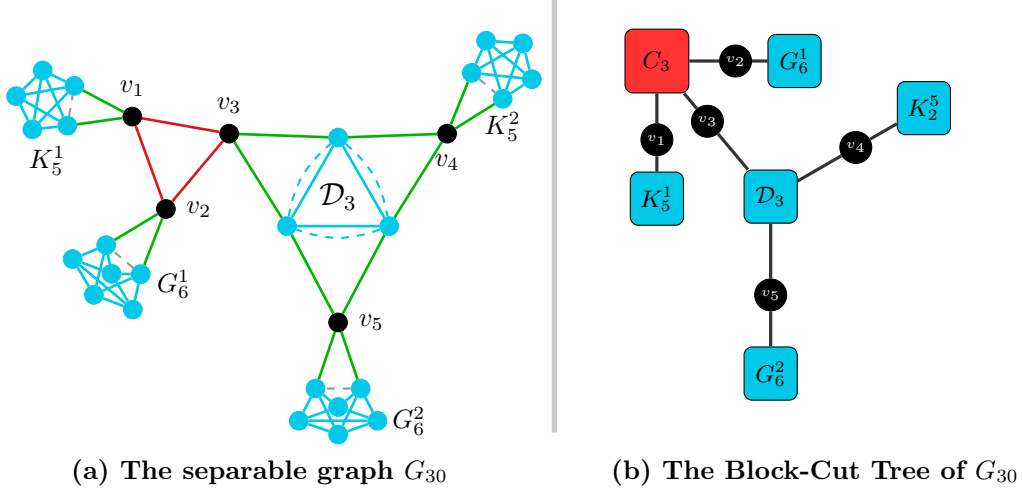

            \centering
            \separablexanple
            \caption{A separable $4$-regular graph $G_{30}$ and its corresponding block-cut tree. 
            The dashed edges do not belong to $G_{30}$. The graph contains four boundary $2$-ARSs and one inner $2$-ARS (highlighted in cyan). 
            With the addition of the dashed edges, the boundary ARSs form their regular representatives: two copies of $K_5$ and two copies of $G_6$ (the unique simple $4$-regular graph with $6$ vertices). 
            Observe that although $G_{30}$ is a simple graph, the regular representative of the inner $2$-ARS is a multigraph that forms a double cycle (see also the proof of Theorem~\ref{th:multi_b}). 
            The set of cut vertices is $V_c=\{v_1, \dots, v_5\}$. 
            The induced subgraph $G[V_c]$ consists of a single cycle ($v_1, v_2, v_3$) and two isolated vertices ($v_4, v_5$).  
            The leaves of the block-cut tree correspond exactly to the boundary $2$-ARSs.}
            \label{fig:ARSexamples}
        \end{figure}

        Finally, inspired by block-cut trees, we introduce a special class of subgraphs, the \emph{almost regular subgraphs (ARS)}(see Figure~\ref{fig:ARSexamples}).

        \begin{definition}\label{def:ARS}
            Let $G=(V,E)$ be a $4$-regular graph and $V_c$ be the set of its cut vertices. 
            Let $V_c^{\prime} \subseteq V_c$ be a non-empty subset of cut vertices. 
            A subgraph $G^{\prime}$ of $G$ is called an \emph{almost regular subgraph (ARS)} of $G$ if it is a connected component, or a union of connected components, of the subgraph $G\setminus V_c^{\prime}$ such that a connected $4$-regular graph $G^*$ is obtained by adding exactly one edge between a pair of vertices in $V(G^{\prime})$ for each common neighbor they share in $V_c^{\prime}$.
            The graph $G^*$ is the \emph{regular representative} of the ARS $G^{\prime}$.
            If, additionally, $G^*$ is biconnected, then $G^\prime$ is termed a \emph{$2$-ARS} of $G$.
        \end{definition}
        
        Recall that a biconnected component corresponding to a leaf of a block-cut tree contains exactly one cut vertex.
        This biconnected component, excluding the cut vertex itself, forms an ARS, particularly a $2$-ARS.
        Such an $2$-ARS is termed a \emph{boundary ARS}, while any non-boundary ARS is termed an \emph{inner ARS}.
        Note that an inner ARS is not necessarily connected, even if its regular representative is biconnected (see Figure~\ref{fig:ARSDthree}).

        Finally, we define two graph operations between regular graphs.

        \begin{definition}\label{def:glue_simple}
            Let $G^*_1, G^*_2$ be connected $4$-regular graphs.
            Given two edges $e_1=(u_{1},u_{2})\in E(G^*_1)$ and $e_2=(v_{1},v_{2})\in E(G^*_2)$, the operation that constructs a new graph $G$ by deleting these edges and connecting all four endpoints to a newly introduced vertex $v\notin V(G^*_1) \cup V(G^*_2)$ is called the \emph{simple gluing} of $G^*_1$ and $G^*_2$.
            This operation is denoted by:
            $$ G=  \underset{(e_1,e_2,v)}{\odot} \{G^*_1,G^*_2\}. $$
        \end{definition}
        
        \begin{definition}\label{def:glue_cycle}
            Let $\mathcal{T}$ be a connected graph in which every biconnected component is a cycle, and let $V^2_{\mathcal{T}}$ denote the set of all degree-$2$ vertices in $\mathcal{T}$.
            Let $G^*_1, \dots, G_{|V^2_{\mathcal{T}}|}^*$ be a collection of connected $4$-regular graphs.
            Given a set of edges $E^\prime= \{e_1, \dots, e_{|V^2_{\mathcal{T}}|}\}$ where $e_i\in E(G^*_i)$, we define an operation that constructs a new graph $G$ by deleting each edge $e_i$ from its respective graph $G^*_i$, and connecting both of its endpoints to a uniquely corresponding vertex $v_i\in V^2_{\mathcal{T}}$.
            This operation is termed the \emph{cycle gluing} of $G^*_1, \dots, G_{|V^2_{\mathcal{T}}|}^*$, and is denoted by:
            $$ G=  \underset{(E^\prime, \mathcal{T})}{\otimes} \{G^*_1,\dots, G^*_{|V^2_{\mathcal{T}}|}\}. $$
        \end{definition}

        \begin{figure}[tbh]
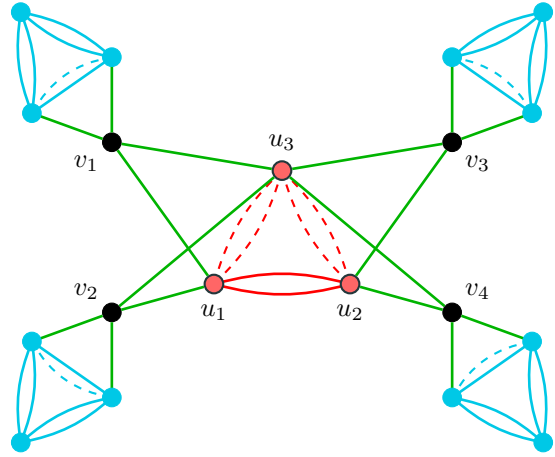

            \centering
            \ARSDthree
            \caption{
            In this figure, we illustrate a separable graph.
            The dashed edges in the figure are those that do not belong to the initial $4$-regular graph and they may be added to obtain the regular representatives of the ARSs.
            The only inner ARS of this graph is a union of two connected components: a double edge $(u_1, u_2)$ and an isolated vertex $u_3$ (both highlighted in red). 
            This disjoint subgraph strictly satisfies our definition of an ARS (see Definition~\ref{def:ARS}).
            By adding exactly one edge between the vertices $u_1, u_2, u_3$ for each common cut vertex they share, a connected $4$-regular graph (specifically, the double cycle $\mathcal{D}_3$ - see also the proof of Theorem~\ref{th:multi_b}) is successfully constructed.
            }
            \label{fig:ARSDthree}
        \end{figure}         
        
        Note that both operations construct larger $4$-regular graphs, and that any graph $G^*_i$ in the above definitions acts as the regular representative of an ARS in $G$.
        
    \section{Bounding the Eulerian orientations.}\label{sec:bound}

        In this section, we first establish an exact bound for biconnected multigraphs, proving Theorem~\ref{th:multi_b}.
        Then, we prove Theorems~\ref{th:bi} and \ref{th:sep}, which bound the number of Eulerian orientations for biconnected and separable simple $4$-regular graphs, respectively.
        All proofs utilize techniques similar to those presented in \cite{bet}.
        Furthermore, we present computational results for simple graphs and compare them with both our derived bounds and the existing bound by Las Vergnas \cite{LasVergnas}.
        
        As noted in Section~\ref{sec:pre}, throughout the elimination process, the undirected subgraph is required to remain connected.
        Naturally, if the procedure is permitted to select vertices of degree $4$, maintaining this connectivity is trivial.
        
        However, in the following lemma, we prove that after the initial step, it is always possible to select and orient the incident edges of a vertex $v$ with degree $d(v) \leq 3$ in the undirected subgraph of any mixed graph $J$ originating from an initially \emph{biconnected $4$-regular graph}.

        \begin{lemma}\label{lem:bi3}
            Let $G_J$ be the undirected subgraph of a mixed graph $J$ during the elimination process after the initial step, and let $G_b$ be a biconnected subgraph of $G_J$.
            Let also $v$ be the next vertex whose edges are to be oriented, such that $v\in G_b$ and $G_{b}^{^\prime}=G_b\setminus v$ is separable.
            Then, every biconnected component $G_{b_i}$ of $G_{b}^{\prime}$ has at least two vertices $v_1, v_2$ such that $d_{G_{b_i}}(v_1) \leq 3$ and $d_{G_{b_i}}(v_2) \leq 3$. 
        \end{lemma}
        
        \begin{proof}
        Let $G_{b_i}$ be a biconnected component of $G_{b}^{\prime}$ and let $v_1\in V(G_{b_i})$ be a cut vertex of $G_{b}^{\prime}$.
        This implies that $v_1$ is connected to at least one vertex $u \notin V(G_{b_i})$, and thus $d_{G_{b_i}}(v_1)\leq 3$.
        Since $G_b$ is biconnected,  $G_b\setminus v_1$ is connected. This implies that  for every $u^\prime\in V(G_{b_i})\setminus v_1$ there exists a path $P$ in $G_b\setminus v_1$ between $u$ and $u^\prime$.
        The vertex $v$ must belong to every such path $P$; otherwise, $v_1$ would not be a cut vertex of $G_{b}^{\prime}$.
        
        In such a path $P$, since $u\notin V(G_{b_i})$ and $u^\prime\in V(G_{b_i})$, there must exist an edge $e=(v_2,v^{\prime})$ in $P$ such that $v_2\in V(G_{b_i})$ and either $v^{\prime}=v$ or $v^{\prime}\in V(G_{b}^{\prime})\setminus V(G_{b_i})$.
        Then $d_{G_{b_i}}(v_2)\leq 3$ and this concludes that $G_{b_i}$ contains at least 2 vertices $v_1, v_2$ of degree at most 3.
        
        \end{proof}
        
        \begin{corollary}\label{cor:bi}
        Let $J$ be a mixed graph appearing after the first step of the elimination process applied to a biconnected $4$-regular graph.
        If $G_b$ is a biconnected component corresponding to a leaf in the block-cut tree of the undirected subgraph $G_J$, then there is a vertex $u$ in $G_b$ such that $d_{G_b}(u)\leq 3$ and $u$ is not a cut vertex in $G_J$.
        \end{corollary}
        
        The lemma and corollary above confirm the existence of an elimination process for biconnected graphs such that, after the first step, whenever a new cut vertex appears during an elimination step, it is guaranteed that at least one non-cut vertex $v\in V_J$ has degree $d(v) \leq 3$ due to the biconnectivity of the component in the previous step.
        Thus, it is always possible to orient the incident edges of vertices with degree at most $3$ in the undirected subgraph, while keeping the undirected graph connected.
        This is evidently applicable to both multigraphs and simple graphs.

        \begin{figure}
            \begin{center}
		      \begin{tabular}{cc}
			
			\begin{tikzpicture}[scale=1.2]
			         \fourbond
			\end{tikzpicture}
			\hspace*{6mm}& \hspace*{6mm}
			\begin{tikzpicture}[scale=1.2]
			         \dfour
			\end{tikzpicture}\\
            a. $4$-bond multigraph $M_2$ \hspace*{6mm}& \hspace*{6mm} b. $\mathcal{D}_4$ graph.
		      \end{tabular} 
		\caption{
            (a) The $4$-bond multigraph $M_2$ with two vertices and four parallel edges, and (b) the graph $\mathcal{D}_4$ consisting of four vertices forming a double cycle. 
            Both graphs exactly attain the upper bound presented in Theorem~\ref{th:multi_b}; specifically, $|\mathcal{E}(M_2)|=2^2+2=6$ and $|\mathcal{E}(\mathcal{D}_4)|=2^4+2=18$. 
            Note that $M_2$ also exactly attains Schrijver's bound~\cite{Schrijver1983}.
            }\label{fig:multi_b}
	       \end{center}
        \end{figure}
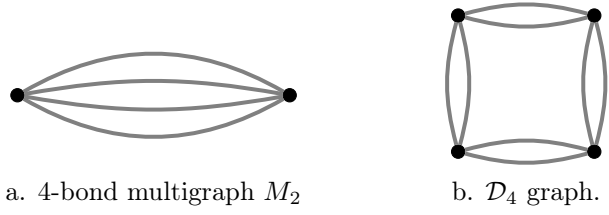

        Now we are ready to treat the case of biconnected multigraphs.

        \begin{proof}[Proof of Theorem~\ref{th:multi_b}]
            Let $G=(V,E)$ be a $4$-regular biconnected multigraph with $n$ vertices.
            Throughout an elimination process as described in Section~\ref{sec:elim}, the vertices are eliminated according to the following degrees:
            \begin{itemize}
                \item $1$ vertex with degree $4$ in the first step, with cost $C(4,0)=6$ (due to Corollary~\ref{cor:bi}).
                \item $n_1$ vertices with degree $1$, with cost $C(1,1)=C(1,2)=1$.
                \item $n_2$ vertices with degree $2$, with a maximum average cost of $\left(\frac{2^{n_2+1}+1}{3}\right)^{1/n_2}$ according to Lemma~\ref{lem:path}.
                \item $n_3$ vertices with degree $3$, with maximizing cost $C^*(3)=3$.
            \end{itemize}
            Thus, the total number of Eulerian orientations is bounded by:
            \begin{equation}\label{eq:pbcbound1}
               |\mathcal{E}(G)| \leq 6\cdot 3^{n_3-1}\cdot \left( 2^{n_2+1}+1\right)= 3^{n_3}\cdot \left( 2^{n_2+2}+2\right).
            \end{equation}

            Since at each step the number of newly oriented edges equals the degree $d(v)$ of the eliminated vertex $v$, the undirected subgraph in the subsequent step has exactly $d(v)$ fewer edges and $1$ fewer vertex.
            Recall also that the elimination process terminates when the undirected graph $G_T=(V_T,E_T)$ forms a tree.
            Therefore, we have the following equations:
            \begin{align*}
              |V_T| &= n-n_1-n_2-n_3-1 \\
              |E_T| &= 2n-n_1-2n_2-3n_3-4\\
              |V_T| & = |E_T|+1
            \end{align*}
            yielding the relation:
            \begin{equation}\label{eq:tree_multi}
                n=2n_3+n_2+2.
            \end{equation}

        	Substituting this into Equation~\eqref{eq:pbcbound1}, we derive:
            \begin{equation}\label{eq:final_multi_1}
               |\mathcal{E}(G)| \leq 3^{n_3}\cdot \left( 2^{n-2n_3}+2 \right).
            \end{equation}

            We will show that the upper bound in Equation~\eqref{eq:final_multi_1} is a non-increasing function of $n_3$ for a fixed $n$. To show this, we define the function $f(n_3)=3^{n_3}\cdot (2^{n-2n_3}+2)$ and examine the ratio:
            $$ \frac{f(n_3)}{f(n^\prime_3)}, $$ 
            with $n_3^\prime=n_3+1$.
            Using Equation~\eqref{eq:tree_multi} and the fact that $n_2 \geq 0$, we establish the inequalities $2n^\prime_3\leq n-2 \Rightarrow 2n_3 \leq n-4$.
            Thus, we evaluate:
            \begin{align*}
              \frac{f(n_3)}{f(n^\prime_3)} & \geq 1 &\Rightarrow  \\
              \frac{3^{n_3}\cdot (2^{n-2n_3}+2)}{3^{n_3+1}\cdot (2^{n-2n_3-2}+2)} & \geq 1 &\Rightarrow \\
              2^{n-2n_3}+2 & \geq   \frac{3}{4}2^{n-2n_3}+6 &\Rightarrow\\
              2^{n-2n_3} & \geq 16 
            \end{align*}
            which clearly holds true given $2n_3 \leq n-4$.
            Thus, the upper bound is strictly maximized when $n_3 = 0$, giving:
            \begin{equation}\label{eq:final_multi}
                |\mathcal{E}(G)| \leq 2^{n}+2.  
            \end{equation}

            Clearly, for the $4$-bond multigraph ($2$ vertices, $4$ edges), the bound is exact, since the first vertex is eliminated with cost $6$ and consequently all the incident edges of the second vertex are already oriented.
            This yields exactly $6$ Eulerian orientations for the whole graph.
            In order to construct larger graphs that attain the upper bound, we must satisfy $n_3=0$ (as shown by the final inequality evaluating Equation~\eqref{eq:final_multi_1}).
            Additionally, the vertices eliminated with degree $2$ must share a double edge, in order to satisfy the equality in Equation~\eqref{eq:pbcbound1} (see the proof of Lemma~\ref{lem:path}).
            
            Let $\mathcal{D}_n$ denote a graph with $n$ vertices that forms a \emph{double cycle}, meaning that each vertex is connected to exactly two other vertices by two parallel edges (see Figure~\ref{fig:multi_b}).
            Every such graph exactly attains the bound; the first vertex is eliminated as a degree $4$ vertex, and all subsequent vertices are eliminated as degree $2$ vertices with a double edge, except for the final one.
            In this specific construction, $n_3=0$ and $n_2=n-2$, which perfectly yields the equality in Equation~\eqref{eq:final_multi}.
        \end{proof}

        Now we treat the case of simple biconnected graphs.

        \begin{proof}[Proof of Theorem~\ref{th:bi}]
            Let $G=(V,E)$ be a simple $4$-regular biconnected graph with $n$ vertices.
            The main difference with the proof of Theorem~\ref{th:multi_b} is that the maximum average cost for vertices eliminated with degree~2 is $5/3$ (see Lemma~\ref{lem:path}).
            Let $n_1, n_2, n_3$ be defined as in the proof Theorem~\ref{th:multi_b}.
            Following the same procedure, applying the maximizing cost for degree-$2$ vertices in simple graphs, and substituting $n_3 = (n-n_2-2)/2$ from the tree relation in Equation~\eqref{eq:tree_multi}, we obtain the following upper bound on the number of orientations:
            $$ |\mathcal{E}(G)| \leq 6\cdot 3^{n_3} \cdot \left(\frac{5}{3} \right)^{n_2}= 6\cdot 3^{(n-n_2-2)/2}\cdot \left(\frac{5}{3}\right)^{n_2} \leq 2\cdot 3^{n/2}, $$
            which holds since $5/3^{3/2} < 1$.
            This establishes the upper bound for the simple biconnected case.
        \end{proof}

        Finally, we give the proof for Theorem~\ref{th:sep}, which applies the same methods as before to provide a bound for simple separable graphs.
        Notice that Lemma~\ref{lem:bi3} and Corollary~\ref{cor:bi} do not apply here for to the graph as a whole; since the graph is separable, the elimination process may be forced to select multiple degree-$4$ vertices to maintain the connectivity of the undirected subgraph.
        For example, during the elimination of $G_{30}$ (see Figure~\ref{fig:ARSexamples}), if one of the biconnected components attached to the cycle of cut vertices is fully eliminated, one must either eliminate a cut vertex (which disconnects the graph) or eliminate a degree-$4$ vertex.
        This occurs due to the absence of $2$-vertex connectivity of $G_{30}$; since the graph as a whole is not a single biconnected component, Corollary~\ref{cor:bi} cannot be applied.
        This naturally yields a larger upper bound than in the biconnected case.

        \begin{proof}[Proof of Theorem~\ref{th:sep}]
            Let $G(V,E)$ be a $4$-regular simple separable connected graph.
            Then, let $n_2, n_3, n_1$ be as in the proof of Theorem~\ref{th:bi}.
            Let also $n_4$ denote the number of vertices eliminated with degree $4$, since multiple degree~4 vertices may be eliminated in that case.
            Then the tree condition leads to the following equations:
            \begin{align*}
              |V_T| &= n-n_1-n_2-n_3-n_4 \\
              |E_T| &= 2n-n_1-2n_2-3n_3-4n_4\\
              |V_T| & = |E_T|+1 \\
            \end{align*}
            yielding
            $$ n=3n_4+2n_3+n_2-1$$
            and the number of orientations is bounded by 
            
            \begin{equation}\label{eq:final_simple}
                |\mathcal{E}(G)| \leq 6^{n_4}\cdot  3^{n_3}\cdot \left( \frac{5}{3}\right)^{n_2} = 
                6^{(n+1)/3}\cdot \left(\frac{3}{6^{2/3}}\right)^{n_3}\cdot \left(\frac{5}{3\cdot 6^{1/3}}\right)^{n_2} \leq 6^{(n+1)/3}.
            \end{equation}
        \end{proof}

        \paragraph{Computational results.}
        Let us now provide computational results regarding the exact number of Eulerian orientations for all simple $4$-regular graphs with up to $15$ vertices, summarized in Table~\ref{tab:comps}.
        To generate all such graphs up to isomorphism, we utilized the \texttt{nauty} graph generator \cite{nauty} within \textsc{Sagemath}, and we computed their respective number of orientations using the algorithm detailed in \cite{code_mBezout}.
        The complete dataset from our exhaustive computations is available in \cite{dataforpaper}.
        Illustrations of the specific graphs that attain the maximum number of orientations for a given $n$, in both the separable and biconnected cases, are provided in Appendix~\ref{app:simple}.
        
        We remark that our theoretical bounds are not tight, at least for graphs with a small number of vertices, although they constitute a major improvement over existing bounds.
        Furthermore, we observe that for $11 \leq n \leq 15$, the separable graph $G$ maximizing $|\mathcal{E}(G)|$ possesses strictly more orientations than any biconnected graph of the same order (note that no simple separable $4$-regular graphs exist for $n \leq 10$).
        Another observation is that the graphs that maximize the orientation count among biconnected graphs with $n\leq 9$ are $3$-connected, whereas for $10 \leq n \leq 15$ they are strictly biconnected (i.e., having a vertex connectivity of exactly $2$; see the figures in Appendix~\ref{app:simple_bi}).

        \begin{table}[tbh]
        \label{tab:comps}
            \centering
            \begin{tabular}{|c|c|c|c|c|c|c|c|}
                \hline
                \multirow{2}{*}{$\bm{n}$} & \multicolumn{3}{c|}{\textbf{Biconnected Graphs}} & \multicolumn{3}{c|}{\textbf{Separable Graphs}} &  \\ \cline{2-7}
                & $\#$  & max  & biconnected & $\#$  & max  & separable & Las Vergnas' \\
                & of graphs & $|\mathcal{E}(G)|$ & bound & of graphs & $|\mathcal{E}(G)|$ &   bound & bound \\
                \hline
                5 & 1 & 24 & 31 & -- & -- & -- & 36 \\ \hline
                6 & 1 & 38 & 54 & -- & -- & -- & 72 \\ \hline
                7 & 2 & 60 & 94 & -- & -- & -- & 144 \\ \hline
                8 & 6 & 96 & 162 & -- & -- & -- & 288 \\ \hline
                9 & 16 & 160 & 281 & -- & -- & -- & 576 \\ \hline
                10 & 59 & 288 & 486 & -- & -- & -- & 1152 \\ \hline
                11 & 264 & 480 & 842 & 1 & 576 & 1296 & 2304 \\ \hline
                12 & 1542 & 800 & 1458 & 2 & 960 & 2355 & 4608 \\ \hline
                13 & 10768 & 1344 & 2525 & 10 & 1600 & 4279 & 9216 \\ \hline
                14 & 88126 & 2304 & 4374 & 42 & 2688 & 7776 & 18432 \\ \hline
                15 & 805281 & 4224 & 7576 & 210 & 4608 & 14130 & 36864 \\ \hline
            \end{tabular}
            \caption{
                Eulerian orientations of simple $4$-regular graphs for the biconnected and separable case, and comparison with bounds.
                The total number of graphs for each vertex count $n$ is evaluated up to isomorphism. 
                Regarding the orientations, only the maximum number attained for a given $n$ is reported. 
                Illustrations of the specific graphs that maximize this count can be found in Appendix~\ref{app:simple}.
                The \textit{biconnected bound} and \textit{separable bound} columns correspond to the upper bounds established in Theorems~\ref{th:bi} and \ref{th:sep}, respectively, while \textit{Las Vergnas' bound} refers to the bound derived in \cite{LasVergnas}. 
                All evaluated bounds are rounded to the nearest integer.
            }
        \end{table}

    \section{Applying the ARS decomposition to count Eulerian orientations for separable graphs. }\label{sec:ARS_count}

        This section investigates the number of Eulerian orientations for separable connected $4$-regular graphs.
        Here, we apply the structural properties established in Section~\ref{sec:stru_s} - in particular, the ARS decomposition and our observations regarding cut vertices.
    
        First, we provide a formula to compute the total orientations as a function of the orientations of the constituent $2$-ARSs and the induced structure of the connected components of the cut vertices.
        This formula leads to a divide-and-conquer algorithm that computes the exact number of Eulerian orientations for separable graphs without relying on exhaustive enumeration.
        Furthermore, we demonstrate that the biconnected bound of Theorem~\ref{th:bi} is strictly insufficient for simple separable graphs, and we conjecture a sharp bound that is asymptotically tighter than the one provided in Theorem~\ref{th:sep}.

    \subsection{Computing the Eulerian orientations of separable connected graphs}

        In order to prove a closed formula for the computation of Eulerian orientations of separable connected graphs, we first establish several lemmas regarding the orientation of edges incident to cut vertices and the Eulerian orientations of ARSs.
        
        Observation~\ref{obs:hand} immediately yields the following lemma:
        
        \begin{lemma}\label{lem:cut}
            Let $G$ be a $4$-regular graph, let $u$ be a cut vertex, and let $G_1$ and $G_2$ be the two connected components of $G \setminus u$, as defined in Observation~\ref{obs:hand}.
            Then, in any Eulerian orientation of $G$, exactly one of the two incident edges connecting $u$ to $G_1$ is directed towards $u$, and the other is directed outwards $u$.
            The same strictly holds for $G_2$.
        \end{lemma}
        
        \begin{proof}
            We treat the case of $G_1$.
            The proof for $G_2$ is analogous.
            Let $m_1 = |V(G_1)|$ be the number of vertices in $G_1$.
            In any Eulerian orientation of $G$, every vertex must attain an outdegree of exactly $2$.
            Therefore, the sum of the outdegrees of all vertices in $V(G_1)$ must equal exactly $2m_1$.
            The only edges in $G$ that have exactly one endpoint in $V(G_1)$ are the two incident edges connecting $u$ to $G_1$.
            This implies that the sum of the degrees of vertices in $G_1$ is exactly $4m_1-2$ and thus $|E(G_1)|=(4m_1 - 2) / 2 = 2m_1 - 1$. 
            Clearly, the edges of $E(G_1)$ contribute exactly $2m_1 - 1$ to the sum of outdegrees of $V(G_1)$ in every Eulerian orientation of $G$.
            Consequently, the remaining required outdegree of $1$ must come from an edge being directed from a vertex in $G_1$ to $u$.
            Thus,  exactly one of these edges must be directed towards $u$ and the other outwards $u$
        \end{proof}
        
        The following lemma establishes a relation between the valid partial orientations of an ARS and the Eulerian orientations of its regular representative.
        
        \begin{lemma}\label{lem:ARS}
            Let $G$ be  a separable connected $4$-regular graph, and let $G^\prime=(V^\prime,E^\prime)$ be an ARS of $G$.
            Let also $G^*(V^\prime,E^*)$, with $E^\prime \subset E^*$, be the regular representative of $G^\prime$.
            Then, there is a bijection between the valid partial Eulerian orientations of $G^\prime$ (as induced by an Eulerian orientation of $G$) and the Eulerian orientations of $G^*$.
            \end{lemma}
            
        \begin{proof}
            Let $\vec{E}$ be an Eulerian orientation of $G$.
            Then every vertex in $G^\prime$ that is not adjacent to a cut vertex has outdegree $2$ in this subgraph $G^\prime$ under $\vec{E}$.
            For the edges connecting $G^\prime$ to a cut vertex $u \notin V^\prime$, recall from Lemma~\ref{lem:cut} that exactly one edge is directed towards $u$ and the other is directed outwards $u$.
            Let $v_1, v_2 \in V^\prime$ be the endpoints of these edges in $G^\prime$ (note that $v_1 = v_2$ if they form a double edge).
            Suppose the edges are directed from $v_1$ to $u$, and from $u$ to $v_2$ in $\vec{E}$.
            Then $v_1$ has outdegree $1$ and $v_2$ has outdegree $2$ in $G^\prime$.
        
            This partial orientation of $G$ is equivalent to an Eulerian orientation $\vec{E}^*$ of $G^*$, defined as follows:
            \begin{itemize}
                \item[a.] every edge in $E^\prime$ has the same orientation as in $\vec{E}$ and
                \item[b.] for every triplet $v_1, v_2, u$ as above there is an edge $(v_1, v_2)\in E^*\setminus E^\prime$ directed from $v_1$ to $v_2$.
            \end{itemize}

            Conversely let $\vec{E^*}$ be an Eulerian orientation of $G^*$.
            Then $\vec{E^*}$ naturally extends to a unique partial Eulerian orientation of $G$ reversing the steps (a.) and (b.), establishing the bijection.
        \end{proof}
        
        Finally, the following lemma is used to count the orientations of edges where both endpoints are cut vertices.
        \begin{lemma}\label{lem:circle}
            Let $u_1, \dots u_r$ be cut vertices forming a simple cycle with edges $e_1=(u_1,u_2) \dots e_r=(u_r,u_1)$, or two cut vertices connected with $2$ parallel edges.
            Then, for every orientation for the rest of the graph, there are exactly two ways to orient these cycles.
        \end{lemma}
        \begin{proof}
            Lemma~\ref{lem:cut} shows that after orienting all edges of a $2$-ARS, an edge incident to the cut vertex is directed outwards it.
            Therefore, within the cycle of cut vertices itself, every vertex has one edge already directed outwards and one edge directed towards it and needs exactly one additional incoming edge and one additional outgoing edge to satisfy the outdegree constraint.
            Thus, there are exactly two ways to obtain oudegree $2$ for all vertices, namely, by directing all cycle edges either clockwise or counter-clockwise. 
        \end{proof}
        
        \begin{algorithm}\label{algo:ARS}
            \caption{Compute orientations for separable graphs.}
        	\DontPrintSemicolon
        	\KwFn{countEulerian}\\
        	\KwInput{$G$ (separable $4$-regular graph)}
        	\KwOutput{\# of Eulerian orientations}
            \tcc{Step 1: Identify cut vertices and their induced subgraphs}
            $V_c \gets $ set of cut vertices of $G$ \;
            $V^0_c \gets $ cut vertices with degree $0$ in $G[V_c]$\;
            $V^2_c \gets $ cut vertices with degree $2$ in $G[V_c]$\;
            $G^* \gets G \setminus V_c$\; 
            \BlankLine
        
            \tcc{Step 2: Construct regular representatives of 2-ARSs connected via a cycle of cut vertices}
            \For{all  $v_c \in V^2_c$}{
                     $\{u_1,u_2\} \gets$  neighbors of $v_c$ in $V \setminus V_c$\;
                     $E(G^*) \gets E(G^*) \cup \{(u_1,u_2) \}$\;
        		}
        
            \tcc{Step 3: Construct regular representatives of 2-ARSs joined at a single cut vertex}
            \For{each $v_c \in V^0_c$}{
                $\{u_1, u_2, u_3, u_4\} \gets $ neighbors of $v_c$ in $G$\;
                $R_{temp} \gets $ the component in $\mathcal{R}$ containing $u_1$\;
                $u^* \gets $ the unique vertex in $\{u_2, u_3, u_4\} \cap V(R_{temp})$\;
                $\{w_1, w_2\} \gets \{u_2, u_3, u_4\} \setminus \{u^*\}$\;
                $E(G^*) \gets E(G^*) \cup \{(u_1, u^*), (w_1, w_2)\}$\;
            }
        
            \tcc{Step 4: Compute Eulerian orientations for each regular representative of $2$-ARS using the function \texttt{orient} that counts outdegree-constrained orientations.}
            
            $|\mathcal{E}| \gets 1$\;
            $\mathcal{R} \gets $ connected components of the updated $G^*$\;
            \For{all  $G_i^*\in \mathcal{R}$}{
                     $|\mathcal{E}| \gets |\mathcal{E}| \cdot \texttt{orient}(G^*_i)$\;
        		}
        
            \tcc{Step 5: Count the cycles of the subgraph induced by the cut vertices.}
            $s \gets $ \# of cycles in $G[V_c]$\\
            \Return($2^s \cdot |\mathcal{E}| $)
            
        \end{algorithm}

        We now have all the ingredients to prove the following theorem, which provides a closed formula to compute the number of Eulerian orientations of a separable connected $4$-regular graph, provided that the number of Eulerian orientations of the regular representatives of its $2$-ARSs is known.
        
        \begin{theorem}\label{th:formula}
            Let $G=(V,E)$ be a separable connected $4$-regular graph, and let $V_c$ denote its set of cut vertices.
            Let $s \geq 0$ be the number of cycles in the induced subgraph $G[V_c]$. 
            Let $G_1, G_2, \dots, G_r$ be the constituent $2$-ARSs of $G$, and let $G^*_1, G^*_2, \dots, G^*_r$ be their respective regular representatives. 
            Then:
            \begin{equation}\label{eq:formula}
                |\mathcal{E}(G)| = 2^s \cdot \prod_{i=1}^r |\mathcal{E}(G^*_i)|
            \end{equation}
        \end{theorem}
        
        \begin{proof}
            According to Lemma~\ref{lem:ARS}, the valid partial orientations of every $2$-ARS map bijectively to the Eulerian orientations of its corresponding regular representative.
            Furthermore, by Lemma~\ref{lem:circle}, every cycle of cut vertices has exactly two possible orientations. 
            The choice of orientation in one $2$-ARS does not constrain the choice in another $2$-ARS or in $G[V_c]$ due to Lemma~\ref{lem:cut}, which ensures that every valid orientation of an ARS provides exactly one incoming and one outgoing edge at each incident cut vertex. 
            Thus, the $2$-ARSs are independent and vertex-disjoint, and Equation~\eqref{eq:formula} follows.
        \end{proof}
        
        Let us give an example that utilizes this formula to compute the number of orientations for a given graph.
        The regular representatives of the $2$-ARS of the graph $G_{30}$ (see Figure~\ref{fig:ARSexamples}) are two copies of $K_5$, two copies of the only $6$-vertex simple $4$-regular graph and one copy of the $\mathcal{D}_3$ graph.
        Additionally the induced subgraph of the cut vertices forms a single cycle.
        Since $|\mathcal{E}(K_5)|=24, |\mathcal{E}(G_6)|=38$ (see Appendix~\ref{app:simple_bi}) and $|\mathcal{E}(\mathcal{D}_3)|=10$ (see Theorem~\ref{th:multi_b}), we conclude that: $$|\mathcal{E} (G_{30})|=2\cdot 24^2 \cdot 38^2 \cdot 10= 16,634,880.$$
        
        Using this closed formula, a simple divide-and-conquer algorithm can be employed to compute the exact number of orientations for separable graphs without relying on exhaustive enumeration (see Algorithm~\ref{algo:ARS}).
        In this algorithm we compute separately the orientations for each regular representative of every $2$-ARS, using the function \texttt{orient} from~\cite{code_mBezout}.

        An implementation of this algorithm can be found in \cite{dataforpaper}.
        Using this implementation we computed the number of Eulerian orientations for large graphs with relatively small ARSs, e.g. $G_{30}$ (see Figure~\ref{fig:ARSexamples}) in less than 20 ms.
        By contrast, performing the computation for $G_{30}$ using the original algorithm from \cite{code_mBezout} required almost 2 hours on an Intel i3-3220 PC (3.3 GHz CPU, 4 GB RAM). 

    \subsection{A Conjecture on the upper bound for the Eulerian orientations of simple separable connected \texorpdfstring{$4$}{4}-regular graphs.}

        Now we are going to present a class of simple separable connected graphs that surpasses the number of Eulerian orientations of every simple biconnected graph with the same number of vertices.
        Recall from Table~\ref{tab:comps} and Appendix~\ref{app:simple_bi} that $|\mathcal{E}(K_5)|=24$, where $K_5$ is the unique simple $4$-regular graph with $5$ vertices.
    
        \begin{figure}[htp!]
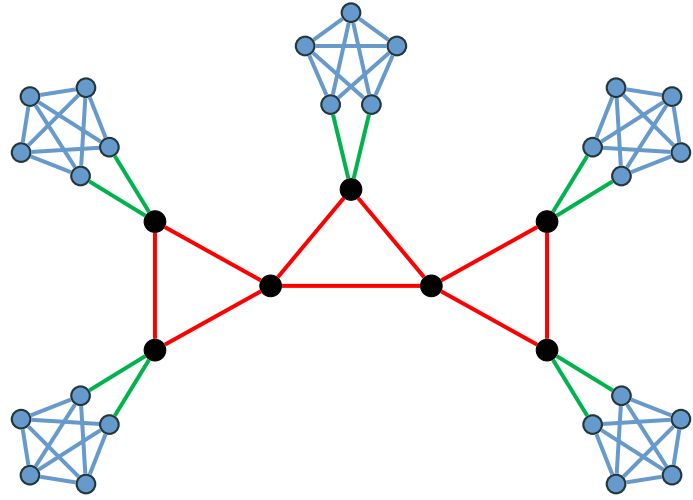

            \centering
            \hthirtyfive
            \caption{
            The graph $H_{32} \in \mathcal{H}$ with $32$ vertices, which exactly attains the number of orientations established in Theorem~\ref{th:48_7}. 
            This graph contains $5$ ARSs, each with $K_5$ as its regular representative, along with $3$ cycles of cut vertices. 
            Consequently, $|\mathcal{E}(H_{32})|=24^5\cdot 2^3= 2^{-2}\cdot 48^{(32+3)/7}$. 
            }
            \label{fig:hexample}
        \end{figure}

        \begin{theorem}\label{th:48_7}
            There exists a class $\mathcal{H}$ of simple separable connected $4$-regular graphs such that for each $H_n \in \mathcal{H}$ with $n$ vertices:
            \begin{itemize}
                \item[(i)] $n\equiv 4 \pmod{7}$ and $n\geq 11$,
                \item[(ii)] $|\mathcal{E}(H_n)| = 2^{-2} \cdot 48^{(n+3)/7}$.
            \end{itemize}
        \end{theorem}
        
        \begin{proof}
            Let $H_n$ be a simple separable graph containing only boundary $2$-ARSs.
            Suppose the regular representative for each $2$-ARS is the complete graph $K_5$.
            Let $H_{11}= \odot\{K^1_5, K^2_5\}$, with $K^1_5, K^2_5$ being isomorphic to $K_5$.
            This means that the number of vertices of $H_{11}$ is $11$ (5 for each ARS and 1 cut vertex), so
        
            $$ |\mathcal{E}(H_{11})| =24^2 = 48^{2}\cdot 2^{-2} = 48^{(11+3)/7} \cdot 2^{-2}. $$
        
            Now, let $\mathcal{W}$ be a family of connected graphs wherein every biconnected component is a cycle of length $3$.
            For any $\mathcal{T} \in \mathcal{W}$, let $V^2_{\mathcal{T}}$ and $V^4_{\mathcal{T}}$ denote the sets of degree-$2$ and degree-$4$ vertices in $\mathcal{T}$, respectively.
            Let also $\{K^1_5, \dots, K^{|V^2_{\mathcal{T}}|}_5 \}$ be a set of $|V^2_{\mathcal{T}}|$ graphs, each isomorphic to $K_5$, and let $E_{K_5}$ be a set containing exactly one edge selected from each $K^i_5$.
            Then $H_n=\underset{(E_{K_5},\mathcal{T})}{\otimes} \{K^1_5, \dots K^{|V^2_{\mathcal{T}}|}_5 \}$ is a graph with $n$ vertices constructed via cycle gluing.
            
            By Lemma~\ref{lem:count_c}, we know that $|V^4_{\mathcal{T}}| = |V^2_{\mathcal{T}}| - 3$.
            Additionally, let $V_b$ denote the set of vertices belonging strictly to the $2$-ARSs.
            Since each degree-$2$ vertex in $\mathcal{T}$ is connected to a uniquely corresponding $2$-ARS (which retains $5$ vertices from $K_5$), we have $|V_b| = 5\cdot |V^2_{\mathcal{T}}|$.
            
            Thus, the total number of vertices $n$ is:
        
            $$ n = |V_b| + |V^2_{\mathcal{T}}| + |V^4_{\mathcal{T}}| = 5|V^2_{\mathcal{T}}| + |V^2_{\mathcal{T}}| + (|V^2_{\mathcal{T}}| - 3) = 7\cdot |V^2_{\mathcal{T}}| - 3. $$
            Solving for $|V^2_{\mathcal{T}}|$ yields $|V^2_{\mathcal{T}}| = (n+3)/7$.
            
            Applying Theorem~\ref{th:formula} and recalling from Lemma~\ref{lem:count_c}(i) that the number of cycles $s = |V^4_{\mathcal{T}}| + 1$, we obtain:
        
            \begin{align*}
            	|\mathcal{E}(H_n)| &= 24^{|V^2_{\mathcal{T}}|} \cdot 2^{|V^4_{\mathcal{T}}|+1}\\
               &= 24^{|V^2_{\mathcal{T}}|} \cdot 2^{|V^2_{\mathcal{T}}|-2} \\
               & = 2^{-2} \cdot 48^{(n+3)/7}.
            \end{align*}
            concluding the proof for (ii).

            By the construction of this class of graphs it is clear that the number $(n+3)/7$ shall be a natural number bigger than $2$ (at least 2 copies of $K_5$), thus (i) follows. 
        \end{proof}

        Even if the computational results (see Table~\ref{tab:comps}), demonstrate that the maximum number of orientations is bigger for separable graphs for $n\leq 15$, notice that for a small number of vertices the number of Eulerian orientations for these graphs does not surpass the bound for simple biconnected graphs (see, for example, Figure~\ref{fig:hexample}).
        However, for $n \geq 116$, every graph in $\mathcal{H}$ strictly surpasses the biconnected bound established in Theorem~\ref{th:bi}.
        
        Based on Theorem~\ref{th:48_7} and several examples we have investigated (not only including the exhaustive computations for graphs with $n\leq 15$), we propose the following conjecture regarding the true upper bound for simple graphs.
        
        \begin{conjecture}\label{con:simple}
            Let $G=(V,E)$ be a simple separable connected $4$-regular graph with $n$ vertices.
            Then, 
            $$ |\mathcal{E}(G)| \leq 2^{-2}\cdot 48^{(n+3)/7}. $$
            This bound behaves asymptotically as $\mathcal{O}(1.7386^n)$.
        \end{conjecture}
        
    \section{Graph construction operations and their effect on the number of orientations}\label{sec:constr}

        In this section, we introduce the established graph operations that construct $4$-regular graphs and multigraphs from smaller ones, and we investigate their impact on the number of Eulerian orientations in cases where it is possible to derive a closed formula.

    \subsection{Construction operations for simple connected \texorpdfstring{$4$}{4}-regular graphs}

        The generation of the class of simple $4$-regular graphs is built upon $K_5$. 
        Any graph within this class can be constructed starting from $K_5$ through a finite sequence of three specific operations (see Figure~\ref{fig:sconstr}):

        \begin{enumerate}
            \item[$\mathcal{S}_1$:] 
            Let $(u_1, u_2)$ and $(u_3, u_4)$ be two  edges with no common endpoints in the graph. 
            The operation consists of removing 2 edges $(u_1, u_2)$ and $(u_3, u_4)$ and consequently adding a new vertex $v$ connecting it to the four endpoints of the removed edges by adding the edges $(u_1,v), (u_2,v), (u_3,v)$ and $(u_4,v)$. 
            This operation increases the number of the vertices of the graph by $1$.
        
            \item[$\mathcal{S}_2$:]  
            Let $v$ be a vertex in the graph with neighbors $u_1, u_2, u_3,$ and $u_4$.
            This operation replaces the vertex $v$ with four new vertices $v_1, v_2, v_3, v_4$ such that the induced subgraph $G[\{v_1,v_2,v_3,v_4\}]$ is isomorphic to $K_4$.
            These vertices are connected with the former neighbors of $v$ by adding the edges $(u_1, v_1),(u_2, v_2),(u_3, v_3)$ and $(u_4, v_4)$. 
            This operation increases the number of vertices of the graph by 3.
            
            \item[$\mathcal{S}_3$:]  
            Let $(u_1,u_2)$ be an edge in the graph. 
            This edge is removed and $5$ vertices $v_1, v_2, v_3, v_4, v_5$ are added that form a complete graph $K_5$ without the edge $(v_1, v_2)$.
            This subgraph is connected with the endpoints of the deleted edge by adding the edges $(u_1, v_1), (u_2, v_2)$.
            This operation increases the number of the vertices of the graph by 5.
        \end{enumerate}
        
        The sufficiency of these three operations to generate every $4$-regular graph from $K_5$ is established by the following theorem.
        
        \begin{theorem}[\cite{Bories1983}]
            All simple and  connected $4$-regular graphs are generated from $K_5$ by a finite number of the three operations $\mathcal{S}_1$, $\mathcal{S}_2$, and $\mathcal{S}_3$.
        \end{theorem}

        \begin{figure}
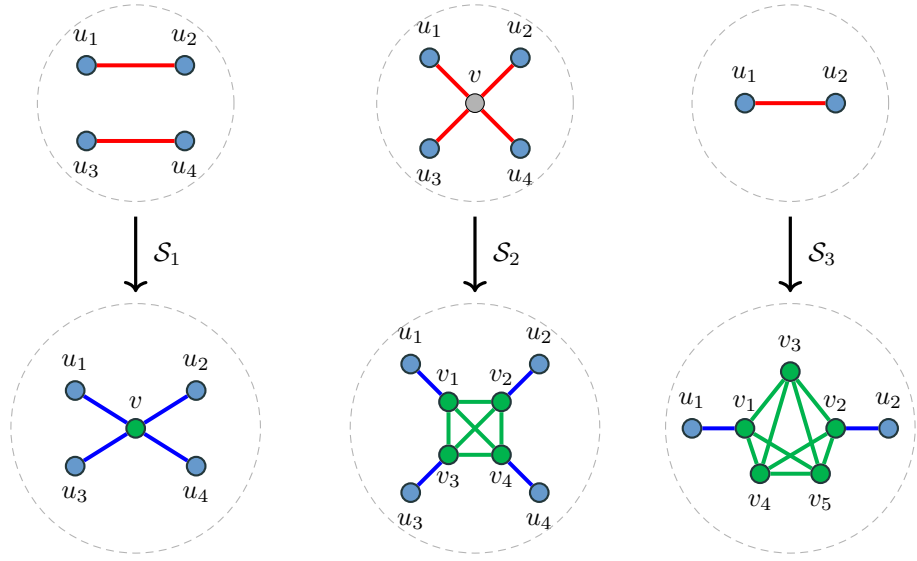
	
            \begin{center}
                \begin{tabular}{ccc}
			
                    \SI \hspace{2mm}	& \hspace{2mm} \SII  & \hspace{2mm} \SIII \\

                \end{tabular} 
            \end{center}
            \caption{The three construction operations for simple connected $4$-regular graphs \cite{Bories1983}.}
            \label{fig:sconstr}
        \end{figure}

        The question is how the Eulerian orientations of a graph $G^*$ relate to those of a graph $G$ obtained from $G^*$ by applying one of these operations once. 
        We provide closed formulas for operations $\mathcal{S}_2$ and $\mathcal{S}_3$.
        
        \begin{theorem}
            Let $G^*$ and $G$ be two simple $4$-regular graphs such that $G$ is generated by $G^*$ after applying operation $\mathcal{S}_2$ once. 
            Then $|\mathcal{E}(G)|=4|\mathcal{E}(G^*)|$.
        \end{theorem}

        \begin{proof}
            Let $v \in V(G^*)$ be the replaced vertex, and let $\{v_1, v_2, v_3, v_4\}$ form the $K_4$ graph in $G$ such that each $v_i$ connects to a distinct neighbor $u_i$ of $v$. 
            
            In any Eulerian orientation of $G$, the sum of the outdegrees of the vertices $v_1, v_2, v_3, v_4$ must be exactly $4 \cdot 2 = 8$. 
            Since the edges of the $K_4$ contribute exactly $6$ to this sum, exactly two of the edges $(u_i, v_i)$ must be directed towards the vertices of $K_4$, and two must be directed outwards them. 
            Therefore, contracting the $K_4$ back into the single vertex $v$ naturally yields a valid Eulerian orientation of $G^*$, where the edges belonging to both $G^*$ and $G$ preserve their orientation, and an edge $(u_i,v)$ in $G^*$ is directed towards (outwards) $v$ if and only if the corresponding edge $(u_i,v_i)$ in $G$ is directed towards (outwards) $v_i$.
            This means that every Eulerian orientation in $\mathcal{E}(G)$ corresponds to an Eulerian orientation in $\mathcal{E}(G^*)$.
            
            Conversely, consider any Eulerian orientation in $\mathcal{E}(G^*)$. 
            The vertex $v$ has exactly two incoming and two outgoing edges. 
            When extending this orientation to $G$, the edges $(u_i, v_i)$ directly inherit these directions. 
            To make the respective orientation of $G$ Eulerian, the edges of the $K_4$ must be oriented such that the two vertices of $K_4$ where the respective edges $(u_i,v_i)$ have been oriented towards (outwards) $v_i$ have two outgoing (incoming) edges from the edges of $K_4$.
            It is easily verified (see Figure~\ref{fig:kfour}) that a $K_4$ has exactly $4$ valid orientations satisfying this specific outdegree constraint (two vertices with an outdegree of $2$ and two vertices with an outdegree of $1$). 
            Consequently, every orientation in $\mathcal{E}(G^*)$ extends to exactly $4$ distinct Eulerian orientations in $\mathcal{E}(G)$.
            
            Combining both directions implies that there is a $4$-to-$1$ correspondence between $\mathcal{E}(G)$ and $\mathcal{E}(G^*)$.
            Thus, we conclude that $|\mathcal{E}(G)| = 4|\mathcal{E}(G^*)|$.
        
        \end{proof}

        \begin{figure}[htp!]
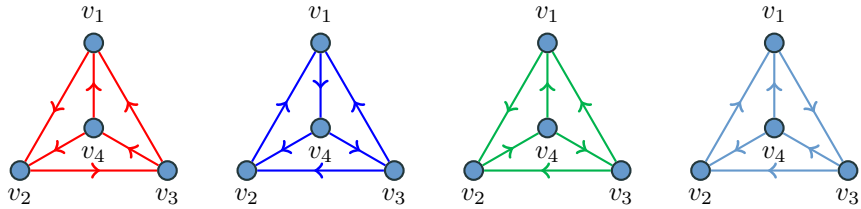

            \begin{center}
                \Kfourorient
                \caption{The four valid orientations of the $K_4$ subgraph, such that vertices $v_1$ and $v_2$ attain an outdegree of $1$, while vertices $v_3$ and $v_4$ attain an outdegree of $2$.}
            \label{fig:kfour}
            \end{center}
        \end{figure}

        \begin{corollary}
            The number of orientations for the class of simple $4$-regular graphs constructed by only $\mathcal{S}_2$ moves starting from $K_5$ is exactly
            $$ 24 \cdot 4^{(n-5)/3} = 6 \cdot 4^{(n-2)/3}. $$
            More generally, given a simple $4$-regular graph $G^*$ with $m$ vertices, any graph $G$ with $n$ vertices constructed solely via operation $\mathcal{S}_2$ starting from $G^*$ has 
            $$ |\mathcal{E}(G)| = |\mathcal{E}(G^*)| \cdot 4^{(n-m)/3} $$
            Eulerian orientations. 
            This bound grows asymptotically as $\mathcal{O}\left(1.5875^n\right)$.
        \end{corollary}
        
        Before deriving a closed formula for operation $\mathcal{S}_3$, we must establish the number of Eulerian orientations given the fixed orientation of a single edge. 
        Given a $4$-regular graph or multigraph $G$ and a subset of its edges $E^\prime$, we denote by $\mathcal{E}(G \mid \overrightarrow{E^\prime})$ the set of Eulerian orientations of $G$ subject to a fixed orientation $\overrightarrow{E^\prime}$ of these edges.

        \begin{lemma}\label{lemma-oriented-edge}
            For every $4$-regular graph or multigraph $G$ and an edge $e$ $$|\mathcal{E}(G\mid \overrightarrow{\{e\}})|=\frac{|\mathcal{E}(G)|}{2}$$
        \end{lemma}
        
        \begin{proof}
            Let $e=(u,v)$. 
            There is a $1$-to-$1$ correspondence between the orientations in which $e$ is directed from $u$ to $v$ and the orientations in which $e$ is directed from $v$ to $u$. 
            This correspondence is naturally obtained by reversing the direction of all edges in any valid Eulerian orientation of $G$. 
            Since these two sets form a perfect partition of $\mathcal{E}(G)$, it implies that $|\mathcal{E}(G \mid \overrightarrow{\{e\}})| = |\mathcal{E}(G)| / 2$.
        \end{proof}

        Let us now prove the theorem for the third operation.
        \begin{theorem}
            Let $G$ and $G^*$ be two simple $4$-regular graphs such that $G$ is generated by $G^*$ after applying  operation $\mathcal{S}_3$ once. Then $|\mathcal{E}(G)|=12|\mathcal{E}(G^*)|$.
        \end{theorem}
        
        \begin{proof}
            Let $(u_1,u_2) \in E(G^*)$ be the removed edge, and let $\{v_1, v_2, v_3, v_4,v
            _5\}$ form the $K_5\setminus\{(v_1,v_2)\}$ subgraph in $G$. 
            The edges $(u_1,v_1)$ and $(u_2,v_2)$ are also added to obtain $G$.
            
            Notice that in any Eulerian orientation of $G$, the sum of the outdegrees of the vertices $v_1, v_2, v_3, v_4,v_5$  must be exactly $5 \cdot 2 = 10$. 
            Since the edges of the $K_5 \setminus \{(v_1,v_2)\}$ subgraph contribute exactly $9$ to this sum, exactly one of the edges $(u_1, v_1)$ and $(u_2, v_2)$ must be directed towards its respective vertex $v_i$, and the other must be directed outwards it.
            Therefore, contracting the vertices of $K_5 \setminus \{(v_1,v_2)\}$ restores the edge $(u_1,u_2)$ and naturally yields a valid Eulerian orientation of $G^*$, where the edges belonging to both $G^*$ and $G$ preserve their orientation, and the edge $(u_1,u_2)$ in $G^*$ is directed towards (outwards) $u_1$ if and only if $(u_1,v_1)$ in $G$ is directed towards (outwards) $u_1$.
            This means that every Eulerian orientation in $\mathcal{E}(G)$ uniquely corresponds to an Eulerian orientation in $\mathcal{E}(G^*)$.
            
            Conversely, consider any Eulerian orientation in $\mathcal{E}(G^*)$.
            When extending this orientation to $G$, the direction of $(u_1, v_1)$ and $(u_2,v_2)$ is obtained by the orientation of the edge $(u_1,u_2)$. 
            To make the respective orientation of $G$ Eulerian, the edges of the $K_5 \setminus \{(v_1,v_2)\}$ subgraph must be oriented such that the single vertex (either $v_1$ or $v_2$) where the  edge $(u_i,v_i)$  is directed outwards has one outgoing edge from the edges of $K_5 \setminus \{(v_1,v_2)\}$. The other $4$ vertices strictly require to have $2$ outgoing edges from the edges of $K_5 \setminus \{(v_1,v_2)\}$.
            This number is equal to $|\mathcal{E}(K_5 \mid \overrightarrow{\{(v_1,v_2)\}})|$ which is $12$ implied by combining Lemma~\ref{lemma-oriented-edge} and the fact that $|\mathcal{E}(K_5)|=24$ (see Appendix~\ref{app:simple_bi}). 
            Consequently, every orientation in $\mathcal{E}(G^*)$ extends to exactly $12$ distinct orientations in $\mathcal{E}(G)$.
            
            Combining both directions implies that there is a $12$-to-$1$ correspondence between $\mathcal{E}(G)$ and $\mathcal{E}(G^*)$.
            Thus, we conclude that $|\mathcal{E}(G)| = 12|\mathcal{E}(G^*)|$.
        
        \end{proof}

        \begin{corollary}
            The number of Eulerian orientations for the class of simple $4$-regular graphs constructed solely via operation $\mathcal{S}_3$ starting from $K_5$ is exactly
            $$24\cdot 12^{(n-5)/5} = 2 \cdot 12^{n/5}.$$
            More generally, given a simple $4$-regular graph $G^*$ with $m$ vertices, any graph $G$ with $n$ vertices constructed solely via operation $\mathcal{S}_3$ starting from $G^*$ has 
            $$|\mathcal{E}(G)| = |\mathcal{E}(G^*)| \cdot 12^{(n-m)/5}$$
            Eulerian orientations.
            This bound grows asymptotically as $\mathcal{O} \left(1.6438^n \right)$.
        \end{corollary}
        
        Regarding the operation $\mathcal{S}_1$, a closed formula could not be derived directly in a similar manner as for $\mathcal{S}_2$ and $\mathcal{S}_3$. 
        
        Let $G$ be obtained from $G^*$ by applying operation $\mathcal{S}_1$ exactly once. 
        There exist Eulerian orientations of $G$, that do not correspond to an Eulerian orientations of $G^*$, in the sense that their common edges have the same direction in both Eulerian orientations.
        Specifically, consider an Eulerian orientation of $G$ in which both $(u_1,v)$ and $(u_2,v)$ are directed towards $v$, and consequently, $(u_3,v)$ and $(u_4,v)$ are directed outwards $v$ (towards $u_3$ and $u_4$, respectively). 
        Suppose that there exists an orientation of $G^*$, such that the shared edges with $G$ have the exact same orientations.
        Because $(u_1,v)$ and $(u_2,v)$ act as outgoing edges for $u_1$ and $u_2$ in $G$, these vertices must already possess an indegree of $2$ and an outdegree of $1$ strictly from their shared edges to satisfy the Eulerian condition. 
        Thus, in $G^*$, both $u_1$ and $u_2$ already possess an indegree of $2$ from these common edges. 
        This means that orienting the edge $(u_1, u_2)$ in either direction increases the indegree for one vertex, thus the orientation is not Eulerian (similarly orienting the edge $(u_3, u_4)$ increases the outdegree for one of the vertices).
        Therefore, no such direct orientation correspondence exists, unlike in the cases of operations $\mathcal{S}_2$ and $\mathcal{S}_3$.

    \subsection{Construction operations for connected loopless \texorpdfstring{$4$}{4}-regular multigraphs}
    
        Recall that the $4$-bond multigraph is the graph that consists of 2 vertices and 4 parallel edges between them.
        The generation of all connected loopless $4$-regular multigraphs from $4$-bond multigraph relies on the following two construction extensions (see Figure~\ref{fig:mconstr}).
        
        \begin{enumerate}
            \item[$\mathcal{M}_1$:] Let $(u_1, u_2)$ and $(u_3, u_4)$ be two edges of the graph (which may share an endpoint) . 
            The operation consists of removing the two edges $(u_1, u_2)$ and $(u_3, u_4)$ and introducing a new vertex $v$, which is then connected to the endpoints of the removed edges by adding the edges $(u_1,v), (u_2,v), (u_3,v)$, and $(u_4,v)$.
            If a vertex $u_i$ happens to be an endpoint for both removed edges, $v$ and $u_i$ will naturally be connected by two parallel edges.
            This operation increases the number of vertices in the graph by $1$.
            
            \item[$\mathcal{M}_2$:] Let $(u_1, u_2)$ be an edge of the graph. 
            The operation consists of removing the edge $(u_1, u_2)$ and introducing two new vertices, $v_1$ and $v_2$, connecting them to each other using $3$ parallel edges, and adding the edges $(u_1, v_1)$ and $(u_2, v_2)$. 
            This operation increases the number of vertices in the graph by $2$.
        \end{enumerate}
        
            These two operations are completely sufficient to generate the entire class of connected $4$-regular multigraphs starting from the $4$-bond graph. 
            This follows from a more general result established in \cite{DING2003329} for all connected regular multigraphs with an even degree. 
            The following theorem specializes this result to $4$-regular multigraphs, which are the primary focus of this paper.
        
        \begin{theorem}[\cite{DING2003329}]
            Every connected $4$-regular multigraph can be generated from the $4$-bond graph by a finite sequence of the two operations $\mathcal{M}_1$ and $\mathcal{M}_2$.
        \end{theorem}
    
        \begin{figure}
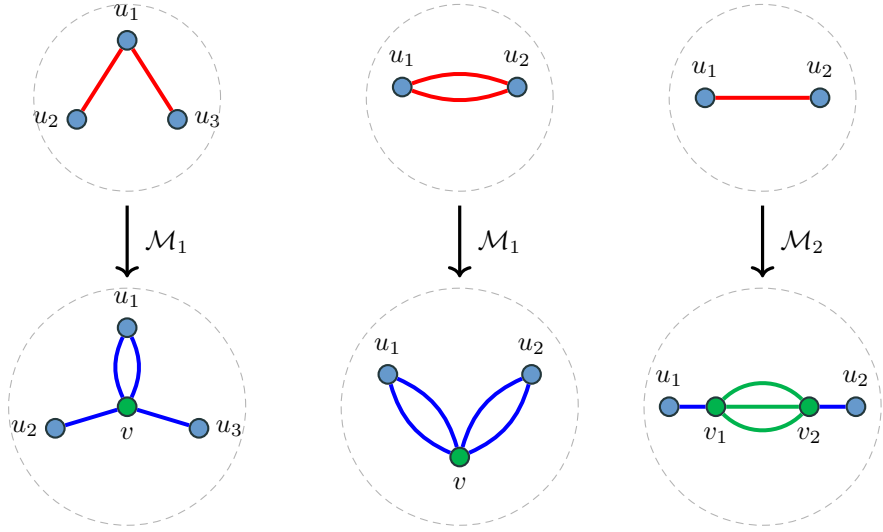
	
        	\begin{center}
        		\begin{tabular}{ccc}
        			
        			\MIone \hspace{2mm}	& \hspace{2mm} \MItwo  & \hspace{2mm} \MII \\
        	
        		\end{tabular} 
        	\end{center}
        	\caption{
            Construction operations for connected loopless $4$-regular multigraphs~\cite{DING2003329}. 
            Regarding operation $\mathcal{M}_1$, the figure illustrates the cases where the two removed edges share either one or both endpoints. 
            If the edges do not share any common endpoint, the operation is structurally identical to the operation $\mathcal{S}_1$ for simple graphs. }
        	\label{fig:mconstr}
        \end{figure}
    
        As with simple graphs, we establish a closed formula for operation $\mathcal{M}_2$.
        
        \begin{theorem}
            Let $G$ and $G^*$ be two $4$-regular multigraphs such that $G$ is obtained from $G^*$ by applying operation $\mathcal{M}_2$ exactly once. 
            Then $|\mathcal{E}(G)| = 3|\mathcal{E}(G^*)|$.
        \end{theorem}
        
        \begin{proof}
            Let $(u_1,u_2)$ be the removed edge of $G^*$, and let $v_1$ and $v_2$ be the added vertices in $G$ under operation $\mathcal{M}_2$. 
            Let also $(u_1,v_1)$ and $(u_2,v_2)$ be the added edges, and let $e_1, e_2$, and $e_3$ be the three parallel edges connecting $v_1$ and $v_2$ in $G$. 
            
            In any Eulerian orientation of $G$, the sum of the outdegrees of the vertices $v_1, v_2$  must be exactly $2 \cdot 2 = 4$. 
            The edges $e_1, e_2, e_3$ contribute exactly $3$ to this sum. 
            At least one of these edges must be directed towards $v_1$, and at least one must be directed towards $v_2$; otherwise, either $v_1$ or $v_2$ would be forced to have outdegree of $3$.
            This implies that in every Eulerian orientation of $G$ exactly one of the following two conditions is true:
         \begin{enumerate}
             \item $(u_1,v_1)$ is oriented towards $v_1$.
             \item $(u_2,v_2)$ is oriented towards $v_2$.
        \end{enumerate}
            Therefore, contracting the vertices $v_1$ and $v_2$ restores the edge $(u_1,u_2)$ and naturally yields a valid Eulerian orientation of $G^*$, where the edges belonging to both $G^*$ and $G$ preserve their orientation, and the edge $(u_1,u_2)$ in $G^*$ is directed towards (outwards) $u_1$ if and only if $(u_1,v_1)$ in $G$ is directed towards (outwards) $u_1$.
            This means that every Eulerian orientation in $\mathcal{E}(G)$ corresponds to an Eulerian orientation in $\mathcal{E}(G^*)$.
            
            Conversely, consider any Eulerian orientation in $\mathcal{E}(G^*)$.  
            When extending this orientation to $G$, the direction of $(u_1, v_1)$ and $(u_2,v_2)$ is strictly determined by the orientation of the edge $(u_1,u_2)$. 
            To make the respective orientation of $G$ Eulerian, the  edges $e_1, e_2, e_3$ must be oriented such that the single vertex (either $v_1$ or $v_2$) where the respective edge  $(u_i,v_i)$ is directed outwards (towards) it has exactly one (two) of the edges $e_1, e_2, e_3$ oriented outwards. 
            Since there are $\binom{3}{1}$ ways to orient one of the edges outwards the vertex with outdegree~$1$, every orientation in $\mathcal{E}(G^*)$ extends to exactly $3$ distinct orientations in $\mathcal{E}(G)$.
            
            Combining both directions implies that there is a $3$-to-$1$ correspondence between $\mathcal{E}(G)$ and $\mathcal{E}(G^*)$. 
            Thus, we conclude that $|\mathcal{E}(G)| = 3|\mathcal{E}(G^*)|$.
        \end{proof}

        \begin{corollary}
            The number of Eulerian orientations for the class of $4$-regular multigraphs constructed solely via operation $\mathcal{M}_2$ starting from the $4$-bond multigraph is exactly
            $$ 6\cdot 3^{(n-2)/2} = 2 \cdot 3^{n/2}. $$
            More generally, given a $4$-regular graph $G^*$ with $m$ vertices, any multigraph $G$ in $n$ vertices constructed solely via operation $\mathcal{M}_2$ starting from $G^*$ has 
            $$ |\mathcal{E}(G)| = |\mathcal{E}(G^*)| \cdot 3^{(n-m)/2} $$
            Eulerian orientations.
            This bound grows asymptotically as $\mathcal{O} \left(1.7321^n \right)$.
        \end{corollary}
        
        A closed formula could not be derived directly in a similar manner for $\mathcal{M}_1$, for the same reason as in the simple graph case and the operation $\mathcal{S}_1$.

    \section{Conclusion and Future Work.}\label{sec:concl}

        In this work, we have established improved upper bounds for the number of Eulerian orientations of connected loopless $4$-regular graphs across three distinct categories: biconnected multigraphs, biconnected simple graphs, and separable connected simple graphs.
        Furthermore, we demonstrated that the biconnected multigraph bound is sharp and realizable.
        Additionally, in the case of simple graphs, we executed exhaustive computations for all graphs up to $15$ vertices.
        We also presented a formula that computes the exact number of orientations of separable connected graphs given the structure of the graph and the number of orientations of certain subgraphs, leading to a divide-and-conquer algorithm.
        Using this method, we also presented a class of simple separable connected graphs that strictly surpasses the upper bound we derived for simple biconnected graphs, conjecturing that this class attains the global upper bound for simple separable graphs.
        Finally, we provided closed formulas for the effect of specific construction operations on the number of Eulerian orientations.

        Regarding future work, the highest priority is to address and mathematically resolve Conjecture~\ref{con:simple}.
        For this purpose, we may consider a technique exploiting the $2$-ARS structure; the delicate aspect here is that an inner $2$-ARS within a simple graph may have a multigraph as its regular representative (see Figure~\ref{fig:ARSexamples}), so one should simultaneously utilize the bounds for both biconnected simple graphs and multigraphs.
        We also aim to lower the established upper bound for biconnected simple graphs, since our computations indicate that it is rather loose, at least for graphs with a small number of vertices.
        
        Finally, another direction is to explore $(\kappa, \lambda)$-tight graphs, i.e., graphs $G=(V,E)$ such that $|E|=\kappa \cdot |V|-\lambda$ and $|E^\prime| \leq \kappa \cdot |V^\prime|-\lambda$ for every induced subgraph $G^\prime=(V^\prime,E^\prime) \subseteq G$ with $|V^\prime| \geq \kappa$, and seek bounds for outdegree-constrained orientations with a maximum outdegree of $\kappa$.
        These graphs apply to rigidity theory and generally to geometric constraint problems.

	\bibliographystyle{plain}
    \bibliography{orientations}

@article{Asahiro2022,
title = {Upper and lower degree-constrained graph orientation with minimum penalty},
journal = {Theoretical Computer Science},
volume = {900},
pages = {53-78},
year = {2022},
issn = {0304-3975},
doi = {https://doi.org/10.1016/j.tcs.2021.11.019},
url = {https://www.sciencedirect.com/science/article/pii/S0304397521006964},
author = {Y. Asahiro and J. Jansson and E. Miyano and H. Ono}
}

@inproceedings{betISSAC,
author = {Bartzos, E. and Emiris, I. Z. and Kotsireas, I. S. and Tzamos, C.},
title = {Bounding the Number of Roots of Multi-Homogeneous Systems},
year = {2022},
isbn = {9781450386883},
publisher = {Association for Computing Machinery},
address = {New York, NY, USA},
url = {https://doi.org/10.1145/3476446.3536189},
doi = {10.1145/3476446.3536189},
booktitle = {Proceedings of the 2022 International Symposium on Symbolic and Algebraic Computation},
pages = {255–262}}

@article{bet,
title = {An asymptotic upper bound for graph embeddings},
journal = {Discrete Applied Mathematics},
volume = {327},
pages = {157-177},
year = {2023},
issn = {0166-218X},
doi = {https://doi.org/10.1016/j.dam.2022.12.010},
url = {https://www.sciencedirect.com/science/article/pii/S0166218X22004607},
author = {E. Bartzos and I. Z. Emiris and C. Tzamos},
}

@Article{bev,
author= {Bartzos, E. and Emiris, I.Z. and Vidunas, R.},
journal={Discrete \& Computational Geometry},
title = {New Upper Bounds for the Number of Embeddings of Minimally Rigid Graphs},
year = 2022, 
url={https://doi.org/10.1007/s00454-022-00370-3}
}

@misc{code_mBezout,
  author       = {Bartzos, E. and Schicho, J.},
  title        = {Source code and examples for the paper ``{On} the multihomogeneous {B\'ezout} bound on the number of embeddings of minimally rigid graphs''},
  month        = Nov,
  year         = {2019},
  publisher    = {Zenodo},
  note          = {https://doi.org/10.5281/zenodo.3542061}
}

@misc{dataforpaper,
  author       = {Bartzos, E. and Samaris, M.},
  title        = {Code, results and examples to accompany the paper "Computing and Bounding the Number of Eulerian Orientations for Certain Classes of 4-Regular Graphs"},
  year         = {2026},
  publisher    = {Zenodo},
  note          = {https://doi.org/10.5281/zenodo.20842790}
}

@incollection{bories1983,
title = {Construction of 4-Regular Graphs},
editor = {C. Berge and D. Bresson and P. Camion and J.F. Maurras and F. Sterboul},
series = {North-Holland Mathematics Studies},
publisher = {North-Holland},
volume = {75},
pages = {99-118},
year = {1983},
booktitle = {Combinatorial Mathematics},
issn = {0304-0208},
doi = {https://doi.org/10.1016/S0304-0208(08)73378-4},
url = {https://www.sciencedirect.com/science/article/pii/S0304020808733784},
author = {F. Bories and J.-L. Jolivet and J.-L. Fouquet}}

@article{DING2003329,
title = {Generating r-regular graphs},
journal = {Discrete Applied Mathematics},
volume = {129},
number = {2},
pages = {329-343},
year = {2003},
issn = {0166-218X},
doi = {https://doi.org/10.1016/S0166-218X(02)00593-0},
url = {https://www.sciencedirect.com/science/article/pii/S0166218X02005930},
author = {Guoli Ding and Peter Chen}
}

@Article{Felsner,
author="Felsner, S. and Zickfeld, F.",
title="On the Number of Planar Orientations with Prescribed Degrees",
journal="The Electronic Journal of Combinatorics",
year="2008",
month="Jun",
volume="15",
number="01",
pages="Research paper R77",
url="https://doi.org/10.37236/801 "
}

@inproceedings{Shortest_Paths,
author = {Frigioni, D. and Marchetti-Spaccamela, A. and Nanni, U.},
title = {Fully Dynamic Shortest Paths and Negative Cycles Detection on Digraphs with Arbitrary Arc Weights},
year = {1998},
publisher = {Springer-Verlag},
pages = {320–331},
numpages = {12},
booktitle = {Proceedings of the 6th Annual European Symposium on Algorithms},
series = {ESA '98}
}

@book{Harary69,
  Author         = {Harary, F.},
  Title          = {Graph Theory},
  Publisher      = {Addison-Wesley},
  year           = 1969
}

@InProceedings{matching,
author="He, M.
and Tang, G.
and Zeh, N.",
editor="Ahn, H.-K.
and Shin, C.-S.",
title="Orienting Dynamic Graphs, with Applications to Maximal Matchings and Adjacency Queries",
booktitle="Algorithms and Computation",
year="2014",
publisher="Springer International Publishing",
address="Cham",
pages="128--140",
isbn="978-3-319-13075-0"
}

@incollection{LasVergnas,
title = {Le Polynôme De Martin D'un Graphe Eulerien},
editor = {C. Berge and D. Bresson and P. Camion and J.F. Maurras and F. Sterboul},
series = {North-Holland Mathematics Studies},
publisher = {North-Holland},
volume = {75},
pages = {397-411},
year = {1983},
booktitle = {Combinatorial Mathematics},
issn = {0304-0208},
doi = {https://doi.org/10.1016/S0304-0208(08)73415-7},
url = {https://www.sciencedirect.com/science/article/pii/S0304020808734157},
author = {M. {Las Vergnas}}
}

@article{LasVergnas2,
title = {An upper bound for the number of Eulerian orientations of a regular graph},
journal = {Combinatorica},
volume = {10},
issue={1},
pages = {61-65},
year = {1990},
issn = {1439-6912},
doi = {https://doi.org/10.1007/BF02122696},
author = {M. {Las Vergnas}},
}

@article{nauty,
title = {Practical graph isomorphism, II},
journal = {Journal of Symbolic Computation},
volume = {60},
pages = {94-112},
year = {2014},
issn = {0747-7171},
doi = {https://doi.org/10.1016/j.jsc.2013.09.003},
url = {https://www.sciencedirect.com/science/article/pii/S0747717113001193},
author = {B. D. McKay and A. Piperno}
}

@article{MihWin,
title = {On the number of Eulerian orientations of a graph},
journal = {Algorithmica},
volume = {16},
issue={4},
pages = {402-414},
year = {1996},
issn = {1432-0541},
doi = {https://doi.org/10.1007/BF01940872},
author = {Mihail, M. and Winkler, P.},
}

@article{Punzi,
title = {Refined Bounds on the Number of Eulerian Tours in Undirected Graphs},
journal = {Algorithmica},
volume = {86},
issue={1},
pages = {194-217},
year = {2024},
issn = {1432-0541},
doi = {https://doi.org/10.1007/s00453-023-01162-8},
author = {Punzi, G. and Conte, A. and Rizzi, R.},
}

@Article{Schrijver1983,
author="Schrijver, A.",
title="Bounds on the number of Eulerian orientations",
journal="Combinatorica",
year="1983",
month="Sep",
volume="3",
number="3",
pages="375--380",
url="https://doi.org/10.1007/BF02579193"
}

@article{graph_coloring,
author = {Solomon, S. and Wein, N.},
title = {Improved Dynamic Graph Coloring},
year = {2020},
issue_date = {July 2020},
publisher = {Association for Computing Machinery},
volume = {16},
number = {3},
issn = {1549-6325},
url = {https://doi.org/10.1145/3392724},
doi = {10.1145/3392724},
journal = {ACM Trans. Algorithms},
month = jun,
articleno = {41}
}

@article{Stefa,
title = {The Complexity of Counting Eulerian Tours in $4$-regular Graphs},
journal = {Algorithmica},
volume = {63},
issue={3},
pages = {588-601},
year = {2024},
issn = {1432-0541},
doi = {https://doi.org/10.1007/s00453-010-9463-4},
author = {Ge, Q. and Štefankovič, Daniel},
}

@misc{welsh1990,
  title={The computational complexity of some classical problems from statistical physics},
  author={Welsh, D.J.A.},
  journal={Disorder in physical systems},
  volume={307},
  pages={307--321},
  year={1990},
  publisher={Clarendon Press, Oxford}
}
    
    \newpage
    
    \appendix

    \section*{Appendix}
    \section{Simple graphs with the maximum number of orientations}\label{app:simple}

        In this appendix, we provide illustrations of the simple $4$-regular graphs that maximize the number of Eulerian orientations for both the biconnected and separable cases (see Table~\ref{tab:comps}).

    \subsection{Simple biconnected graphs}\label{app:simple_bi}

        \begin{center}

            \begin{tabular}{cccc}
                $\bm{n=5}$ &  $\bm{n=6}$ & \multicolumn{2}{c}{$\bm{n=7}$ (2 graphs) } \vspace{2mm}\\
                \Kfive & \Hocta &
                \Gsevenfirst & \Gsevensecond  \\
                $K_5$ & $G_6$ & $G_{7a}$ & $G_{7b}$ \vspace{1.5mm}\\
                $|\mathcal{E}(K_5)| =24 $ &  $|\mathcal{E}(G_6)| =38 $ & \multicolumn{2}{c}{$|\mathcal{E}(G_{7a})|= |\mathcal{E}(G_{7b})|= 60$} \vspace{1.5mm}\\
            \end{tabular}
            \noindent\rule{\textwidth}{0.4pt}
            \bigskip

            \begin{tabular}{cccc}
                \multicolumn{3}{c}{$\bm{n=8}$ (3 graphs)} & $\bm{n=9}$ \vspace{2mm}\\
                \GeightA & \GeightB & \GeightC & \Gnine  \\
                $G_{8a}$ & $G_{8b}$ & $G_{8c}$ & $G_{9}$ \vspace{1.5mm}\\
                \multicolumn{3}{c}{$|\mathcal{E}(G_{8a})|= |\mathcal{E}(G_{8b})|= |\mathcal{E}(G_{8c})|= 96$} & $|\mathcal{E}(G_9)| =160 $ \vspace{1.5mm}\\
            \end{tabular}
            \noindent\rule{\textwidth}{0.4pt}
            \bigskip

            \begin{tabular}{ccc}
                \multicolumn{2}{c}{$\bm{n=10}$ (2 graphs)} \hspace{5mm} & \hspace{5mm} $\bm{n=11}$ \vspace{2mm}\\
                \GtenA & \GtenB   \hspace{5mm} &  \hspace{5mm} \Geleven  \\
                $G_{10a}$ &  $G_{10b}$ \hspace{5mm} & \hspace{5mm} $G_{11}$ \vspace{1.5mm}\\
                \multicolumn{2}{c}{$|\mathcal{E}(G_{10a})|= |\mathcal{E}(G_{10b})|= 288$}  \hspace{5mm} & \hspace{5mm} $|\mathcal{E}(G_{11})| =480 $ \vspace{1.5mm}\\
            \end{tabular}
            \noindent\rule{\textwidth}{0.4pt}
            \bigskip

            \begin{tabular}{cc}
                $\bm{n=12}$  \hspace{5mm} & \hspace{5mm} $\bm{n=13}$ \vspace{2mm}\\
                \Gtwelve    \hspace{5mm} &  \hspace{5mm} \Gthirteen  \\
                $G_{12}$  \hspace{5mm} & \hspace{5mm} $G_{13}$ \vspace{1.5mm}\\
                $|\mathcal{E}(G_{12})| =800 $  \hspace{5mm} & \hspace{5mm} $|\mathcal{E}(G_{13})| =1344 $ \vspace{1.5mm}\\
            \end{tabular}
            \noindent\rule{\textwidth}{0.4pt}
            \bigskip

            \begin{tabular}{cc}
                $\bm{n=14}$  \hspace{5mm} & \hspace{5mm} $\bm{n=15}$ \vspace{2mm}\\
                \Gfourteen    \hspace{5mm} &  \hspace{5mm} \Gfifteen  \\
                $G_{14}$  \hspace{5mm} & \hspace{5mm} $G_{15}$ \vspace{1.5mm}\\
                $|\mathcal{E}(G_{14})| =2304 $  \hspace{5mm} & \hspace{5mm} $|\mathcal{E}(G_{15})| =4224 $ \vspace{1.5mm}\\
            \end{tabular}
            \noindent\rule{\textwidth}{0.4pt}
            \bigskip
                    
        \end{center}

    \subsection{Simple separable graphs}\label{app:simple_sep}

        \begin{center}
            \begin{tabular}{ccc}
                $\bm{n=11}$  \hspace{2.5mm} & \hspace{2.5mm} $\bm{n=12}$ & \hspace{2.5mm} $\bm{n=13}$ \vspace{2mm}\\
                \Heleven    \hspace{2.5mm} &  \hspace{2.5mm} \Htwelve &  \hspace{2.5mm} \Hthirteen  \\
                $H_{11}$  \hspace{2.5mm} & \hspace{2.5mm} $H_{12}$ &  \hspace{2.5mm} $H_{13}$ \vspace{1.5mm} \\
                $|\mathcal{E}(H_{11})| =576 $  \hspace{2.5mm} & \hspace{2.5mm} $|\mathcal{E}(H_{12})| =960$ & \hspace{2.5mm} $|\mathcal{E}(H_{13})| =1600 $ \vspace{1.5mm}\\
            \end{tabular}
            \noindent\rule{\textwidth}{0.4pt}
            \bigskip

            \begin{tabular}{cc}
                $\bm{n=14}$  \hspace{5mm} & \hspace{5mm} $\bm{n=15}$ \vspace{2mm}\\
                \Hfourteen    \hspace{5mm} &  \hspace{5mm} \Hfifteen  \\
                $H_{14}$  \hspace{5mm} & \hspace{5mm} $H_{15}$ \vspace{1.5mm}\\
                $|\mathcal{E}(H_{14})| =2688 $  \hspace{5mm} & \hspace{5mm} $|\mathcal{E}(H_{15})| =4608 $ \vspace{1.5mm}\\
            \end{tabular}

        \end{center}

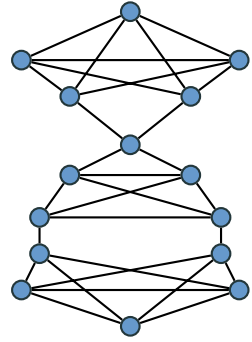
\end{document}